\documentclass[12pt]{article}
\usepackage[utf8]{inputenc}
\usepackage{mathtools}
\usepackage[normalem]{ulem}
\usepackage{amsfonts}
\usepackage{amsthm}
\usepackage{graphicx}
\usepackage{algorithmic}
\usepackage{algorithm}

\usepackage{enumitem}
\setlist[enumerate]{leftmargin=.5in}
\setlist[itemize]{leftmargin=.5in}

\usepackage{amsmath,accents}
\usepackage[margin=1in]{geometry}
\usepackage{thm-restate}
\usepackage{wrapfig}
\usepackage{amssymb}
\usepackage{url}
\usepackage[colorlinks]{hyperref}

\usepackage{color}
\usepackage{lineno}
\usepackage{bbm}
\usepackage{framed}

\newcommand{\R}{\mathbb{R}}

\newcommand{\disna}{\mathrm{dis}_{\infty}}
\newcommand{\spec}{\mathrm{spec}}

\newcommand{\dis}{\mathrm{dis}}

\newcommand{\disp}{\mathrm{dis}_{p}}

\newcommand{\ums}{\mathcal{U}}
\newcommand{\ms}{\mathcal{M}}

\newcommand{\ufin}{\mathcal{U}^\mathrm{fin}}
\newcommand{\p}{^{\prime}}
\newcommand{\dlp}{\dgh^{\scriptscriptstyle{(p)}}}

\newcommand{\dghp}[1]{\dgh^{\scriptscriptstyle{(#1)}}}

\newcommand{\dgh}{d_{\mathrm{GH}}}
\newcommand{\ugh}{u_{\mathrm{GH}}}

\newcommand{\eps}{\varepsilon}
\newcommand{\diam}{\mathrm{diam}}

\newcommand{\absult}{\mathrm{ult^{abs}}}

\newcommand{\lc}{\left(}
\newcommand{\rc}{\right)}
\newcommand{\lb}{\left\{}
\newcommand{\rb}{\right\}}

\newtheorem{theorem}{Theorem}
\newtheorem{corollary}[theorem]{Corollary}
\newtheorem{lemma}[theorem]{Lemma}
\newtheorem{definition}[theorem]{Definition}
\newtheorem{remark}[theorem]{Remark}
\newtheorem{proposition}[theorem]{Proposition}
\newtheorem{claim}[theorem]{Claim}
\newtheorem{example}[theorem]{Example}

\usepackage{nomencl}
\makenomenclature
\renewcommand{\nompreamble}{The list below describes several symbols that will be used in the article.}

\newcommand{\ct}[1]{_{\mathfrak{c}\left(#1\right)}}

\ifpdf
\hypersetup{
  pdftitle={The Gromov-Hausdorff distances between ultrametric spaces},
  pdfauthor={F. M\'emoli, Z. Smith, and Z. Wan}
}
\fi

\begin{document}
\title{The Gromov-Hausdorff distance between ultrametric spaces:\\
its structure and computation\thanks{This work was partially supported by the NSF through grants
CCF-1740761, and CCF-1526513, and DMS-1723003.}}

\author{Facundo M\'emoli\thanks{Department of Mathematics and Department of Computer Science and Engineering, The Ohio State University, Columbus, OH, USA 
  (\texttt{memoli@math.osu.edu}, \url{http://facundo-memoli.org/}).}
\and Zane Smith\thanks{Department of Computer Science and Engineering, University of Minnesota 
  (\texttt{smit9474@umn.edu}).}
\and Zhengchao Wan\thanks{Department of Mathematics, The Ohio State University, Columbus, OH, USA 
  (\texttt{wan.252@osu.edu}, \url{https://zhengchaow.github.io}).}}
  
\maketitle

\begin{abstract}
  The Gromov-Hausdorff distance ($d_\mathrm{GH}$) provides a natural way of quantifying the dissimilarity between two given metric spaces. It is known  that computing $d_\mathrm{GH}$ between two finite metric spaces is NP-hard, even in the case of finite ultrametric spaces which are highly structured metric spaces in the sense that they satisfy the so-called \emph{strong triangle inequality}. Ultrametric spaces naturally arise in many applications such as hierarchical clustering, phylogenetics, genomics, and even linguistics. By exploiting the special structures of ultrametric spaces, (1) we identify a one parameter family $\{\dghp p\}_{p\in[1,\infty]}$ of distances defined in a flavor similar to the Gromov-Hausdorff distance on the collection of finite ultrametric spaces, and in particular $\dghp 1 =\dgh$. The extreme case when $p=\infty$, which we also denote by $\ugh$, turns out to be an ultrametric on the collection of ultrametric spaces. Whereas for all $p\in[1,\infty)$, $\dghp p$ yields NP-hard problems, we prove that surprisingly $u_\mathrm{GH}$ can be computed in polynomial time. The proof is based on a structural theorem for $\ugh$ established in this paper; (2) inspired by the structural theorem for $u_\mathrm{GH}$, and by carefully leveraging  properties of ultrametric spaces, we also establish a structural theorem for $d_\mathrm{GH}$ when restricted to ultrametric spaces. This structural theorem allows us to identify special families of ultrametric spaces on which $\dgh$ is computationally tractable. These families are determined by properties related to the doubling constant of metric space. Based on these families, we devise a fixed-parameter tractable (FPT) algorithm for computing the exact value of $\dgh$ between ultrametric spaces. We believe ours is the first such algorithm to be identified.
\end{abstract}

\nomenclature{$B_{(\eps)}$}{The set of $\eps$-maximal unions of closed balls of the ball $B$; page \pageref{symbol:Beps}.}

\nomenclature{$\theta_X$}{The dendrogram corresponding to the ultrametric space $X$; page \pageref{def:dendrogram}.}

\nomenclature[12]{$\lc X_{\mathfrak{c}(t)},u_{X_{\mathfrak{c}(t)}}\rc$}{The $t$-closed quotient of the ultrametric space $X$; page \pageref{def:ultraquotient}.}

\nomenclature[13]{$ X_{\mathfrak{o}(t)}$}{The $t$-open partition of the ultrametric space $X$; page \pageref{def:ultraquotient-open}.}

\nomenclature{$V_X$}{The set of all closed balls in the ultrametric space $X$; page \pageref{sec:dp-main-text}.}

\nomenclature[05]{$\ums$}{The collection of all compact ultrametric spaces.}

\nomenclature[06]{$\ums^\mathrm{fin}$}{The collection of all finite ultrametric spaces.}

\nomenclature[07]{$(X,u_X)$}{An ultrametric space.}

\nomenclature[01]{$\dgh$}{The Gromov-Haussdorff distance; page \pageref{def:dgh}.}

\nomenclature[02]{$\ugh$}{The Gromov-Haussdorff ultrametric; page \pageref{eq:ugh}.}

\nomenclature{$\delta_\eps(X)$}{A convenient symbol representing $\delta_\eps(X)\coloneqq\diam(X)-\eps$; page \pageref{thm:ums-dgh}.}

\nomenclature{$\rho_\eps(X)$}{A convenient symbol representing $\rho_\eps(X)\coloneqq\max(\diam\lc X\rc -2\eps,0)$; page \pageref{rhos_eps}.}

\nomenclature[09]{$[x]_{\mathfrak{c}(t)}^X$}{The equivalence class of the closed equivalence relation; page \pageref{closed relation}.}

\nomenclature[10]{$[x]_{\mathfrak{o}(t)}^X$}{The equivalence class of the open equivalence relation; page \pageref{open relation}.}

\nomenclature{$\#X$}{The cardinality of set $X$.}

\nomenclature[03]{$\dghp{p}$}{$p$-Gromov-Hausdorff distance; page \pageref{eq:dgh-p-distortion}.}

\nomenclature[04]{$\dis_p$}{$p$-distortion; page \pageref{eq:dist}.}

\nomenclature[05]{$\Lambda_p$}{Absolute $p$-difference; page \pageref{sec:dghp}.}

\nomenclature[11]{$[\![x]\!]_{\mathfrak{c}(t)}^X$}{The equivalence class of the closed equivalence relation in the ultra dissimilarity space $X$; page \pageref{eq:equivalence class treegram}.}

\nomenclature{$\spec(X)$}{The spectrum of a metric space $X$: $\spec(X)\coloneqq\{d_X(x,x'):\,x,x'\in X\}$; page \pageref{thm:ugh-struct}.}

\section{Introduction and main results}
\label{sec:introduction}

Edwards \cite{edwards1975structure} and Gromov \cite{gromov1981groups} independently introduced a notion nowadays called the \emph{Gromov-Hausdorff distance} $\dgh$ for comparing metric spaces. This distance enjoys many pleasing mathematical properties: if we let $\ms$ denote the collection of all compact metric spaces, then modulo isometry,  $(\ms,\dgh)$ is a complete and separable metric space  {\cite[Proposition 43]{petersen2006riemannian}}, with rich pre-compact classes \cite{gromov2007metric}. It has also recently been proved that this space is geodesic \cite{ivanov2015gromov,chowdhury2018explicit}. This distance has been widely used in differential geometry \cite{petersen2006riemannian}, as a model for shape matching procedures \cite{memoli2007use, bronstein2010gromov}, and applied algebraic topology \cite{chazal2009gromov} for establishing stability properties of invariants.

Despite admitting many lower bounds which can be computed in polynomial time \cite{chazal2009gromov,memoli2012some}, computing $\dgh$ itself between arbitrary finite metric spaces leads to solving certain generalized quadratic assignment problems \cite{memoli2012some} which have been shown to be NP-hard  \cite{schmiedl2015shape,schmiedl2017computational,agarwal2018computing}. In fact, in \cite{schmiedl2017computational} Schmiedl proved the following stronger result (see however \cite{majhi2019approximating} for the case of point sets on the real line where the authors describe a poly time approximation algorithm):
\begin{theorem}[{\cite[Corollary 3.8]{schmiedl2017computational}}]\label{thm:NP-hard}
The Gromov-Hausdorff distance cannot be approximated within any factor less than 3
in polynomial time, unless $\mathcal{P}=\mathcal{NP}$.
\end{theorem}

The proof of this result reveals that the claim still holds even in the case of \emph{ultrametric spaces}.
An ultrametric space $(X,d_X)$ is a metric space which satisfies the \emph{strong triangle inequality}: 
\[\forall x,x',x''\in X, \text{ one has } d_X(x,x')\leq\max\lc d_X(x,x''),d_X(x'',x')\rc.\]
In this paper, we will henceforth use $u_X$ instead of $d_X$ to represent an ultrametric. Ultrametric spaces appear in many applications: they arise in statistics as a geometric encoding of \emph{dendrograms} \cite{jardine-sibson,carlsson2010characterization},  in taxonomy and phylogenetics \cite{phylo-book} as representations of phylogenies, and in linguistics \cite{roberts}.
In theoretical computer science, ultrametric spaces arise as building blocks for the probabilistic approximation of finite metric spaces \cite{bartal}.

In many of the aforementioned applications (including  phylogenetics), in order to characterize the difference between relevant objects, it is important to compare ultrametric spaces via meaningful metrics. This is one of the main motivations behind our study of the Gromov-Hausdorff distance between ultrametric spaces. 

Being a well understood and highly structured type of metric spaces, we are particularly interested in exploiting possible advantages associated to either \emph{restricting} or \emph{adapting} $\dgh$ to the collection $\ufin$ of all finite ultrametric spaces. In this paper, we provide positive  answers to the following two questions naturally arising from trying to bypass/overcome the hardness result in Theorem \ref{thm:NP-hard}:

\begin{framed}

\begin{enumerate}

\item[(Q1)] Is there any suitable variant of the Gromov-Hausdorff distance on the collection of finite ultrametric spaces which can be approximated/computed in polynomial time? 

\item[]

\item[(Q2)] Is there any subcollection of ultrametric spaces on which the Gromov-Hausdorff distance can be approximated/computed in polynomial time? 

\end{enumerate}
\end{framed}

In this paper we provide positive answers to these two questions and in the course of answering these questions, we establish structural theorems for both $\dgh$ and a suitable \emph{ultrametric} variant $\ugh$ which in each case allow us to convert the problem of comparing two given spaces into instances of the problem on strictly smaller spaces.

\paragraph{Related work}
The Gromov-Hausdorff ultrametric, which we denote by $\ugh$, on the collection $\ums$ of compact ultrametric spaces was first introduced by Zarichnyi \cite{zarichnyi2005gromov} in 2005 as an ultrametric counterpart to $\dgh$. Moreover, the author proved that $(\ums,\ugh)$ is a complete but not separable (ultra) metric space, where $\ums$ denotes the collection of all compact ultrametric spaces. $\ugh$ was further studied by Qiu in \cite{qiu2009geometry} where the author established several characterizations of $\ugh$  similar to the classical ones for $\dgh$  (cf. \cite[Chapter 7]{burago}) such as those arising via the notions of $\eps$-isometry and $(\eps,\delta)$-approximation. Qiu has also found a suitable version of Gromov’s pre-compactness theorem for $(\ums,\ugh)$.

Phylogenetic tree shapes (unlabled rooted trees) are closely related to ultrametric spaces. In \cite{colijn2018metric}, Colijn and Plazzotta studied a metric between tree shapes to compare evolutionary trees of influenza. In \cite{liebscher2018new}, Liebscher studied a class of metrics analogous to $\dgh$ between unrooted phylogenetic trees. In \cite{lafond2019complexity}, Lafond et al. extended different types of metrics on phylogenetic trees to metrics between tree shapes via optimization over permutations of labels. They studied the computational aspect of these metric extensions. In particular, they proved that computing the extension of the path distance is NP-complete via a similar argument used for proving that approximating $\dgh$ between merge trees in NP-complete \cite{agarwal2018computing}. Moreover, they devised an FPT algorithm which computes the extension of the so-called Robinson-Foulds distances. Their FPT algorithm is a recursive algorithm comparing subtrees of nodes at each iteration, which is of similar flavor to our algorithms (Algorithms \ref{algo-dGH-rec} and \ref{algo-dGH-dyn}) for computing $\dgh$ between ultrametric spaces.

In \cite{touli2018fpt}, Touli and Wang devised FPT algorithms for the computation of the interleaving distance $d_\mathrm{I}$  between merge trees \cite{morozov2013interleaving}. Since any finite ultrametric space can be naturally represented by a {merge tree} (see for example \cite{gasparovic2019intrinsic}) it turns out that $d_\mathrm{I}$ between ultrametric spaces as merge trees is a 2-approximation of $\dgh$ between the ultrametric spaces (see \cite[Corollary 6.13]{memoli2019gromov}). Thus, one could potentially adapt the algorithm  from \cite{touli2018fpt} for computing a 2-approximation for $\dgh$ between ultrametric spaces, which is FPT. In this paper we obtain essentially the same time complexity for the \emph{exact} computation of $\dgh$ (see Remark \ref{rem:tw}) via algorithms specifically tailored for ultrametric spaces.

\subsection{Our results}
In this section, we summarize our main results obtained in the course of answering the two major questions mentioned above.

\subsubsection{Polynomial time computable variant of $\dgh$}

Let $(X,d_X)$ and $(Y,d_Y)$ be two metric spaces. A \emph{correspondence} $R$ between the underlying sets $X$ and $Y$ is any subset of $X\times Y$ such that the images of $R$ under the canonical projections $p_X:X\times Y\rightarrow X$ and $p_Y:X\times Y\rightarrow Y$ are full: $p_X(R)=X$ and $p_Y(R)=Y$. Then, the Gromov-Hausdorff distance $\dgh$ between $X$ and $Y$ is defined as follows {\cite{memoli2007use}}:
\begin{equation}\label{eq:dgh-distortion}
    \dgh(X,Y)\coloneqq\frac{1}{2}\inf_{R}\sup_{(x,y),(x',y')\in R}|d_X (x,x')- d_Y (y,y')|,
\end{equation}
where the infimum is taken over all correspondences $R$ between $X$ and $Y$. {The term appeared above $\sup_{(x,y),(x',y')\in R}|d_X (x,x')- d_Y (y,y')|$ is called the \emph{distortion} of $R$, denoted by $\dis(R)$.}

We modify Equation \eqref{eq:dgh-distortion} to obtain a one-parameter family of related quantities: given $p\in[1,\infty)$, define a quantity $\dghp{p}(X,Y)$ as follows:

\begin{equation}\label{eq:dghp-no-dist}
    \dghp{p}(X,Y)\coloneqq 2^{-\frac{1}{p}}\inf_{R}\sup_{(x,y),(x',y')\in R}\left|(d_X (x,x'))^p- (d_Y (y,y'))^p\right|^\frac{1}{p}.
\end{equation}

In this way, as $p$ increases, the discrepancy between large distance values is more heavily penalized. It turns out that for each $p\in[1,\infty)$, $\dghp{p}$ is a metric on the collection $\mathcal{U}$ of all compact ultrametric spaces. Moreover, we will later show as a consequence of Theorem \ref{thm:NP-hard} the following as one of our motivations of considering $\dghp{p}$:

\begin{restatable}{corollary}{thmapprox}\label{thm:approximate}
For each $p\in[1,\infty)$ and for any $X,Y\in\ufin$, $\dghp{p}(X,Y)$ cannot be approximated within any factor less than $3^\frac{1}{p}$
in polynomial time, unless $\mathcal{P}=\mathcal{NP}$.
\end{restatable}

Note that the factor $3^\frac{1}{p}$ approaches $1$ as $p\rightarrow\infty$. This suggests us considering $\dghp{\infty}\coloneqq\lim_{p\rightarrow\infty}\dghp{p}$, which could potentially be a computationally tractable quantity. Before stating our computational result for $\dghp{\infty}$, it is worth mentioning that $\dghp{\infty}$ turns out to be an ultrametric on $\ums$. Moreover, it actually coincides with the Gromov-Hausdorff ultrametric $\ugh$ defined by Zarichnyi \cite{zarichnyi2005gromov}. In the sequel, we will hence use $\ugh$ to denote $\dghp{\infty}$.

One of our main contributions in the paper is the following structural characterization of $\ugh$. This structural result eventually leads to a polynomial time computable algorithm for computing $\ugh$ which we will state later.

\begin{restatable}[Structural theorem for $\ugh$]{theorem}{thmugh}\label{thm:ugh-struct}
For any $X,Y\in\mathcal{U}^\mathrm{fin}$ one has that 
\begin{equation*}
    \ugh(X,Y) = \min\left\{t\geq 0:\, \lc X_{\mathfrak{c}(t)},u_{X_{\mathfrak{c}(t)}}\rc \,\mbox{is isometric to}\, \lc Y_{\mathfrak{c}(t)},u_{Y_{\mathfrak{c}(t)}}\rc \right\}.
\end{equation*}
\end{restatable}
Here $\lc X_{\mathfrak{c}(t)},u_{X_{\mathfrak{c}(t)}}\rc $ is the $t$-\emph{closed quotient} of $X$ where $x$ and $x'$ are identified if $u_X(x,x')\leq t$ (cf. Definition \ref{def:ultraquotient}). See Figure \ref{fig:strutural theorem} for an illustration of Theorem \ref{thm:ugh-struct}.

\begin{figure}[ht]
    \centering
    \includegraphics[width=0.4\textwidth]{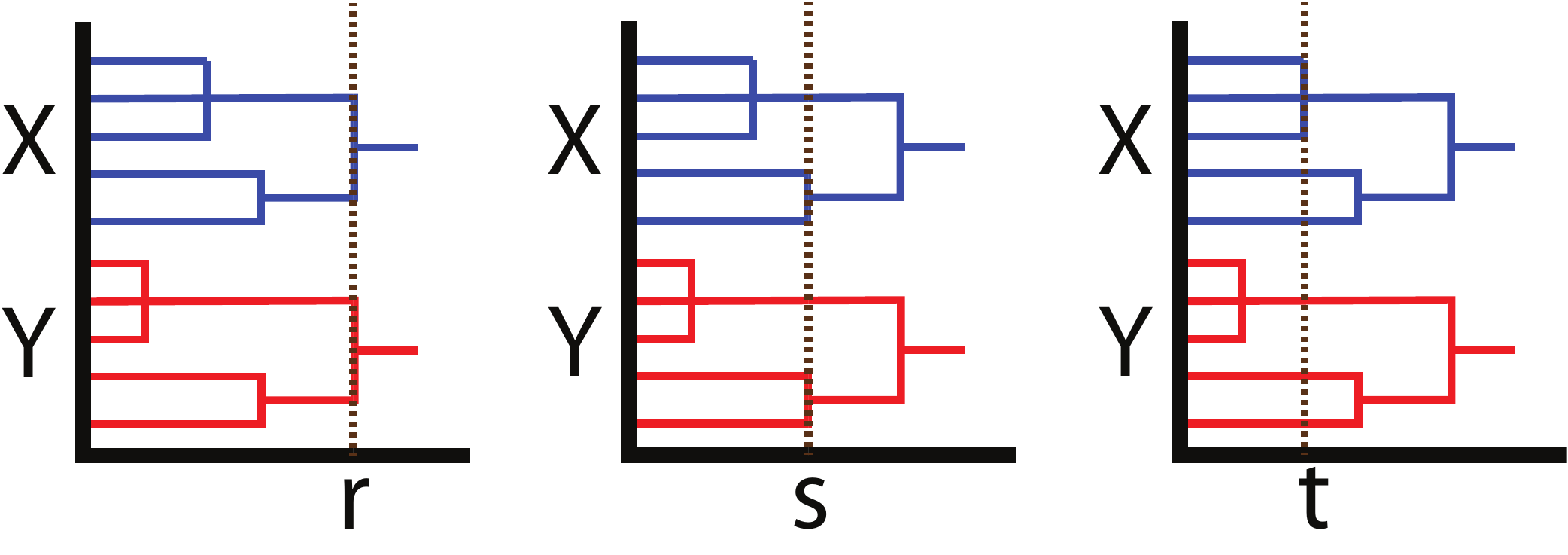}
    \caption{\textbf{Illustration of Theorem \ref{thm:ugh-struct}.} We represent two ultrametric spaces $X$ and $Y$ as dendrograms (See Theorem \ref{thm:dendroultra} for more details.). Imagine that we move a vertical dotted line from \emph{right to left} ($r>s>t$) to obtain successive quotient spaces according to the parameter indicated by the line, as described in Definition \ref{def:ultraquotient}. It is obvious from the figure that ($\cong$ denotes isometry) $X_{\mathfrak{c}(r)}\cong Y_{\mathfrak{c}(r)}$, $X_{\mathfrak{c}(s)}\cong Y_{\mathfrak{c}(s)}$, $X_{\mathfrak{c}(t)}\cong Y_{\mathfrak{c}(t)}$, and that $t$ is the minimum value such that the resulting quotient spaces are isometric. Thus, $\ugh(X,Y)=t.$ }
    \label{fig:strutural theorem}
\end{figure}

Based on Theorem \ref{thm:ugh-struct}, we devise an algorithm for computing $\ugh$ as follows. For a finite ultrametric space $X$,  the isometry type of $X_{\mathfrak{c}(t)}$ only changes finitely many times along $0\leq t<\infty$. In fact, the set of all $t$s when $X_{\mathfrak{c}(t)}$ changes its isometry type is exactly the \emph{spectrum} $\spec(X)\coloneqq\{u_X(x,x'):\,x,x'\in X\}$ of $X$. Then, in order to compute $\ugh(X,Y)$, we simply check whether $X_{\mathfrak{c}(t)}\cong Y_{\mathfrak{c}(t)}$, starting from the largest $t$ and progressively scanning all possible $t$s until reaching the smallest $t$ in $\mathrm{spec}(X)\bigcup\mathrm{spec}(Y)$; the smallest $t$ such that $X_{\mathfrak{c}(t)}\cong Y_{\mathfrak{c}(t)}$ will be $\ugh(X,Y)$. 
Since ultrametric spaces can be regarded as weighted trees (cf. Section \ref{sec:tree structure of dend}), determining whether two ultrametric spaces are isometric is equivalent to determining whether two weighted trees are isomorphic, which can be achieved in polynomial time on the number of vertices involved.

We prove that computing $\ugh$ can be done in time $O(n\log(n))$ where $n$ is the maximum of the cardinalities of $X$ and $Y$ (cf. Theorem \ref{thm:com-ugh-alg} and Remark \ref{rmk:ugh-log-alg}). See Section \ref{sec:ugh-struct} for the pseudocode (cf. Algorithm \ref{algo-uGH}) of the algorithm described above and a detailed complexity analysis, and also see Appendix \ref{sec:ext-treegram} for an extension of $\ugh$ to the case of ultra-dissimilarity spaces. We also remark that our computational results regarding the determination of $\ugh$ (between finite ultrametric spaces) can be interpreted as providing a novel computationally tractable instance of the well known quadratic assignment problem (cf. Remark \ref{rmk:QAP}).

In the end, we summarize our complexity results in Figure \ref{fig:np-hard}

\begin{figure}[ht]
    \centering
    \includegraphics[width=0.5\textwidth]{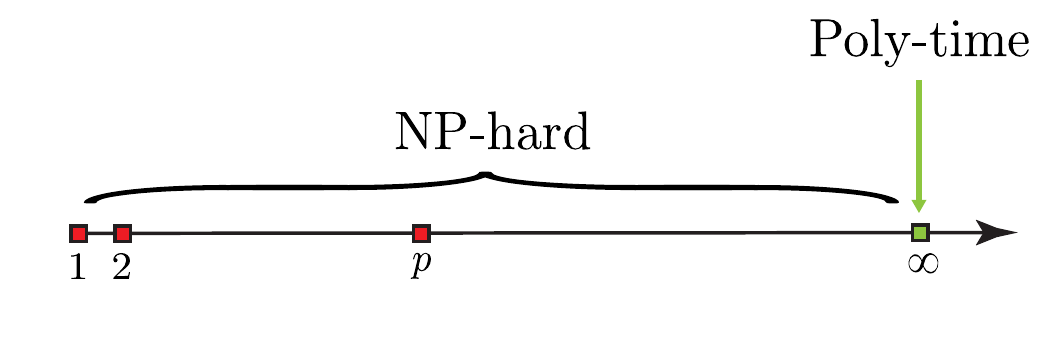}
    \caption{\textbf{Complexity of $\dghp{p}$.} Computing/approximating $\dghp{p}$ is NP-hard for each $p\in[1,\infty)$ whereas computing $\dghp{\infty}$ can be done in polynomial time.  }
    \label{fig:np-hard}
\end{figure}

\subsubsection{Polynomial time computable family with respect to $\dgh$}
Inspired by the usefulness of Theorem \ref{thm:ugh-struct} for $\ugh$, we exploit special properties of ultrametric spaces and establish a structural theorem for $\dgh$ between ultrametric spaces (Theorem \ref{thm:ums-dgh}).

\begin{figure}
  \begin{center}
    \includegraphics[width=0.35\textwidth]{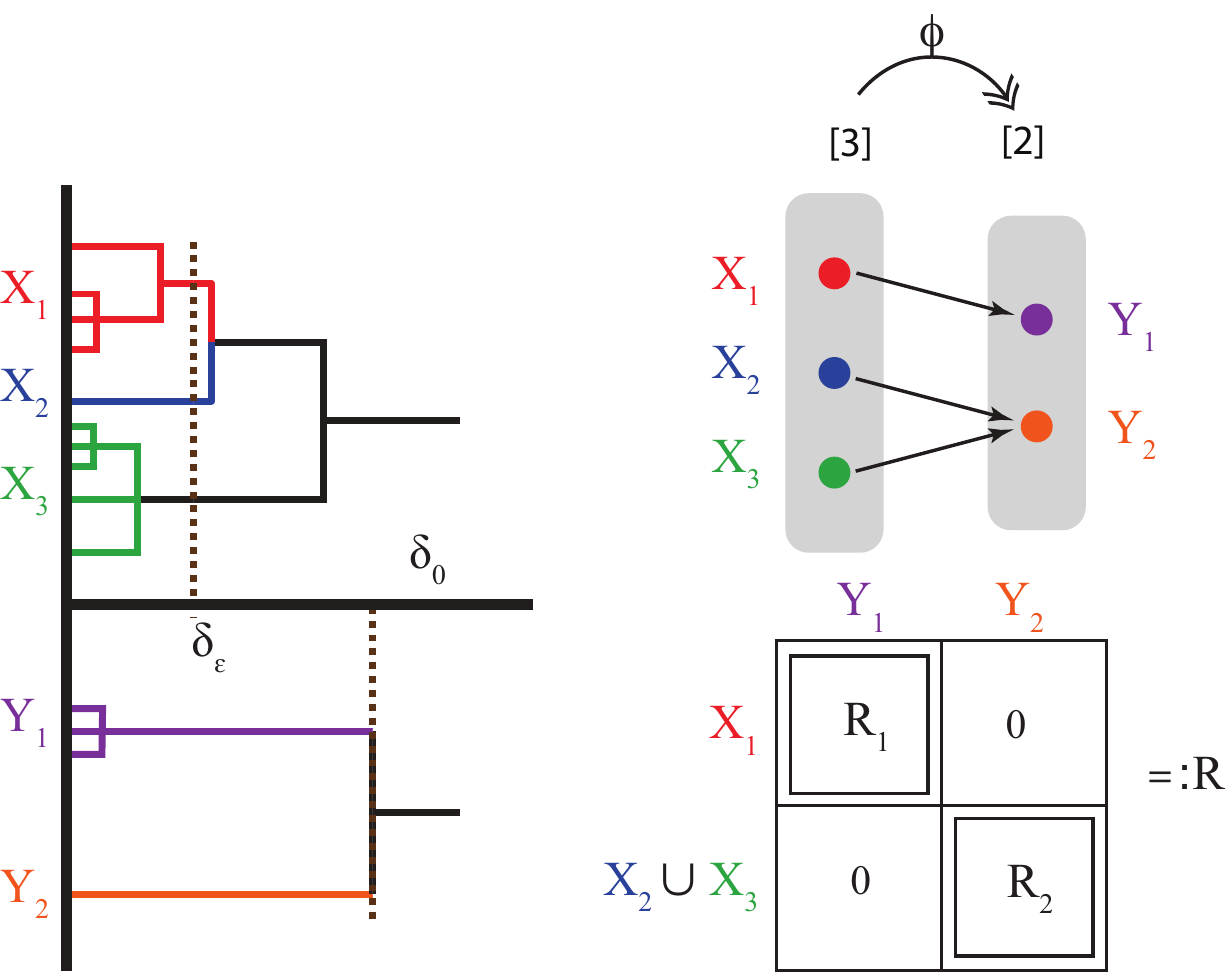}
  \end{center}
  \caption{\textbf{Illustration of the claim in Theorem \ref{thm:ums-dgh}.} }
\end{figure}

Below, $X_{\mathfrak{o}(t)}$ denotes the $t$-\emph{open partition} of $X$ where $x$ and $x'$ belong to the same block if $u_X(x,x')< t$ (cf. Definition \ref{def:ultraquotient-open}) and we call any correspondence $R$ between $X$ and $Y$ with distortion (cf. Section \ref{sec:dghp}) bounded above by $\eps\geq 0$ an \emph{$\eps$-correspondence.} Given a metric space $X$ and $\eps\geq 0$, we let $\delta_\eps(X)\coloneqq\diam(X)-\eps$. With this notation, $\delta_0(X)=\diam(X)$.

\begin{restatable}[Structural theorem for $\dgh$]{theorem}{thmdgh}\label{thm:ums-dgh}
Let $X,Y\in\mathcal{U}^\mathrm{fin}$ and $\eps\geq 0$ be such that 
\begin{equation}\label{eq:condition of thm}
    |\delta_0(X)-\delta_0(Y)|\leq\eps<\delta_0(Y).
\end{equation}
Consider the following open partitions  
\[X_{\mathfrak{o}\lc \delta_\eps(Y)\rc}:=\left\{X_i\right\}_{i=1}^{N_X} \,\,\,\mbox{and}\,\,\,\, Y_{\mathfrak{o}\lc \delta_0(Y)\rc}:=\left\{Y_j\right\}_{j=1}^{N_Y}.\] 

Then, there exists an $\eps$-correspondence between $X$ and $Y$ if and only if:

\begin{itemize}
    \item[(1)] there exists a surjection $\Psi:[N_X]\twoheadrightarrow [N_Y]$ {and, with this surjection,} 
    
    \item[(2)] for every $j\in [N_Y]$  there exists an $\eps$-correspondence between $\left(\!X_{\Psi^{-1}(j)},u_X|_{X_{\Psi^{-1}(j)}\times X_{\Psi^{-1}(j)}}\!\right)$ and $\left(Y_j,u_Y|_{Y_j\times Y_j}\right)$ where for each $j\in [N_Y]$,  $X_{\Psi^{-1}(j)}\coloneqq\bigcup_{i\in\Psi^{-1}(j)}X_i.$
\end{itemize} 
\end{restatable}

\begin{remark}[Interpretation of Equation (\ref{eq:condition of thm})]
Note that for any correspondence $R$ between $X$ and $Y$, the relations $|\delta_0(X)-\delta_0(Y)|\leq \dis(R)\leq \max(\delta_0(X),\delta_0(Y))$ always hold (cf. Proposition \ref{prop:diam-dgh}). Therefore, (1) If $|\delta_0(X)-\delta_0(Y)|>\eps$, then there exists no $\eps$-correspondence between $X$ and $Y$; (2) If $\max(\delta_0(X),\delta_0(Y))\leq\eps$, then every correspondence $R$ between $X$ and $Y$ is an $\eps$-correspondence. In this way, in order to analyze existence of $\eps$-correspondence we only need to consider the case when
\begin{equation}\label{eq:interpretation}
    |\delta_0(X)-\delta_0(Y)|\leq\eps< \max(\delta_0(X),\delta_0(Y)).
\end{equation}
Therefore, equation (\ref{eq:condition of thm}) is  simply an asymmetric variant of Equation (\ref{eq:interpretation}).
\end{remark}

\begin{remark}\label{rmk:union-corr}
In the course of proving Theorem \ref{thm:ums-dgh} (cf. Section \ref{sec:proof-dgh-str}), we actually establish the following result: under the assumption that a surjection $\Psi:[N_X]\twoheadrightarrow [N_Y]$ and an $\eps$-correspondence $R_j$ between $X_{\Psi^{-1}(j)}$ and $Y_j$ for each ${j\in [N_Y]}$ as above all exist, the set 
\[R\coloneqq\bigcup_{j\in [N_Y]}R_j\]
is an explicit $\eps$-correspondence between $X$ and $Y$. This fact will be used in Algorithms \ref{algo-dGH-rec} and \ref{algo-dGH-dyn}.
\end{remark}

The structural theorem for $\dgh$ is `anatomically' similar to  the structural theorem for $\ugh$ in that, in some sense, it converts the problem related to comparing two spaces into smaller problems related to comparing subspaces. This naturally suggests considering a \emph{divide-and-conquer strategy} for devising a recursive algorithm (Algorithm \ref{algo-dGH-rec}) for (asserting the existence of and) finding an $\eps$-correspondence between two given ultrametric spaces. It turns out that the recursive algorithm performs many repetitive computations, so we further improve this strategy via a dynamic programming (DP) idea to obtain a more efficient algorithm (Algorithm \ref{algo-dGH-dyn}). See Section \ref{sec:computing-dgh} for a detailed description of both algorithms.

One key factor which will influence the complexity of either the recursive or the DP algorithm is the \emph{size} of the subproblems. By exploiting the inherent tree-like structure of ultrametric spaces, we identified in Definition \ref{def:growth} the first $(\eps,\gamma)$-growth condition (FGC) which suitably quantifies the structural complexity of ultrametric spaces and thus controls the size of the subproblems in our recursive algorithm. When two ultrametric spaces satisfy the FGC for some fixed parameters, the recursive algorithm  (Algorithm \ref{algo-dGH-rec}) is proved to run in polynomial time (Theorem \ref{thm:complexityrecdgh}). 

A similar but more general second $(\eps,\gamma)$-growth condition (SGC) is identified in Definition \ref{def:growth2} for the DP algorithm (Algorithm \ref{algo-dGH-dyn}). If we denote by $\mathcal{U}_2(\eps,\gamma)$ the collection of all finite ultrametric spaces satisfying the second $(\eps,\gamma)$-growth condition, then for any $X,Y\in \mathcal{U}_2(\eps,\gamma)$, we can determine whether $\dgh(X,Y)\leq \frac{\eps}{2}$ in time $O\lc n^2\log(n)2^{\gamma}\gamma^{\gamma+2}\rc$, where $n:=\max(\#X,\#Y)$ (cf. Theorem \ref{thm:com-dyn-alg}); and under the assumption $2\dgh(X,Y)\leq \eps$, we can compute the exact value of $\dgh(X,Y)$ in time $O\lc n^4\log(n)2^{\gamma}\gamma^{\gamma+2}\rc$ (cf. Theorem \ref{thm:dGH-complex}). In particular, this implies that our DP algorithm is \emph{fixed-parameter tractable} (FPT) with respect to the parameters given in the SGC.

Based on our algorithms for computing $\dgh$ between ultrametric spaces, we further establish an FPT algorithm to additively approximate $\dgh$ between arbitrary doubling (non necessary ultra)  metric spaces which are themselves quantitatively close to being ultrametric spaces (cf. Corollary \ref{coro:complexity approximation}). One of the key observations leading  to this approximation algorithm is the `transfer' of the doubling condition on a metric space into the satisfaction of the second growth condition by its corresponding single-linkage ultrametric space (cf. Lemma \ref{lm:ultrametricty sgc}).

\paragraph{Implementations} The GitHub repository \cite{github2019} provides implementations of some of our algorithms as well as an experimental demonstration.

\subsection{Organization of the paper}
In Section \ref{sec:pre} we review facts about ultrametric spaces and dendrograms, and introduce the quotient operations mentioned above. In Section \ref{sec:dghp} we discuss the $p$-Gromov-Hausdorff distance $\dghp p$ and connect $\dghp\infty$ with the Gromov-Hausdorff ultrametric $\ugh$. In Section \ref{sec:ugh-struct} we prove Theorem \ref{thm:ugh-struct} and provide details of an algorithm (Algorithm \ref{algo-uGH}) for computing $\ugh$. In Section \ref{sec:computing-dgh} we prove Theorem \ref{thm:ums-dgh} and discuss how to utilize Theorem \ref{thm:ums-dgh} for devising algorithms (Algorithms \ref{algo-dGH-rec} and \ref{algo-dGH-dyn}) computing $\dgh$. In Appendix \ref{sec:data-structure} we specify the data structure for ultrametric spaces used in algorithms throughout the paper. In Appendix \ref{sec:ext-treegram} we provide details for generalizing $\ugh$ to the so-called ultra-dissimilarity spaces. Some proofs are relegated to Appendix \ref{sec:proofs}.

\printnomenclature

\section{Ultrametric spaces}\label{sec:pre}

Ultrametric spaces, as defined in the introduction, are metric spaces which satisfy the strong triangle inequality. The following basic properties of ultrametric spaces are direct consequences of the strong triangle inequality.

\begin{proposition}[Basic properties of ultrametric spaces]\label{prop:basic property ultrametric}
Let $X$ be an ultrametric space. Then, $X$ satisfies the following basic properties:
\begin{enumerate}
    \item (\emph{Isosceles triangles}) Any three distinct points $x,x',x''\in X$ constitute an isosceles triangle, i.e., two of $u_X(x,x'),u_X(x,x'')$ and $u_X(x',x'')$ are the same and are greater than the rest.
    \item (\emph{Center of closed balls}) Let $B_t(x)\coloneqq\{x'\in X:\,u_X(x,x')\leq t\}$ denote the closed ball centered at $x\in X$ with radius $t\geq 0$. Then, for any $x'\in B_t(x)$ we have that $B_t(x')=B_t(x)$.
    \item (\emph{Relation between closed balls}) For any two closed balls $B$ and $B'$ in $X$, if $B\cap B'\neq\emptyset$, then either $B\subseteq B'$ or $B'\subseteq B$.
    \item (\emph{Cardinality of spectrum}) Suppose $X$ is a finite space. Then, $\#\spec(X)\leq \#X$.
\end{enumerate}
\end{proposition}
\begin{proof}
The first three items are well known (and easy to prove) and we omit their proof. As for the fourth item, see for example \cite[Corollary 3]{gurvich2012characterizing}.
\end{proof}

Next, we introduce two important notions for ultrametric spaces: quotient operations and dendrograms.
\subsection{Quotient operations}\label{sec:quotient operation}
There are two special equivalence relations on ultrametric spaces whose respectively induced quotient operations will be helpful in revealing the structure of both $\ugh$ and $\dgh$.

\paragraph{A `closed' equivalence relation }\label{closed relation}For any ultrametric space $(X,u_X)$, we introduce a relation $\sim_{\mathfrak{c}(t)}$ on $X$ such that $x\sim_{\mathfrak{c}(t)} x'$ iff $u_X(x,x')\leq t$. Due to the strong triangle inequality, $\sim_{\mathfrak{c}(t)}$ is an equivalence relation which we call the \emph{closed equivalence relation}. For each $x\in X$ and $t\geq 0$, denote by $[x]_{\mathfrak{c}(t)}^X$ the equivalence class of $x$ under $\sim_{\mathfrak{c}(t)}$. We abbreviate $[x]_{\mathfrak{c}(t)}^X$ to $[x]_{\mathfrak{c}(t)}$ whenever the underlying set is clear from the context. Consider the set $X_{\mathfrak{c}(t)}\coloneqq\{[x]_{\mathfrak{c}(t)}:\,x\in X\}$ of all $\sim_{\mathfrak{c}(t)}$ equivalence classes.

\begin{remark}[Relationship with closed balls]\label{rmk:relation with closed ball}
Note that for each $x\in X$, the equivalence class $[x]_{\mathfrak{c}(t)}^X$ satisfies $[x]_{\mathfrak{c}(t)}^X=\{x'\in X:\,u_X(x,x')\leq t\}$. This implies that $[x]_{\mathfrak{c}(t)}$ coincides with the closed ball $ {B}_t(x)\coloneqq\{x'\in X:\,u_X(x,x')\leq t\}$. We will henceforth use both notation $[x]_{\mathfrak{c}(t)}$ and $B_t(x)$ to represent closed balls interchangeably.
\end{remark}

Now, we introduce a function $u_{X_{\mathfrak{c}(t)}}:X_{\mathfrak{c}(t)}\times X_{\mathfrak{c}(t)}\rightarrow\mathbb R_{\geq 0}$ as follows: 
\begin{equation}\label{eq:closed quotient}
    u_{X_{\mathfrak{c}(t)}}\left([x]_{\mathfrak{c}(t)},[x']_{\mathfrak{c}(t)}\right) \coloneqq  \left\{
\begin{array}{cl}
u_X(x,x') & \mbox{if $[x]_{\mathfrak{c}(t)}\neq[x']_{\mathfrak{c}(t)}$}\\
0 & \mbox{if $[x]_{\mathfrak{c}(t)}=[x']_{\mathfrak{c}(t)}$.}
\end{array}
\right.
\end{equation}

It is clear that $u_{X_{\mathfrak{c}(t)}}$ is an ultrametric on $X_{\mathfrak{c}(t)}$. 
\begin{definition}[$t$-closed quotient]\label{def:ultraquotient}
For any ultrametric space $(X,u_X)$ and $t\geq 0$, 
we call $\lc X_{\mathfrak{c}(t)},u_{X_{\mathfrak{c}(t)}}\rc $ the $t$-\emph{closed quotient of $X$}.
\end{definition}

For each $t\geq 0$, the $t$-closed quotient gives rise to a map which we call the \emph{$t$-closed quotient operator} $Q_{\mathfrak{c}\lc t\rc}:\ufin\rightarrow\ufin$ sending $X\in \ufin$ to $X_{\mathfrak{c}(t)}\in\ufin$.

\paragraph{An `open' equivalence relation}\label{open relation} Given an ultrametric space $X$ and $t> 0$, let $\sim_{\mathfrak{o}(t)}$ be such that $x\sim_{\mathfrak{o}(t)}x'$ if $u_X(x,x')<t.$ Due to the strong triangle inequality again, $\sim_{\mathfrak{o}(t)}$ is an equivalence relation on $X$ which we call the \emph{open equivalence relation}. Its difference with the closed equivalence relation $\sim_{\mathfrak{c}(t)}$ is that we now require a \emph{strict} inequality for defining the equivalence relation. Denote by $[x]_{\mathfrak{o}(t)}^X$ the equivalence class of $x\in X$ under $\sim_{\mathfrak{o}(t)}$. We will use the simpler notation $[x]_{\mathfrak{o}(t)}$ instead of $[x]_{\mathfrak{o}(t)}^X $ when the underlying set is clear from the context. 

\begin{remark}[Relation with open and closed balls]
In the same way that $[x]_{\mathfrak{c}(t)}$ is the closed ball centered at $x$ with radius $t$, when $t>0$ $[x]_{\mathfrak{o}(t)}=\{x'\in X:\, u_X(x,x')<t\}$ is actually the open ball centered at $x$ with radius $t$. If $X$ is finite, then each open ball is actually a closed ball: for any open ball $[x]_{\mathfrak{o}(t)}$,  we have $[x]_{\mathfrak{o}(t)}=[x]_{\mathfrak{c}(t')}$, where $t'\coloneqq \diam([x]_{\mathfrak{o}(t)})$.
\end{remark}

Now in analogy with Equation (\ref{eq:closed quotient}), we introduce an ultrametric $u_{X_{\mathfrak{o}(t)}}$ on $X_{\mathfrak{o}(t)}$ as follows: 
\begin{equation}\label{eq:open quotient}
    u_{X_{\mathfrak{o}(t)}}\left([x]_{\mathfrak{o}(t)},[x']_{\mathfrak{o}(t)}\right) \coloneqq  \left\{
\begin{array}{cl}
u_X(x,x') & \mbox{if $[x]_{\mathfrak{o}(t)}\neq[x']_{\mathfrak{o}(t)}$}\\
0 & \mbox{if $[x]_{\mathfrak{o}(t)}=[x']_{\mathfrak{o}(t)}$.}
\end{array}
\right.
\end{equation}

\begin{definition}[$t$-open quotient]\label{def:ultraquotient-open}
For any ultrametric space $(X,u_X)$ and any $t> 0$, 
we call $(X_{\mathfrak{o}(t)},u_{X_{\mathfrak{o}(t)}})$ the \emph{$t$-open quotient of $X$}. When $t=0$, by definition we let $(X_{\mathfrak{o}(0)},u_{X_{\mathfrak{o}(0)}})\coloneqq(X,u_X)$ be the $0$-open quotient of $X$.
\end{definition}

Given a finite set $X$, a set $P=\{B_1,\ldots,B_n\}$ of non-empty subsets of $X$ is called a \emph{partition} of $X$ if $\bigcup_{i=1}^n B_i=X$ and $B_i\bigcap B_j=\emptyset$ if $i\neq j$. It is well known that any equivalence relation on a given set induces a partition of that set. For the open equivalence relation, instead of the metric $u_{X_{\mathfrak{o}(t)}}$ we will mainly focus on the partition induced by $\sim_{\mathfrak{o}(t)}$, i.e., the partition $\left\{[x]_{\mathfrak{o}(t)}:\, x\in X\right\}$. We call this partition the $t$-\emph{open partition} of $X$.

\begin{example}[$t$-closed and open quotients when $t=$ diameter]
Let $X$ be a finite ultrametric space with at least two points and let $\delta\coloneqq\diam(X)$. Then, $X_{\mathfrak{c}(\delta)}=\{X\}$ is the one point space whereas $\#X_{\mathfrak{o}\lc \delta\rc}>1$. Indeed, for any point $x\in X$, $X=[x]_{\mathfrak{c}(\delta)}$ and thus $X_{\mathfrak{c}(\delta)}=\{X\}$; since $X$ is finite and $\#X\geq 2$, there exist $x,x'\in X$ such that $u_X(x,x')=\diam(X)=\delta$, then $[x]_{\mathfrak{o}(\delta)}\neq[x']_{\mathfrak{o}(\delta)}$ and thus $\#X_{\mathfrak{o}\lc \delta\rc}\geq \#\{[x]_{\mathfrak{o}(\delta)},[x']_{\mathfrak{o}(\delta)}\}>1$.
\end{example}

\subsection{Dendrograms}\label{sec:ultra-dendrogram}

One essential mental picture to evoke when thinking about ultrametric spaces is that of a \emph{dendrogram} (see Figure \ref{fig:ultra-dendro}). To proceed with the definition of dendrograms, we first introduce some related terminology.

\paragraph{Partitions} Given any finite set $X$ and a partition $P=\{B_1,\ldots,B_n\}$ of $X$, we call each $B_i\in P$ a \emph{block} of $P$. We denote by $\mathrm{Part}(X)$ the collection of all partitions of $X$. Given two partitions $P_1,P_2\in\mathrm{Part}(X)$, we say that $P_1$ is a \emph{refinement} of $P_2$, or equivalently, that $P_2$ is \emph{coarser} than $P_1$, if every block in $P_1$ is contained in some block in $P_2$. 

\begin{definition}[Dendrograms, \cite{carlsson2010characterization}]\label{def:dendrogram}
A dendrogram $\theta_X$ over a finite set $X$ is any function $\theta_X:[0,\infty)\rightarrow \mathrm{Part}(X)$ satisfying the following conditions:
\begin{enumerate}
    \item[(1)] $\theta_X(0)=\{\{x_1\},\ldots,\{x_n\}\}.$
    \item[(2)] For any $s<t$, $\theta_X(s)$ is a refinement of $\theta_X(t)$.
    \item[(3)] There exists $t_X>0$ such that $\theta_X(t_X)=\{X\}.$
    \item[(4)] For any $r\geq0$, there exists $\eps>0$ such that $\theta_X(r)=\theta_X(t)$ for $t\in[r,r+\eps].$
\end{enumerate}
\end{definition}

There exists a close relationship between dendrograms and ultrametric spaces. Fix a finite set $X$, by $\mathcal{U}(X)$ denote the collection of all ultrametrics over $X$ and by $\mathcal{D}(X)$ denote the collection of all dendrograms over $X$. We define a map $\Delta_X:\mathcal{U}(X)\rightarrow\mathcal{D}(X)$ by sending $u_X$ to a dendrogram $\theta_X$ as follows via the closed quotient: 
given $t\geq 0$, we let $\theta_X(t)\coloneqq X_{\mathfrak{c}(t)}=\{[x]_{\mathfrak{c}(t)}:\,x\in X\}$. It turns out that the map $\Delta_X$ is bijective. In fact, the inverse $\Upsilon_X:\mathcal{D}(X)\rightarrow\mathcal{U}(X)$ of $\Delta_X$ is the following map: for any dendrogram $\theta_X$,\label{dendrogram block} $u_X\coloneqq\Upsilon(\theta_X)$ is defined by $    u_X(x,x')\coloneqq\inf\{t\geq 0:\, [x]_{t}^{\theta_X}=[x']_{t}^{\theta_X}\}$ for any $x,x'\in X,$
where $[x]_{t}^{\theta_X}\in\theta_X(t)$ denotes the block containing $x$. It turns out that $[x]_{t}^{\theta_X}$ coincides with the equivalence class $[x]_{\mathfrak{c}(t)}$ of the closed equivalence relation with respect to $u_X=\Upsilon_X(\theta_X)$. Hence, we also use either $[x]_{\mathfrak{c}(t)}^X$ or $[x]_{\mathfrak{c}(t)}$ to represent the block containing $x$ in a dendrogram $\theta_X$ at level $t$. We summarize our discussion above into the following theorem.

\nomenclature{$[x]_{t}^{\theta_X}$}{The block in $\theta_X(t)$ containing $x$; page \pageref{dendrogram block}.}

\begin{figure}[ht]
    \centering
    \includegraphics[width=0.6\textwidth]{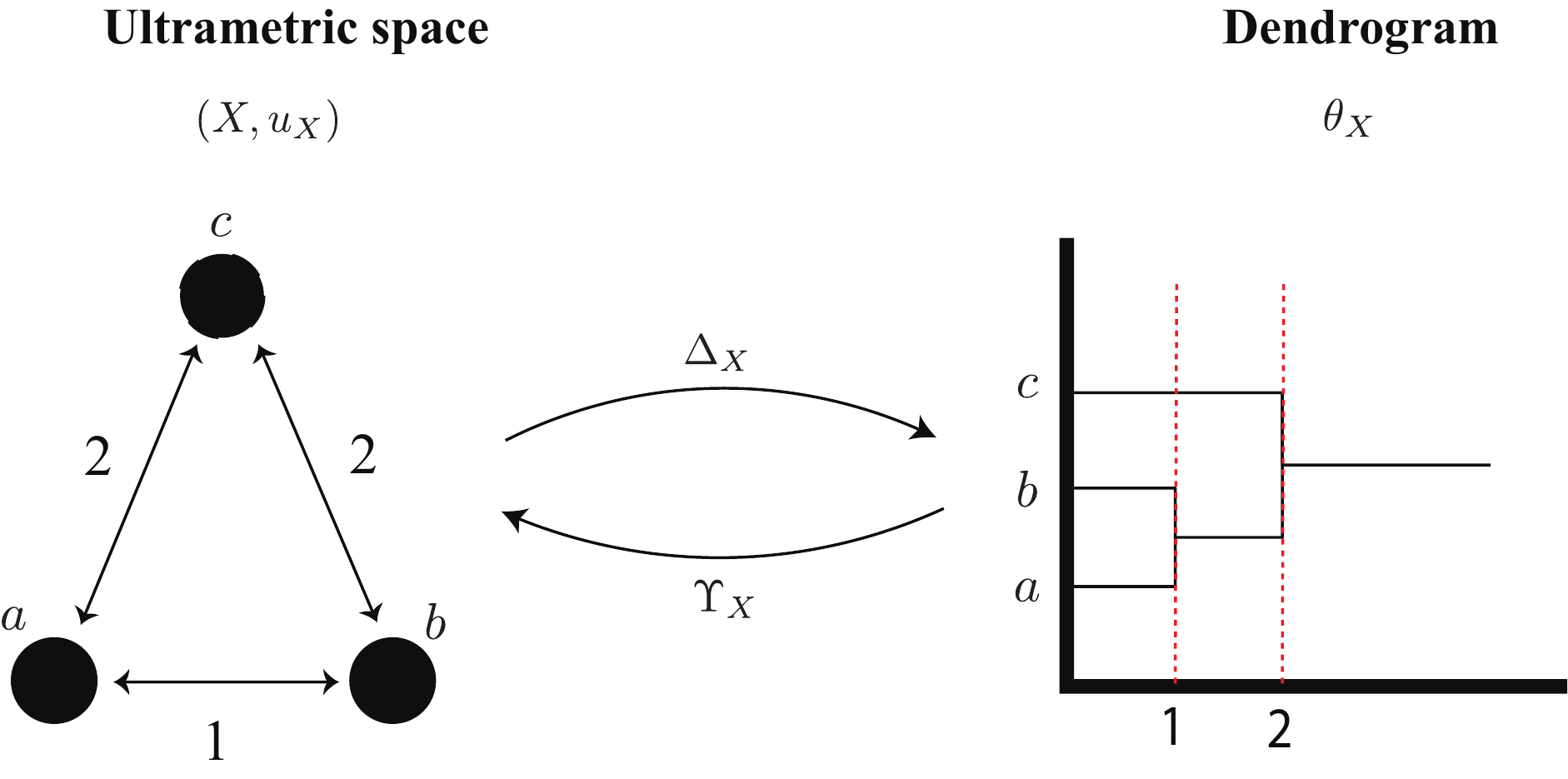}
    \caption{\textbf{Transforming ultrametric spaces into dendrograms.}}
    \label{fig:ultra-dendro}
\end{figure}

\begin{theorem}[Dendrograms as ultrametric spaces,  {\cite[Theorem 9]{carlsson2010characterization}}]\label{thm:dendroultra}
Given a finite set $X$, then $\Delta_X:\mathcal{U}(X)\rightarrow\mathcal{D}(X)$ is bijective with inverse $\Upsilon_X:\mathcal{D}(X)\rightarrow\mathcal{U}(X)$.
\end{theorem}

Theorem \ref{thm:dendroultra} above establishes that dendrograms and ultrametric spaces are equivalent concepts -- a point of view which helps to formulate subsequent ideas in this paper.

\begin{example}[$t$-closed/open quotient in terms of dendrograms]
It is helpful to understand both the open and closed $t$-quotients by viewing ultrametric spaces as dendrograms: both $t$-quotients simply forget the details of a given dendrogram strictly below scale $t$. Whereas the $t$-open equivalence relation retains the partition information at scale $t$, the $t$-closed equivalence does not. See Figure \ref{fig:open-closed-relation} for an illustration.
\end{example}

\begin{figure}[ht]
    \centering
    \includegraphics[width=0.5\textwidth]{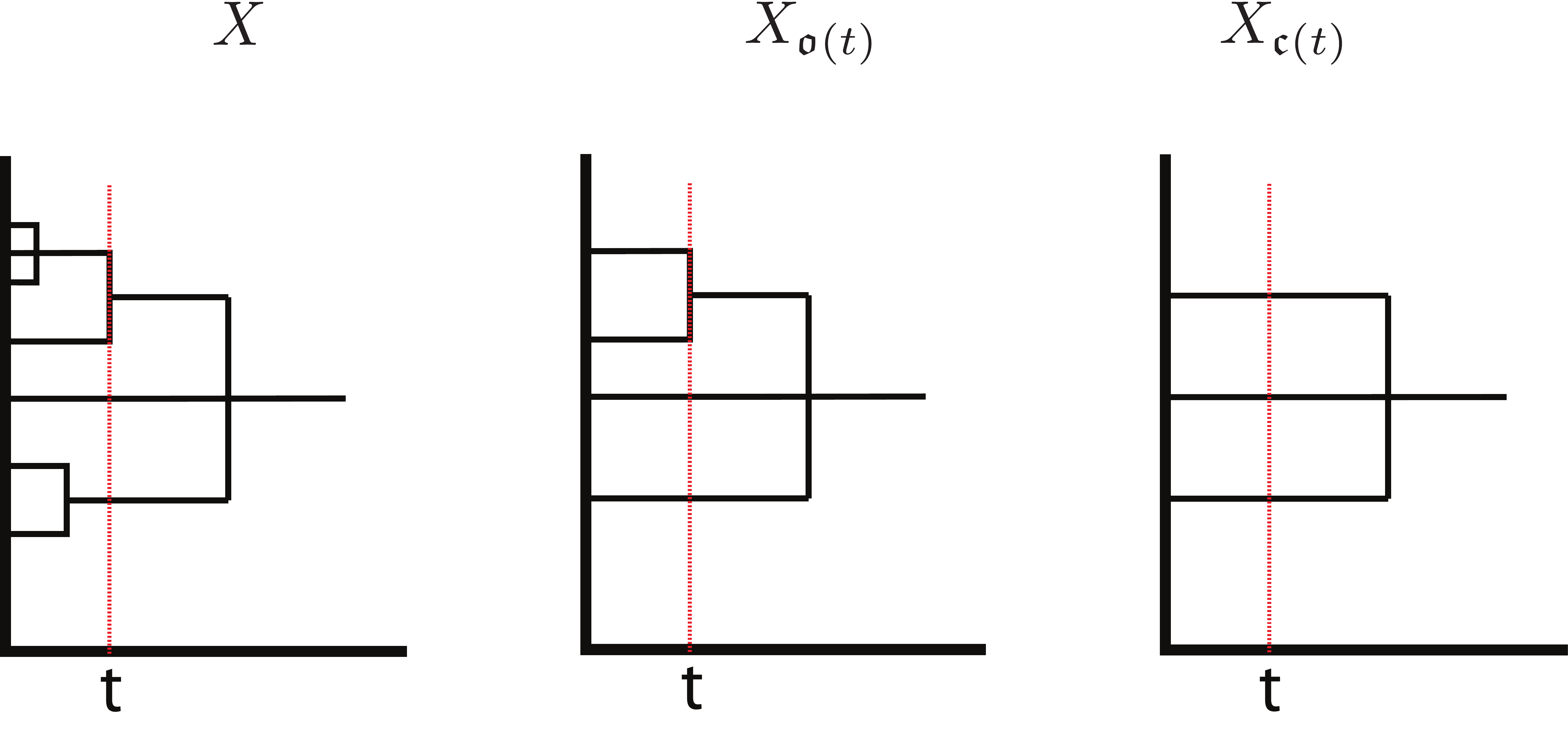}
    \caption{\textbf{Illustration of open and closed equivalence relations.} The leftmost figure is a dendrogram representing a 7-point ultrametric space $X$. The middle figure is the dendrogram corresponding to $X_{\mathfrak{o}(t)}$ whereas the rightmost figure is the dendrogram corresponding to $X_{\mathfrak{c}(t)}$. }
    \label{fig:open-closed-relation}
\end{figure}

\section{Gromov-Hausdorff distances between ultrametric spaces} \label{sec:dghp}

For convenience, we adopt the following notation to represent the \emph{absolute $p$-difference} (for $p\in[1,\infty]$) between two non-negative numbers $a,b\in\mathbb{R}_{\geq 0}$:
\begin{align*}
    \Lambda_p(a,b) &\coloneqq |a^p-b^p|^\frac{1}{p}, &\mbox{for }p\in[1,\infty); \\
    \Lambda_\infty(a,b) &\coloneqq\begin{cases}
\max(a,b),&a\neq b\\
0, & a= b
\end{cases}, &\mbox{for }p=\infty.
\end{align*}

Note that $\Lambda_1(a,b)=|a-b|$ is the usual Euclidean distance and that $\lim_{p\rightarrow\infty}\Lambda_p(a,b)=\Lambda_\infty(a,b)$. In particular, for any $a,b\geq 0$, we have the following obvious characterization of $\Lambda_\infty(a,b)$:
\begin{equation}\label{eq:lambda-infty-other-def}
    \Lambda_\infty(a,b)=\inf\{c\geq 0:\,a\leq \max(b,c)\text{ and }b\leq \max(a,c) \}.
\end{equation}
\begin{proof}
We have the following two cases.
\begin{enumerate}
    \item If $a=b$, then $\Lambda_\infty(a,b)=0\leq c.$
    \item If $a\neq b$, we assume without loss of generality that $a>b$. Then, $a\leq \max(b,c)$ implies that $\Lambda_\infty(a,b)=\max(a,b)=a\leq c$.
\end{enumerate}
\end{proof}

Now, given $p\in[1,\infty]$ and two ultrametric spaces $(X,u_X)$ and $(Y,u_Y)$, for any non-empty subset $S\subseteq X\times Y$, we define its \emph{$p$-distortion} with respect to $u_X$ and $u_Y$ as follows:
\begin{equation}\label{eq:dist}
    \dis_p\lc S,u_X,u_Y\rc\coloneqq\sup_{(x,y),(x',y')\in S}\Lambda_p(u_X(x,x'),u_Y(y,y')).
\end{equation}
We abbreviate $\dis_p\lc S,u_X,u_Y\rc$ to $\dis_p(S)$ whenever clear from the context. In particular, for a map $\varphi:X\rightarrow Y$, we define its $p$-distortion by 
\[\dis_p(\varphi)\coloneqq\dis_p(\mathrm{graph}(\varphi))=\sup_{x,x'\in X}\Lambda_p\big(u_X(x,x'),u_Y(\varphi(x),\varphi(x'))\big),\]
where $\mathrm{graph}(\varphi)\coloneqq\{(x,\varphi(x))\in X\times Y:\,x\in X\}$. Note that when $p=1$, we usually drop the subscript and simply write $\dis\coloneqq\dis_1$ and call the $1$-distortion simply the distortion.

Recall from the introduction that a {correspondence} $R$ between the underlying sets $X$ and $Y$ is any subset of $X\times Y$ such that the images of $R$ under the canonical projections $p_X:X\times Y\rightarrow X$ and $p_Y:X\times Y\rightarrow Y$ are full: $p_X(R)=X$ and $p_Y(R)=Y$. 

\begin{example}
Let $X=\{x_1,x_2\}$ and $Y=\{y_1,y_2\}$ be a pair of two-point spaces. Assume two ultrametrics $u_X$ and $u_Y$ on $X$ and $Y$, respectively, such that $u_X(x_1,x_2)=1$ and $u_Y(y_1,y_2)=2$. Let $R\coloneqq\{(x_1,y_1),(x_2,y_2)\}$. Then, $R$ is a correspondence between $X$ and $Y$. For any $p\in[1,\infty]$, it is clear that 
\[\dis_p(R,u_X,u_Y)=\Lambda_p\big( u_X (x_1,x_2), u_Y (y_1,y_2)\big)=\begin{cases}|1^p-2^p|^\frac{1}{p} & p\in[1,\infty)\\
2&p=\infty\end{cases}. \]
\end{example}

Now, for any $p\in[1,\infty]$, we define $\dghp{p}$ as follows, which is an extension of Equation \eqref{eq:dghp-no-dist} defined only for $p\in[1,\infty)$:
\begin{equation}\label{eq:dgh-p-distortion}
    \dghp p(X,Y)\coloneqq 2^{-\frac{1}{p}}\inf_{R}\dis_p(R),
\end{equation}
where the infimum is taken over all correspondences between $X$ and $Y$. Here we adopt the convention that $\frac{1}{\infty}=0$. Note that $\dghp 1=\dgh$, and as a consequence of our convention 
\[\dghp\infty(X,Y)=\inf_R\disna(R).\]

As already mentioned in the introduction, when only considering \emph{finite} ultrametric spaces, we easily have the following property of the family $\{\dghp p\}_{p\in[1,\infty]}$:

\begin{proposition}\label{prop:continuity-dghp}
Given any $X,Y\in\mathcal{U}^\mathrm{fin}$, the function $p\mapsto \dghp p(X,Y)$ is continuous and increasing with respect to $p\in[1,\infty]$. In particular, $\lim_{p\rightarrow\infty}\dghp p(X,Y)=\dghp\infty(X,Y)$.
\end{proposition}

\begin{remark}
Note that $\dghp p(X,Y)$ in Equation (\ref{eq:dgh-p-distortion}) is actually well-defined for any two metric spaces $X$ and $Y$, i.e., $X$ and $Y$ are not restricted to be ultrametric spaces. See Section \ref{sec:alternative formulate} for alternative definitions of $\dghp p(X,Y)$ on classes of metric spaces larger than $\mathcal{U}$. 
\end{remark}

\begin{example}[Distance to the one point space]
Since there exists a unique correspondence $R_\ast \coloneqq X\times \ast$ between a given finite set $X$ and the one point space $\ast$, we have for each $p\in[1,\infty]$ that 
\[\dlp(X,\ast)=2^{-\frac{1}{p}}\disp(R_\ast) = 2^{-\frac{1}{p}}\diam(X).\]
\end{example}

\subsection{Computing $\dghp p$ is NP-hard when $p\in[1,\infty)$}
Given an ultrametric space $(X,u_X)$ and any positive real number $\alpha$, the function $(x,x')\mapsto \lc u_X(x,x')\rc^\alpha$ is still an ultrametric on $X$ so that the space $(X,(u_X)^\alpha)$ is still an ultrametric space. We let $S_\alpha(X)$ denote the ultrametric space $(X,(u_X)^\alpha)$. Then, we have the following transformation between $\dghp p$ and $\dgh$ for all $p\in[1,\infty)$. The proof of the following result is relegated to Appendix \ref{sec:proofs}. 

\begin{proposition}\label{prop:dgh-dlp-eq}
Given $1\leq p<\infty$ and any two ultrametric spaces $X$ and $Y$, one has 
\[\dlp(X,Y) = \big(\dgh(S_p(X),S_p(Y))\big)^\frac{1}{p}.\]
Conversely, 
\[\dgh(X,Y)=\lc\dlp\lc S_\frac{1}{p}(X),S_\frac{1}{p}(Y)\rc\rc^p. \]
\end{proposition}

Therefore, solving an instance of $\dghp p$ is equivalent to solving an instance of $\dgh$. Since it is NP-hard to compute $\dgh$ between finite ultrametric spaces \cite{schmiedl2017computational}, it is also NP-hard to compute $\dghp p$ between finite ultrametric spaces for every $p\in[1,\infty)$.

Moreover, combining Proposition \ref{prop:dgh-dlp-eq} with Theorem \ref{thm:NP-hard}, we have the following more precise statement:
\thmapprox*
\begin{proof}
Suppose otherwise that there exist ultrametric spaces $X$ and $Y$ such that $\dghp p(X,Y)$ can be approximated within a factor $c^\frac{1}{p}<3^\frac{1}{p}$ in polynomial time. Then, Proposition \ref{prop:dgh-dlp-eq} implies that one can approximate $\big(\dgh(S_p(X),S_p(Y))\big)^\frac{1}{p}$ within a factor $c^\frac{1}{p}<3^\frac{1}{p}$ in polynomial time. This implies that one can approximate $\dgh(S_p(X),S_p(Y))$ within a factor $c<3$ in polynomial time which contradicts with Theorem \ref{thm:NP-hard}.
\end{proof}

As mentioned in the introduction, the factor $3^\frac{1}{p}$ in Corollary \ref{thm:approximate} approaches $1$ as $p\rightarrow\infty$. This suggests that $\dghp{\infty}$ could be a computationally tractable quantity; later in Section \ref{sec:comp-comp-ugh} we will illustrate this point.

\subsection{An estimate of $\dghp p$ via diameters of input spaces}
The following result shows how the Gromov-Hausdorff distance interacts with the diameters of the input spaces.
\begin{proposition}[{\cite[Theorems 3.3 and 3.4]{memoli2012some}}]\label{prop:diam-dgh}
For any finite metric spaces $X$ and $Y$, let $R$ be a correspondence between them. Then, we have
\[|\diam(X)-\diam(Y)|\leq \dis(R)\leq \max(\diam(X),\diam(Y)). \]
In particular, 
\[\frac{1}{2}|\diam(X)-\diam(Y)|\leq\dgh(X,Y)\leq \frac{1}{2}\max(\diam(X),\diam(Y)). \]
\end{proposition}

Invoking Proposition \ref{prop:dgh-dlp-eq}, we immediately have the following analogue to the second part of Proposition \ref{prop:diam-dgh} for $\dghp p$: 

\begin{proposition}\label{prop:diam-dghp}
For any $p\in[1,\infty)$ and finite ultrametric spaces $X$ and $Y$, we have that
\[2^{-\frac{1}{p}}\Lambda_p(\diam(X),\diam(Y))\leq\dghp p(X,Y)\leq 2^{-\frac{1}{p}}\max(\diam(X),\diam(Y)). \]
\end{proposition}

Note that $\lim_{p\rightarrow\infty}2^{-\frac{1}{p}}\Lambda_p(\diam(X),\diam(Y))=\Lambda_\infty(\diam(X),\diam(Y))$. When $\diam(X)\neq\diam(Y)$, $\lim_{p\rightarrow\infty}2^{-\frac{1}{p}}\Lambda_p(\diam(X),\diam(Y))=\max(\diam(X),\diam(Y))$.
Then, by continuity of $\dghp p$ with respect to $p\in[1,\infty]$ (cf. Proposition \ref{prop:continuity-dghp}), we have the following property for $\dghp\infty$:

\begin{proposition}\label{prop:diam-dgh-infty}
Let $X$ and $Y$ be any two finite ultrametric spaces with different diameters, then
\[ \dghp\infty(X,Y)=\max(\diam(X),\diam(Y)).\]
\end{proposition}
This proposition indicates that, when $X$ and $Y$ have different diameters,  $\dghp\infty(X,Y)$ is determined by the diameter values of the input spaces, which suggests that $\dghp\infty$ is more rigid than other $\dghp p$ when $p<\infty$ and as a consequence, $\dghp\infty$ exhibits a distinct computational behavior in comparison to $\dghp p$ when $p<\infty$.

\begin{example}[Distance between homothetic spaces]
Let $X$ be a finite ultrametric space and let $\lambda>0$. Then, for any $p\in[1,\infty]$ we have
\[\dghp p((X,u_X),(X,\lambda\cdot u_X))=2^{-\frac{1}{p}}\Lambda_p(\lambda,1)\cdot\diam(X). \]
Indeed, since $\diam(X,\lambda\cdot u_X)=\lambda\cdot \diam(X)$, by Proposition \ref{prop:diam-dghp} and Proposition \ref{prop:diam-dgh-infty} we have that 
\[\dghp p((X,u_X),(X,\lambda\cdot u_X))\geq 2^{-\frac{1}{p}}\Lambda_p(\lambda\cdot\diam(X),\diam(X))=2^{-\frac{1}{p}}\Lambda_p(\lambda,1)\cdot\diam(X). \]
For the converse, consider the identity correspondence $R_\mathrm{id}\coloneqq\{(x,x)\in X\times X:x\in X\}$. Then,

\[\dghp p((X,u_X),(X,\lambda\cdot u_X))\leq 2^{-\frac{1}{p}}\dis_p\big(R_\mathrm{id},u_X,\lambda\cdot u_X\big)= 2^{-\frac{1}{p}}\Lambda_p(\lambda,1)\cdot\diam(X).\]
\end{example}

\subsection{$\dghp\infty$ exactly coincides with the Gromov-Hausdorff ultrametric $\ugh$}\label{sec:alternative formulate}
Recall that on a metric space $Z$, the \emph{Hausdorff distance} $d_\mathrm{H}^Z$ between two subsets $A,B\subseteq Z$ is defined as 
\[   d_\mathrm{H}^Z(A,B):=\max\left(\sup_{a\in A}\inf_{b\in B} d_Z(a,b),\sup_{b\in B}\inf_{a\in A} d_Z(a,b)\right).\]
Given two metric spaces $X$ and $Y$, we say a map $\varphi:X\rightarrow Y$ is an \emph{isometric embedding}  from $X$ to $Y$ if for every $x,x'\in X$, $d_X(x,x')=d_Y(\varphi(x),\varphi(x'))$. We usually use the symbol $\hookrightarrow$ (instead of $\rightarrow$) to represent isometric embeddings. Then, the Gromov-Hausdorff distance can be characterized as follows:

\begin{theorem}[{Duality formula for $\dgh$, \cite[Theorem 7.3.25]{burago}}]\label{def:dgh}
The Gromov-Hausdorff distance $\dgh$ between two compact metric spaces $X$ and $Y$ satisfies the following:
\begin{equation}\label{eq:dgh-hausdorff}
    \dgh(X,Y)=\inf d^Z_\mathrm{H}(\varphi_X(X),\varphi_Y(Y)),
\end{equation}
where the infimum is taken over all metric spaces $Z$ and isometric embeddings $\varphi_X:X\hookrightarrow Z$ and $\varphi_Y:Y\hookrightarrow Z$. 
\end{theorem}

In fact, Equation (\ref{eq:dgh-hausdorff}) is the original definition of the Gromov-Hausdorff distance given by Gromov in \cite{gromov1981groups}. In the spirit of Equation (\ref{eq:dgh-hausdorff}), Zarichnyi defines in \cite{zarichnyi2005gromov} the \emph{Gromov-Hausdorff ultrametric}, {which we denote by $\ugh$}, between \emph{compact} ultrametric spaces $X$ and $Y$ as follows:
\begin{equation*}\label{eq:ugh}
    \ugh(X,Y):=\inf d^Z_\mathrm{H}(\varphi_X(X),\varphi_Y(Y)),
\end{equation*}
where the infimum is taken over all \emph{ultrametric spaces} $Z$ and isometric embeddings $\varphi_X:X\hookrightarrow Z$ and $\varphi_Y:Y\hookrightarrow Z$. It turns out that $\ugh$ agrees with $\dghp\infty$ as defined by Equation (\ref{eq:dgh-p-distortion}).

\begin{theorem}[{Duality formula for $\dghp\infty$}]\label{thm:ugh-dual}
Given two \emph{compact} ultrametric spaces $X$ and $Y$, we have that
\begin{equation}\label{eq:distortion-ugh}
    \dghp\infty(X,Y)=\ugh(X,Y).
\end{equation}
\end{theorem}
See Appendix \ref{sec:proofs} for the proof.

Zarichnyi proved in \cite{zarichnyi2005gromov} that $\ugh$ (and thus $\dghp\infty$) is an ultrametric on the collection of all compact ultrametric spaces. Similar metric properties also hold for $\dghp p$, for $p\in[1,\infty)$; see the following remark.

\begin{remark}[Duality formula for $\dghp p$]
A similar alternative formulation via the Hausdorff distance exists for $\dghp p$ for each $p\in[1,\infty]$. For $p\in[1,\infty)$, we call a metric space $X$ a \emph{$p$-metric space} if it satisfies the $p$-\emph{triangle inequality}: 
\[\forall x,x',x''\in X,\,(d_X(x,x'))^p\leq (d_X(x,x''))^p+(d_X(x'',x'))^p.\]
In that case, we refer to $d_X$ as a $p$-metric. Then, for any two compact $p$-metric spaces, we have that $\dghp p(X,Y)=\inf d^Z_\mathrm{H}(\varphi(X),\varphi_Y(Y)),$ where the infimum is taken over all \emph{$p$-metric spaces} $Z$ and isometric embeddings $\varphi_X:X\hookrightarrow Z$ and $\varphi_Y:Y\hookrightarrow Z$. Moreover, $\dghp p$ is actually a $p$-metric on the collection of isometry classes of compact $p$-metric spaces. We do not provide details here;  see our technical report \cite{memoli2019gromov} for  these.
\end{remark}

\section{Structural results for $\ugh$ and computational implications} \label{sec:ugh-struct}

In this section, we first prove our central observation regarding $\ugh$, the structural theorem for $\ugh$ (Theorem \ref{thm:ugh-struct}). Then, we utilize this theorem to devise a poly-time algorithm for computing $\ugh$. We remark that the distance $\ugh$ as well as Theorem \ref{thm:ugh-struct} and Algorithm \ref{algo-uGH} can be extended to the so-called ultra-dissimilarity spaces, which can be regarded as generalization of ultrametric spaces. See Appendix \ref{sec:ext-treegram} for details.

\subsection{Proof of Theorem \ref{thm:ugh-struct}}

Recall the statement of the structural theorem for $\ugh$:
\thmugh*

\begin{proof}
We first prove the following weaker version of Theorem \ref{thm:ugh-struct} (with $\inf$ instead of $\min$):
\begin{equation}\label{eq:ughinf}
    \ugh(X,Y) = \inf\left\{t\geq 0:\, \lc X_{\mathfrak{c}(t)},u_{X_{\mathfrak{c}(t)}}\rc \cong \lc Y_{\mathfrak{c}(t)},u_{Y_{\mathfrak{c}(t)}}\rc \right\}.
\end{equation}

Suppose first that $X_{\mathfrak{c}(t)}\cong Y_{\mathfrak{c}(t)}$ for some $t\geq 0$, i.e. there exists an isometry  $f_t:X_{\mathfrak{c}(t)}\rightarrow Y_{\mathfrak{c}(t)}$. Define 
\[R_t\coloneqq \left\{(x,y)\in X\times Y:\,[y]_{\mathfrak{c}(t)}^Y= f_t\lc [x]^X_{\mathfrak{c}(t)}\rc\right\}.\]
That $R_t$ is a correspondence between $X$ and $Y$ is clear since: $f_t$ is bijective, every $x\in X$ belongs to exactly one block in $X_{\mathfrak{c}(t)}$ and every $y\in Y$ belongs to exactly one block in $Y_{\mathfrak{c}(t)}$. For any $(x,y),(x\p,y\p)\in R_t$, if $u_X(x,x')\leq t$, then we already have $u_X(x,x')\leq \max(t,u_Y(y,y'))$. Otherwise, if $u_X(x,x\p)>t$, then we have $[x]^X_{\mathfrak{c}(t)}\neq[x']^X_{\mathfrak{c}(t)}$. Since $f_t$ is bijective, we have that $[y]^Y_{\mathfrak{c}(t)}=f_t([x]^X_{\mathfrak{c}(t)})\neq f_t([x']^X_{\mathfrak{c}(t)})=[y']^Y_{\mathfrak{c}(t)}.$ Then,
\[u_Y(y,y\p)=u_{Y_{\mathfrak{c}(t)}}\lc [y]^Y_{\mathfrak{c}(t)},[y']^Y_{\mathfrak{c}(t)}\rc=u_{X_{\mathfrak{c}(t)}}\lc [x]^X_{\mathfrak{c}(t)},[x']^X_{\mathfrak{c}(t)}\rc=u_X(x,x\p).\]
Therefore, $u_X(x,x\p)\leq\max(t,u_Y(y,y\p))$. Similarly, $u_Y(y,y\p)\leq\max(t,u_X(x,x\p))$. Then, by Equation (\ref{eq:lambda-infty-other-def}), $\Lambda_\infty(u_X(x,x'),u_Y(y,y'))\leq t$ and thus $\disna(R_t)\leq t$. This implies that 
\[\ugh(X,Y) \leq \inf\left\{t\geq 0:\, X_{\mathfrak{c}(t)}\cong Y_{\mathfrak{c}(t)}\right\}.\]

Conversely, let $R$ be any correspondence between $X$ an $Y$ and let $t\coloneqq\disna(R)$.
Consider any $(x,y),(x\p,y\p)\in R$ such that $[x']^X_{\mathfrak{c}(t)}=[x]^X_{\mathfrak{c}(t)}$ (i.e. $u_X(x\p,x)\leq t$). We define a map $f_t:X_{\mathfrak{c}(t)}\rightarrow Y_{\mathfrak{c}(t)}$ as follows: for each $[x]_{\mathfrak{c}(t)}^X\in X_{\mathfrak{c}(t)}$, suppose $y\in Y$ is such that $(x,y)\in R$, then we let $f_t\lc[x]_{\mathfrak{c}(t)}^X\rc\coloneqq[y]_{\mathfrak{c}(t)}^Y$.  

$f_t$ is well-defined. Indeed, if $[x]_{\mathfrak{c}(t)}^X=[x']_{\mathfrak{c}(t)}^X$ and $y,y'\in Y$ are such that $(x,y),(x',y')\in R$, then 
\[\Lambda_\infty(u_X(x,x'),u_Y(y,y'))\leq\disna(R)=t.\]
This implies that $u_Y(y\p,y)\leq t$ which is equivalent to $[y]^Y_{\mathfrak{c}(t)}=[y']^Y_{\mathfrak{c}(t)}$. Similarly, there is a well-defined map $g_t:Y_{\mathfrak{c}(t)}\rightarrow X_{\mathfrak{c}(t)}$ sending $[y]_{\mathfrak{c}(t)}^Y\in Y_{\mathfrak{c}(t)}$ to $[x]_{\mathfrak{c}(t)}^X$ whenever $(x,y)\in R$. It is clear that $g_t$ is the inverse of $f_t$ and thus $f_t$ is bijective.
Now suppose that $u_{X_{\mathfrak{c}(t)}}\lc[x]^X_{\mathfrak{c}(t)},[x']^X_{\mathfrak{c}(t)}\rc =s>t$, which implies that $u_X(x,x\p)=s$. Let $y,y'\in Y$ be such that $(x,y),(x',y')\in R$. Then, since $\Lambda_\infty(u_X(x,x'),u_Y(y,y'))\leq\disna(R)=t<s$, $u_Y(y,y\p)$ is forced to be equal to $u_X(x,x')=s$. Therefore,
\[u_{Y_{\mathfrak{c}(t)}}\lc f_t\lc[x]^X_{\mathfrak{c}(t)}\rc,f_t\lc[x']^X_{\mathfrak{c}(t)}\rc\rc=u_{Y_{\mathfrak{c}(t)}}\lc [y]^Y_{\mathfrak{c}(t)},[y']^Y_{\mathfrak{c}(t)}\rc=s=u_{X_{\mathfrak{c}(t)}}\lc[x]^X_{\mathfrak{c}(t)},[x']^X_{\mathfrak{c}(t)}\rc.\] 
This proves that $f_t$ is an isometry and thus \[\ugh(X,Y) \geq \inf\left\{t\geq 0:\, X_{\mathfrak{c}(t)}\cong Y_{\mathfrak{c}(t)}\right\}.\]

Since $X$ and $Y$ are finite, for each $t\geq 0$, there exists $\eps>0$ such that for all $s\in[t,t+\eps]$ $X_{\mathfrak{c}(t)}\cong X_{\mathfrak{c}(s)}$ and $Y_{\mathfrak{c}(t)}\cong Y_{\mathfrak{c}(s)}$. This implies that the infimum in Equation (\ref{eq:ughinf}) is attained and thus we obtain the claim.
\end{proof}

\begin{remark}
Theorem \ref{thm:ugh-struct} actually holds for compact ultrametric spaces; see our technical report \cite{memoli2019gromov} for  details.
\end{remark}

\subsection{A poly-time algorithm for computing $\ugh$} \label{sec:comp-comp-ugh}
In Algorithm \ref{algo-uGH} below we provide pseudocode for computing $\ugh$ and in Theorem \ref{thm:com-ugh-alg} we prove that Algorithm \ref{algo-uGH} runs in time $O(n^2)$; see also Remark \ref{rmk:ugh-log-alg} for details about improving  this time complexity  to  $O(n\log(n))$.

Recall that the spectrum $\mathrm{spec}(X)$ of the metric space $X$ is the set of values defined by $\mathrm{spec}(X)\coloneqq \{u_X(x,x'):\,\forall x,x'\in X\}$. The pseudocode for the function $\mathbf{ClosedQuotient}$ implementing the closed quotient operation is given in Algorithm \ref{algo-c-quotient} in Appendix \ref{sec:data-structure}. The function $\mathbf{is\_iso}$ determines whether two ultrametric spaces are isometric, for which we adapt the algorithm in \cite[Example 3.2]{aho1974design}.

 \begin{algorithm}[htb]
\caption{$\mathbf{uGH}(X,Y)$}\label{algo-uGH}
\begin{algorithmic}[1]
  \STATE $\spec \gets$ \textbf{sort}($\spec(X)\cup \spec(Y)$, `descend')
  \FOR{$i=1: \mathrm{length(spec)}$}
     \STATE{$t = \spec(i)$}
     \IF{$\thicksim \mathbf{is\_iso}(\mathbf{ClosedQuotient}(X,t),\mathbf{ClosedQuotient}(Y,t))$}
     \RETURN $\spec(i-1)$
     \ENDIF
     \ENDFOR
     \RETURN $0$
\end{algorithmic}
\end{algorithm}

\paragraph{Complexity analysis of Algorithm \ref{algo-uGH}} 
Let $n\coloneqq\max(\#X,\#Y)$. By Proposition \ref{prop:basic property ultrametric}, $\#\spec(X)\leq\#X$. Then,
\[\#\spec=\#\left(\mathrm{spec}(X)\bigcup\mathrm{spec}(Y)\right)=O(n).\]
Thus, it takes time $O(n\log(n))$ to construct and to sort the sequence $\spec\coloneqq\mathrm{spec}(X)\bigcup\mathrm{spec}(Y)$ (cf. Lemma \ref{lm:spectrum computation}). Now, for each $t\in\mathrm{spec}$, we need time $O(n)$ for running Algorithm $\mathbf{ClosedQuotient}$ (Algorithm \ref{algo-c-quotient}) with input $(X,t)$ and $(Y,t)$.

Following Appendix \ref{sec:data-structure} and Lemma \ref{lm:isometry linear time}, since $\max(\#X_{\mathfrak{c}(t)},\#Y_{\mathfrak{c}(t)})\leq \max(\#X,\#Y)=O(n)$, the function $\mathbf{is\_iso}$ with input $(X_{\mathfrak{c}(t)},Y_{\mathfrak{c}(t)})$ runs in time $O(n)$ as well (cf. Lemma \ref{lm:isometry linear time}).

Thus, the time complexity associated to computing $\ugh(X,Y)$ via Algorithm \ref{algo-uGH} is
\[O(n\log(n))+\mathrm{length(spec)}\cdot O(n)=O(n\log(n))+O(n)\cdot O(n)=O(n^2).\]

In this way we have proved the following theorem.

\begin{theorem}[Time complexity of Algorithm $\mathbf{uGH}$ (Algorithm \ref{algo-uGH})]\label{thm:com-ugh-alg}
Let $X$ and $Y$ be finite ultrametric spaces. Then, algorithm $\mathbf{uGH}(X,Y)$ (Algorithm \ref{algo-uGH}) runs in time $O\lc n^2\rc$, where $n\coloneqq\max(\#X,\#Y)$.
\end{theorem}

\begin{remark}[Acceleration via binary search]\label{rmk:ugh-log-alg}
By replacing the for-loop over $i=1: \mathrm{length(spec)}$ in $\mathbf{uGH}$ (Algorithm \ref{algo-uGH}) with binary search, the total complexity will drop to $ O(n\log(n))+O(\log(n))\cdot O(n) =O(n\log(n)).$
\end{remark}

\begin{remark}[A novel poly time solvable instance of the quadratic assignment problem]\label{rmk:QAP}
Given $p\in[1,\infty]$ and two finite ultrametric spaces $X=\{x_1,\ldots,x_{n_X}\}$ and $Y=\{y_1,\ldots,y_{n_Y}\}$, we formulate the computation of $\dghp{p}(X,Y)$ as the following \emph{generalized}\footnote{Here, `generalized' refers to the fact that we are allowing matchings more general than permutations.} bottleneck quadratic assignment problem ($\mathrm{GQBAP}_p$) as in \cite[Remark 3.4]{memoli2012some} (cf. Equation (\ref{eq:dgh-p-distortion})):
\[\mathrm{GQBAP}_p(\mathbf{a},\mathbf{b})\coloneqq 2^{\frac{1}{p}}\,\dghp{p}(X,Y)=\min_R\max_{i,j,k,l}\Lambda_p(a_{ik},b_{jl})\,R_{ij}\,R_{kl}, \]
where $\mathbf{a}=(a_{ik})$ is an $n_X\times n_X$ matrix such that $a_{ik}\coloneqq u_X(x_i,x_k)$,  $\mathbf{b}=(b_{jl})$ is an $n_Y\times n_Y$ matrix such that $b_{jl}\coloneqq u_Y(y_j,y_l)$, and $R$ is a correspondence, which is regarded as $(R_{ij})$, an $n_X\times n_Y$ matrix such that  $R_{ij}\in\{0,1\}$  and
\begin{enumerate}
    \item $\sum_{i=1}^{n_X}R_{ij}\geq 1$ for all $j$;
    \item $\sum_{j=1}^{n_Y}R_{ij}\geq 1$ for all $i$.
\end{enumerate}

Then, by Corollary \ref{thm:approximate} and Theorem \ref{thm:com-ugh-alg}, whereas solving $\mathrm{GQBAP}_p$ is NP-hard for each $p\in[1,\infty)$, the problem $\mathrm{GQBAP}_\infty$ can be solved in time $O(n\log(n))$ where $n\coloneqq\max(n_X,n_Y)$.

In general, quadratic assignment problems are NP-hard  \cite{pardalos1994quadratic}. This includes instances such as $\mathrm{GQBAP}_p$ in the case when $p<\infty$. However, by the above `cost' matrices of the form of $\left(\Lambda_\infty(a_{ik},b_{jl})\right)$ yield computationally tractable instances. 
\end{remark}

\section{Structural results  for $\dgh$ and computational implications}\label{sec:computing-dgh}

In this section, we first prove the structural theorem for $\dgh$ (Theorem \ref{thm:ums-dgh}), and then develop efficient algorithms for computing $\dgh$ based on Theorem \ref{thm:ums-dgh}.

\subsection{Proof of Theorem \ref{thm:ums-dgh}}\label{sec:proof-dgh-str}
Recall the structural theorem for the Gromov-Hausdorff distance:

\thmdgh*

\begin{proof}
First suppose that there exists an $\eps$-correspondence $R$ between $X$ and $Y$. Then, we define a map $\Psi:[N_X]\rightarrow [N_Y]$ as follows: for any $i\in[N_X]$, pick an arbitrary $x\in X_i$ and assume that $(x,y)\in R$ for some $y\in Y$; further assume that $y\in Y_j$ for some $j\in[N_Y]$, then we let $\Psi(i)\coloneqq j$. Now, we verify that this map $\Psi$ is well-defined, i.e., $\Psi$ is independent of choice of $x\in X_i$ and choice of $(x,y)\in R$. For any $i\in [N_X]$ and $x,x\p\in X_i$, we have by assumption that $u_X(x,x\p)<\delta_\eps(Y)$. Suppose $y,y\p\in Y$ are such that $(x,y),(x\p,y\p)\in R$. Then, 
\[u_Y(y,y\p)\leq u_X(x,x')+\dis(R)\leq u_X(x,x\p)+\eps<\delta_\eps(Y)+\eps=\delta_0(Y).\]
Therefore, there exists a common $j\in [N_Y]$ such that both $y\in Y_j$ and $y\p\in Y_j$. This implies that $\Psi$ is well-defined. Since $R$ is a correspondence, $\Psi$ must be surjective. Then, for each $j\in[N_Y]$, we define the set
\[R_j\coloneqq R\bigcap (X_{\Psi^{-1}(j)}\times Y_j)\]
It is obvious that $R_j$ is a correspondence between $X_{\Psi^{-1}(j)}$ and $ Y_j$.  Moreover, $\dis(R_j)\leq \dis(R)\leq \eps$ for each $j\in[N_Y]$. Therefore, for each $j\in[N_Y]$, $R_j$ is an $\eps$-correspondence between $X_{\Psi^{-1}(j)}$ and $ Y_j$.

Conversely, suppose that there exist a surjection $\Psi:[N_X]\twoheadrightarrow [N_Y]$ and for each $j\in [N_Y]$ an $\eps$-correspondence $R_j$ between $X_{\Psi^{-1}(j)}$ and $Y_j$. Then, define $R\coloneqq\bigcup_{j\in [N_Y]}R_j.$ It is clear that $R$ is a correspondence between $X$ and $Y$ because 
\[p_X\lc\bigcup_{j\in [N_Y]}R_j\rc=\bigcup_{j\in [N_Y]}p_X(R_j)=\bigcup_{j\in [N_Y]}X_{\Psi^{-1}(j)}=X\]
and
\[p_Y\lc\bigcup_{j\in [N_Y]}R_j\rc=\bigcup_{j\in [N_Y]}p_Y(R_j)=\bigcup_{j\in [N_Y]}Y_j=Y,\]
where $p_X:X\times Y\rightarrow X$ and $p_Y:X\times Y\rightarrow Y$ are the canonical coordinate projections. Given any $(x,y),(x\p,y\p)\in R$, suppose $(x,y)\in R_j$ and  $(x\p,y\p)\in R_{j'}$ for some  $j,j'\in [N_Y]$. Then, we verify that $|u_X(x,x\p)-u_Y(y,y\p)|\leq\eps$ in the following two cases:
\begin{enumerate}
\item if $j=j'$, then $|u_X(x,x\p)-u_Y(y,y\p)|\leq\mathrm{dis}(R_j)\leq\eps;$
\item if $j\neq j'$, then $x$ and $x'$ belong to different blocks of $X_{\mathfrak{o}\lc \delta_\eps(Y)\rc}$, and $y$ and $y'$ belong to different blocks of $Y_{\mathfrak{o}\lc \delta_0(Y)\rc}$. Then,
$u_X(x,x\p)\geq \delta_\eps(Y)=\diam(Y)-\eps$ and $u_Y(y,y\p)=\delta_0(Y)=\diam(Y)$. So, $u_X(x,x')\geq u_Y(y,y')-\eps$. By the assumption that $|\diam(X)-\diam(Y)|\leq \eps$, we have that $u_X(x,x')\leq\diam(X)\leq \diam(Y)+\eps=u_Y(y,y')+\eps$. Therefore, $|u_X(x,x\p)-u_Y(y,y\p)|\leq\eps.$
\end{enumerate}
Therefore, $\mathrm{dis}(R)\leq\eps$ and thus $R$ is an $\eps$-correspondence between $X$ and $Y$. This also proves Remark \ref{rmk:union-corr}.
\end{proof}

\subsection{Algorithms for computing $\dgh$ based on Theorem \ref{thm:ums-dgh}}
The main goal of this section is to develop an efficient algorithm for computing the exact value of $\dgh$ between ultrametric spaces. To achieve the goal, we first consider the following decision problem:

\begin{framed}
\noindent \textbf{Decision Problem \textbf{\textsf{GHDU-dec}} ($\dgh$ distance computation between finite ultrametric spaces)}

\noindent \textbf{Inputs:} Finite ultrametric spaces $X$ and $Y$, as well as $\eps\geq 0$.

\noindent \textbf{Question:} {Is there an $\eps$-correspondence between $X$ and $Y$}?
\end{framed}

\subsubsection{Strategy for solving \textbf{\textsf{GHDU-dec}}}\label{sec:strategy solving GHDU-dec}

\paragraph{Base cases for Problem \textbf{\textsf{GHDU-dec}}} Proposition \ref{prop:diam-dgh} shows how the Gromov-Hausdorff distance $\dgh$ interacts with the diameters of the input spaces. This theorem then implies that \textbf{\textsf{GHDU-dec}} is solved immediately in the following two base cases:
\begin{description}
    \item[Base Case 1:] If $|\diam(X)-\diam(Y)|>\eps$, then there exists no $\eps$-correspondence between $X$ and $Y$.
    \item[Base Case 2:] If $\max(\diam(X),\diam(Y))\leq\eps$, then every correspondence $R$ between $X$ and $Y$  is an $\eps$-correspondence.
\end{description}
Base Case 1 justifies our assumption that $|\diam(X)-\diam(Y)|\leq \eps$ in Theorem \ref{thm:ums-dgh} since otherwise we would be in one of the two base cases. Note that the situation when one of the two spaces is the one point space will automatically fall in either of the above two base cases.

\paragraph{Application of Theorem \ref{thm:ums-dgh}} Suppose that we are given two ultrametric spaces $X$ and $Y$ and $\eps\geq 0$ not falling in either of the two base cases mentioned above. This implies that one of $\diam(X)$ or $\diam(Y)$ must be strictly larger than $\eps$. 

Suppose $\diam(Y)>\eps$ (otherwise we swap the roles of $X$ and $Y$) and  apply the open partition operation to $X$ and $Y$ to obtain $X_{\mathfrak{o}\lc \delta_\eps(Y)\rc}:=\{X_i\}_{i\in [N_X]}$ and $Y_{\mathfrak{o}\lc \delta_0(Y)\rc}:=\{Y_j\}_{j\in [N_Y]}$. Here we use the same notation as in Theorem \ref{thm:ums-dgh} that for each $i\in [N_X]$, $X_i$ denotes an open equivalence class $[x_i]_{\mathfrak{o}(\delta_\eps(Y))}$ for some $x_i\in X$ and similarly for notation $Y_j$.

If there is no surjection from $[N_X]$ to $[N_Y]$, i.e., $N_X<N_Y$, then we conclude from Theorem \ref{thm:ums-dgh} that there is no $\eps$-correspondence between $X$ and $Y$. Otherwise, for each surjection $\Psi:[N_X]\twoheadrightarrow [N_Y]$ {and for each $j\in [N_Y]$}, we solve one instance of the decision problem \textbf{\textsf{GHDU-dec}} with input $(X_{\Psi^{-1}(j)},Y_j,\eps)$. If for some surjection $\Psi$, there exist $\eps$-correspondences $R_j$ between $X_{\Psi^{-1}(j)}$ and $Y_j$ for all $j\in [N_Y]$, then the union of all $R_j$s is an $\eps$-correspondence between $X$ and $Y$ ({cf. Remark \ref{rmk:union-corr}}). Otherwise, by Theorem \ref{thm:ums-dgh} again, there exists no $\eps$-correspondence between $X$ and $Y$. 

For each pair $(X_{\Psi^{-1}(j)},Y_j)$ as described above, it is easy to see that $\#X_{\Psi^{-1}(j)}<\#X$ and $\#Y_j<\#Y$. So, if we repeatedly apply the open partition operation as in Theorem \ref{thm:ums-dgh}, we will eventually reduce the problem to one of the two base cases.

\subsubsection{A recursive algorithm}

From the analysis above we identify a recursive algorithm $\mathbf{FindCorrRec}$ (Algorithm \ref{algo-dGH-rec}) which takes as input two ultrametric spaces $X$ and $Y$ and a parameter $\eps\geq 0$. If there exists an $\eps$-correspondence between $X$ and $Y$, then $\mathbf{FindCorrRec}(X,Y,\eps)$ returns such an $\eps$-correspondence. If there exists no $\eps$-correspondence, $\mathbf{FindCorrRec}(X,Y,\eps)$ returns 0.


\begin{algorithm}[htb]
\caption{$\mathbf{FindCorrRec} (X,Y,\eps)$}\label{algo-dGH-rec}
\begin{algorithmic}[1]
\STATE{BoolSwap $\gets$ FALSE}
 \IF {$\diam(X)>\diam(Y)$}
 \STATE{Swap $X$ and $Y$; BoolSwap $\gets$ TRUE}
 \ENDIF
\IF{$\max\left(\diam(X),\diam(Y)\right)\leq\eps$}
\STATE $R\gets\mathrm{ones}\lc{\#X , \#Y}\rc$
\ENDIF
    \IF {\emph{BoolSwap}}
    \STATE{Transpose $R$}
\RETURN $R$
    \ENDIF

 \IF{ $|\diam(X) - \diam(Y)| > \eps$}
 \RETURN 0
 \ENDIF
\STATE{$\{X_i\}_{i\in [N_X]}=\mathbf{OpenPartition}(X,\delta_\eps(Y))$}
\STATE{$\{Y_j\}_{j\in [N_Y]}=\mathbf{OpenPartition}(Y,\delta_0(Y))$}
 \FOR {\emph{Each surjection} $\Psi:[N_X]\twoheadrightarrow [N_Y]$}
 \FOR{$j\in [N_Y]$}
 \STATE{$R_j\gets \mathbf{FindCorrRec}\lc X_{\Psi^{-1}(j)},Y_j,\eps\rc$}
 \IF{$\left(\left( R_j \, != 0 \right) \forall j\right)$}
 \STATE $R\gets \bigcup_{j=1}^{N_Y} R_j$
     \IF {\emph{BoolSwap}} 
     \STATE{Transpose $R$}
     \ENDIF
     \RETURN $R$
     \ENDIF
     \ENDFOR
\ENDFOR
\RETURN 0
\end{algorithmic}
\end{algorithm}

\paragraph{Complexity analysis} To analyze the complexity of this recursive algorithm, we need to control the size of subproblems, i.e., the sizes of the blocks of the partitions produced by the open equivalence relations. The following structural condition on ultrametric spaces serves this purpose.

\begin{definition}[First $(\eps,\gamma)$-growth condition]\label{def:growth}
For $\eps\geq 0$, and $\gamma>1 $, we say that an ultrametric space $(X,u_X)$ satisfies the \emph{first $(\eps,\gamma)$-growth condition} (FGC) if for all $x\in X$, and $t\geq \eps$,
\[{\#[x]_{\mathfrak{c}\lc t\rc}}\leq \gamma\cdot{\#[x]_{\mathfrak{o}\lc t-\eps\rc}}.\]
Note that on the left-hand side of the inequality above we consider a `closed' equivalence class whereas on the right-hand side we consider an `open' equivalence class. We denote by $\mathcal{U}_1(\eps,\gamma)$ the collection of all finite ultrametric spaces satisfying the first $(\eps,\gamma)$-growth condition. See Figure \ref{fig:def9} for an illustration and Remark \ref{rmk:numberofbranch} for an interpretation.
\end{definition}

\nomenclature{$\mathcal{U}_1(\eps,\gamma)$}{The collection of all finite ultrametric spaces satisfying the first $(\eps,\gamma)$-growth condition; page \pageref{def:growth}.}

\begin{figure}[ht]
    \centering
    \includegraphics[width=0.5\textwidth]{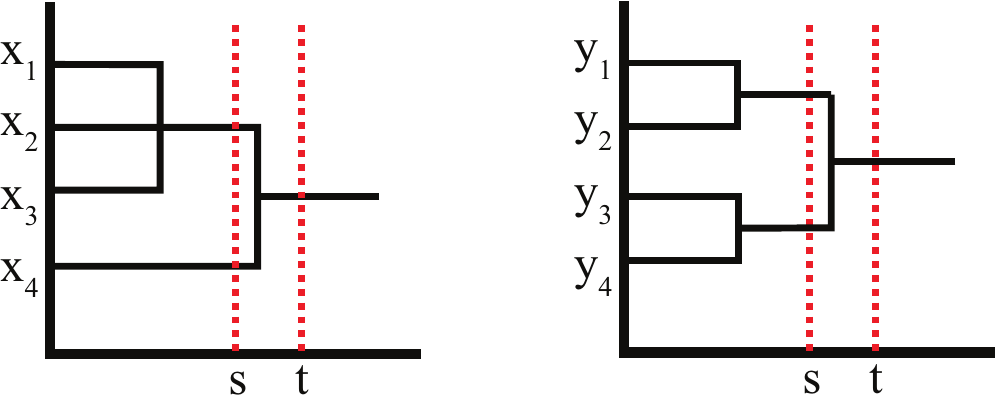}
    \caption{\textbf{Illustration of Definition \ref{def:growth}.} $X$ and $Y$ are two 4-point ultrametric spaces. Suppose $s=t-\eps$ for some $\eps>0$. It is easy to see that $Y\in\mathcal{U}(\eps,2)$. Since $2\,\#[x_4]_{\mathfrak{o}\lc s\rc}=2<4=\#[x_4]_{\mathfrak{o}\lc t\rc}$, it is easy to see that, in contrast, $X\notin\mathcal{U}(\eps,2)$. This example illustrates that the FGC prevents a given equivalence class in $X_{\mathfrak{o}(t)}$ from containing most of the points of $X$ and thus its dendrogram will tend to split `evenly'. }
    \label{fig:def9}
\end{figure}

\begin{remark}[Interpretation of the FGC]\label{rmk:numberofbranch}
The main idea behind the first $(\eps,\gamma)$-growth condition is that for each $t>0$ we want to have some degree of control over both the cardinalities of and the number of descendants of each $[x]_{\mathfrak{c}(t)}$ in $X_{\mathfrak{c}(t)}$, where we say that $[x']_{\mathfrak{o}\lc s\rc}$ is an (open) \emph{descendant} of $[x]_{\mathfrak{c}(t)}$, or conversely that $[x]_{\mathfrak{c}(t)}$ is a (closed) \emph{ancestor} of $[x']_{\mathfrak{o}\lc s\rc}$, if $[x']_{\mathfrak{o}\lc s\rc}\subseteq [x]_{\mathfrak{c}(t)}$.

More precisely, we write explicitly the $(t-\eps)$-open partition of $[x]_{\mathfrak{c}(t)}$ by $[x]_{\mathfrak{c}\lc t\rc}=\sqcup_{i=1}^N[x_i]_{\mathfrak{o}\lc t-\eps\rc}$ for some $x_i\in [x]_{\mathfrak{c}(t)}$, $i=1,\ldots,N$. 

First, we note that $[x]_{\mathfrak{c}\lc t\rc}=[x_i]_{\mathfrak{c}\lc t\rc}$ for each $i=1,\ldots,N$ and thus the FGC implies that 
\[{\#[x_i]_{\mathfrak{o}\lc t-\eps\rc}}\geq\frac{\#[x_i]_{\mathfrak{c}\lc t\rc}}{\gamma}= \frac{\#[x]_{\mathfrak{c}\lc t\rc}}{\gamma}.\]
This means that each descendant at scale $t-\eps$ of a given block $[x]_{\mathfrak{c}\lc t\rc}$ contains at least a fixed proportion $\frac{1}{\gamma}$ of the number of points in its ancestor $[x]_{\mathfrak{c}\lc t\rc}$. 

Moreover, we have
\[\#[x]_{\mathfrak{c}\lc t\rc}=\sum_{i=1}^N\#[x_i]_{\mathfrak{o}\lc t-\eps\rc}\geq\frac{N}{\gamma}\#[x]_{\mathfrak{c}\lc t\rc}. \]
Therefore $N\leq\gamma$, which implies that each $[x]_{\mathfrak{c}\lc t\rc}$ has at most $\gamma$ many descendants at scale $t-\eps$.
\end{remark}

By invoking the master theorem \cite{cormen2009introduction} we now prove the following theorem which provides an upper bound on the complexity of Algorithm \ref{algo-dGH-rec}. See Section \ref{sec:proof-rec} for its proof.

\begin{theorem}[Time complexity of Algorithm $\mathbf{FindCorrRec}$ (Algorithm \ref{algo-dGH-rec})]\label{thm:complexityrecdgh}
Fix some $\eps\geq 0$ and $\gamma\geq 2$. Then, for any $X,Y\in\mathcal{U}_1(\eps,\gamma)$, $\mathbf{FindCorrRec}(X,Y,\eps)$ (Algorithm \ref{algo-dGH-rec}) runs in time $O\left(n^{(\gamma+1)\log_{\mathrm{b}(\gamma)}\gamma}\right)$, where $n:=\max(\#X,\#Y)$ and $\mathrm{b}(\gamma):=\frac{\gamma^2}{\gamma^2-1}$.

\end{theorem}

{Under the FGC, our recursive algorithm $\mathbf{FindCorrRec}$ (Algorithm \ref{algo-dGH-rec})  exhibits time complexity $O\left(n^{(\gamma+1)\log_{\mathrm{b}(\gamma)}\gamma}\right)$. Since the exponent of $n$ depends on $\gamma$, this is only partially satisfactory. In other words, Algorithm \ref{algo-dGH-rec} is not yet fixed-parameter tractable, a notion which requires the exponent to be independent of the parameters involved. This motivates us to further examine and improve Algorithm \ref{algo-dGH-rec} in order to develop an FPT algorithm.}
Note that in the for-loop over surjections in Algorithm \ref{algo-dGH-rec}, for different surjections $\Psi_1,\Psi_2:[N_X]\twoheadrightarrow [N_Y]$, there could be some $j_0\in [N_Y]$ such that $\Psi^{-1}_1(j_0)=\Psi^{-1}_2(j_0)$. This {would result} in repetitive computations of $\mathbf{FindCorrRec}\lc X_{\Psi^{-1}_1(j_0)},Y_{j_0},\eps\rc$. With the goal of eliminating such repetitions, in the next section we devise a \emph{dynamic programming algorithm} {which eventually turns out to be FPT.}


\subsubsection{A dynamic programming algorithm}\label{sec:dp-main-text}
In this section, we introduce a dynamic programming algorithm $\mathbf{FindCorrDP}$  for solving the decision problem \textbf{\textsf{GHDU-dec}} for which we  provide pseudocode in Algorithm \ref{algo-dGH-dyn}. To proceed with the description of Algorithm \ref{algo-dGH-dyn}, we first introduce some notation. 

We let $V_X$ denote the set of all closed balls in $X$. For each closed ball $B\in  V_X $, let $\rho_\eps\lc B\rc \coloneqq\max(\diam\lc B\rc -2\eps,0)$ and write the $\rho_\eps\lc B\rc$-open partition of $B$ as: 
\[B_{\mathfrak{o}\lc \rho_\eps\lc B\rc\rc}:=\left\{[x_i]_{\mathfrak{o}\lc \rho_\eps(B)\rc}^{B}\right\}_{i=1}^{N_{B}}\]\label{rhos_eps}
where $x_i\in B$ for $i=1,\ldots,N_{B_X}$. For notational simplicity, we let $B_i\coloneqq[x_i]_{\mathfrak{o}\lc \rho_\eps(B)\rc}^{B}$ for each $i=1,\ldots,N_{B}$. It is obvious that each $[x]_{\mathfrak{o}\lc \rho_\eps(B)\rc}^{B}=[x]_{\mathfrak{o}\lc \rho_\eps(B)\rc}^{X}$ is actually a closed ball in $X$. Then, $B_i\in  V_X $ for all $i=1,\ldots,N_{B}$. Note that for any $I\subseteq\{1,\ldots,N_{B}\}$, $\diam(\bigcup_{i\in I}B_i)\leq \diam\lc B\rc .$ If the equality is achieved, we call $\bigcup_{i\in I}B_i$ an $\eps$-\emph{maximal union of closed balls} of $B$. Denote by $B_{(\eps)}$ the set of all $\eps$-{maximal unions} of closed balls \emph{in $B$}. Then, define a new set $ V_X^{(\eps)}\coloneqq \bigcup_{B\in  V_X }B_{(\eps)}$ by replacing each $B\in V_X $ with the set $B_{(\eps)}$. We use the notation $U^X$ to represent a generic element in $ V_X^{(\eps)}$. See Figure \ref{fig:LX} for an illustration of $ V_X^{(\eps)}$.\label{symbol:Beps}

\begin{figure}[ht]
    \centering
    \includegraphics[width=0.2\textwidth]{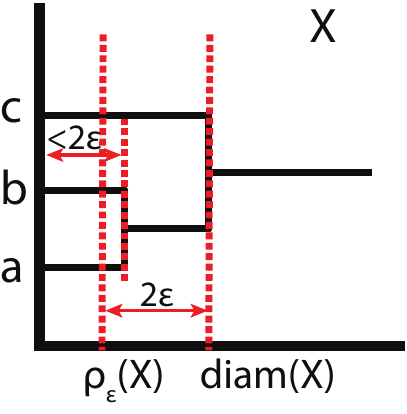}
    \caption{\textbf{Illustration of $ V_X^{(\eps)}$ from Section \ref{sec:dp-main-text}.} Note that $ V_X =\{ \{a\},\{b\},\{c\},\{a,b\},\{a,b,c\} \}$. For the ball $B\coloneqq\{a,b,c\}$, we have  $B_{(\eps)}=\{\{a,c\},\{b,c\},\{a,b,c\}\}$. For other balls $B'$ in $V_X$, we have $B'_{(\eps)}=\{B'\}$. For example, $\{a,b\}_{\mathfrak{o}\lc{\rho_\eps(\{a,b\})}\rc}=\{\{a\},\{b\}\}$, so $\{a,b\}_{(\eps)}=\{\{a,b\}\}$. Therefore, $ V_X^{(\eps)}=\{ \{a\},\{b\},\{c\},\{a,b\},\{a,c\},\{b,c\},\{a,b,c\}\}$.}
    \label{fig:LX}
\end{figure}

\begin{remark}\label{rmk:rho-eps-delta-eps}
The value $\rho_\eps(X)$ originates from Theorem \ref{thm:ums-dgh} as a lower bound for $\delta_\eps(Y)$: when $|\diam(X)-\diam(Y)|\leq\eps$ and $\diam(Y)>\eps$, we have that $\delta_\eps(Y)=\diam(Y)-\eps\geq\max(\diam(X)-2\eps,0)=\rho_\eps(X)$. This inequality results in the following favorable property of $ V_X^{(\eps)}$: for any $U^X\in  V_X^{(\eps)}$ and $B^Y\in V_Y $, if $|\diam\lc U^X\rc -\diam\lc B^Y\rc |\leq \eps$ and $\diam\lc B^Y\rc >\eps$, then each block of the open partition $U^X_{\mathfrak{o}\lc \delta_\eps\lc B^Y\rc\rc}$ belongs to $ V_X^{(\eps)}$. See Appendix \ref{sec:proof of remark rho} for a proof of this property. 
\end{remark}

Fix an input triple $(X,Y,\eps)$. It is clear that the pair $(X,Y)$ belongs to $ V_X^{(\eps)}\times  V_Y $. We sort $V_X^{(\eps)}$ and $V_Y$ according to ascending diameter values and denote by $\mathrm{LX}^{(\eps)}$ and $\mathrm{LY}$ the respective sorted arrays (with details provided in Appendix \ref{sec:implementation detail}). In particular, we require that $X$ and $Y$ are at the end of lists $\mathrm{LX}^{(\eps)}$ and $\mathrm{LY}$, respectively. We devise our DP algorithm $\mathbf{FindCorrDP}$ (Algorithm \ref{algo-dGH-dyn}) so that it maintains a binary variable $\mathrm{DYN}\lc U^X,B^Y\rc $ for each pair $\lc U^X,B^Y\rc \in\mathrm{ LX_{(\eps)}\times LY}$, such that $\mathrm{DYN}\lc U^X,B^Y\rc =1$ if there exists an $\eps$-correspondence between $U^X$ and $B^Y$, and $\mathrm{DYN}\lc U^X,B^Y\rc =0$, otherwise. Now, we elaborate the main idea behind Algorithm $\mathbf{FindCorrDP}$ (Algorithm \ref{algo-dGH-dyn}).

Algorithm \ref{algo-dGH-dyn} starts by looping over all $B^Y\mathrm{\in LY}$. Inside the loop, it computes $\mathrm{DYN}\lc U^X,B^Y\rc $ by looping over all $U^X\mathrm{\in LX_{(\eps)}}$. Most pairs ${\lc U^X,B^Y\rc }$ fall in the base cases and $\mathrm{DYN}\lc U^X,B^Y\rc $ is determined by comparing diameters. For non base cases, we have the following two situations:
\begin{enumerate}\label{two cases}
    \item If $\diam\lc B^Y\rc >\eps$, we determine $\mathrm{DYN}\lc U^X,B^Y\rc $ by (1) computing the open partition of $U^X$ and $B^Y$, respectively, to obtain $U^X_{\mathfrak{o}\lc \delta_\eps\lc B^Y\rc\rc}=\lb U^X_i\rb_{i\in [N_{U_X}]}$ and $B^Y_{\mathfrak{o}\lc \delta_0\lc B^Y\rc\rc}=\lb B^Y_j\rb_{j\in [N_{B_Y}]}$ and (2) by exploiting the precomputed values  
    \[\lb\mathrm{DYN}\lc U^X_{\Psi^{-1}(j)},B^Y_j\rc \rb_{j\in [N_{B_Y}],\text{surjection }\Psi:[N_{U_X}]\rightarrow[N_{B_Y}] }\footnote{Here $U^X_{\Psi^{-1}(j)}\coloneqq\cup_{i\in\Psi^{-1}(j)}U^X_i$}\]
    via the strategy discussed in Section \ref{sec:strategy solving GHDU-dec}. That $\lc U^X_{\Psi^{-1}(j)},B^Y_j\rc \in  V_X^{(\eps)}\times  V_Y $ follows from Remark \ref{rmk:rho-eps-delta-eps} and that the values $\mathrm{DYN}\lc U^X_{\Psi^{-1}(j)},B^Y_j\rc$ for all $j\in[N_{B_Y}]$ are pre-computed follows from the fact that $\diam(B^Y_j)<\diam(B^Y)$ and $\mathrm{LY}$ is ordered according to increasing diameter values.
    
    \item If $\diam\lc B^Y\rc \leq\eps$, {we determine $\mathrm{DYN}\lc U^X,B^Y\rc $ directly by applying Algorithm \ref{algo-dGH-small} (which arises from Proposition \ref{prop:smallcorr} below).}
\end{enumerate}

\begin{remark}[Interpretation of situation 2]
In order to reduce redundant computations, Algorithm $\mathbf{FindCorrDP}$ (Algorithm \ref{algo-dGH-dyn}) only inspects pairs in $ V_X^{(\eps)}\times  V_Y $ instead of the much larger symmetric set  $ V_X^{(\eps)}\times  V_Y^{(\eps)}$.
Due to the asymmetry of $V_X^{(\eps)}\times V_Y$, the exceptional case in item 2 above may arise. {This case is dealt with in the recursive algorithm $\mathbf{FindCorrRec}$ (Algorithm \ref{algo-dGH-rec}) by swapping the roles of $U^X$ and $B^Y$.} However, this swapping technique is not feasible for $\mathbf{FindCorrDP}$. We now further elaborate on this point. Suppose we replace line 10 in Algorithm \ref{algo-dGH-dyn} with a swapping between $U^X$ and $B^Y$. Subsequently, we obtain open partitions of $B^Y$ and $U^X$ as follows:
\[B^Y_{\mathfrak{o}\lc \delta_\eps\lc U^X\rc\rc}=\lb B^Y_i\rb_{i\in [N_{B_Y}]}\text{ and }U^X_{\mathfrak{o}\lc \delta_0\lc U^X\rc\rc}=\lb U^X_j\rb_{j\in [N_{U_X}]}.\]
Then, for each surjection $\Psi:[N_{B_Y}]\twoheadrightarrow [N_{U_X}]$, we need to inspect values of $\mathrm{DYN}\left(U^X_j,B^Y_{\Psi^{-1}(j)}\right)$. Being a union of closed balls in $B^Y$, $B^Y_{\Psi^{-1}(j)}$ does not necessarily belong to $ V_Y $, the set of closed balls in $Y$. This implies that the value $\mathrm{DYN}\left(U^X_j,B^Y_{\Psi^{-1}(j)}\right)$ does not necessarily exist for which Algorithm \ref{algo-dGH-dyn} may fail to continue.
\end{remark}

\begin{proposition}\label{prop:smallcorr}
Let finite ultrametric spaces $X$ and $Y$ and  $\eps\geq 0$ be such that the following two conditions hold:
\begin{enumerate}
    \item $\diam(X)>\eps$, and
    \item $\diam(Y)\leq \eps$
\end{enumerate} Then, there exists an $\eps$-correspondence between $X$ and $Y$ if and only if there exists an \emph{injective} map $\varphi:X\ct{\eps}\rightarrow Y$ such that $\dis(\varphi)\leq \eps$.
\end{proposition}
See Appendix \ref{sec: small-algo-prop} for a proof.
\begin{algorithm}[htb]
\caption{$\mathbf{FindCorrSmall}(X,Y,\eps)$}\label{algo-dGH-small}
\begin{algorithmic}[1]
 \STATE{\textbf{Assert} $\diam(X) > \eps$ \textbf{and} $\diam(Y) \leq \eps$}
\IF{ $|\diam(X) - \diam(Y)| > \eps$}
\RETURN 0
\ENDIF
\STATE{$X\ct{\eps}=\mathbf{ClosedQuotient}(X,\eps)$}
\FOR {\emph{Each injective map} $\Phi:X\ct{\eps}\rightarrow Y $}
\STATE Compute $\dis(\Phi)$
\IF {$\dis(\Phi)\leq\eps$}
\RETURN $1$
\ENDIF
\ENDFOR
\RETURN 0
\end{algorithmic}
\end{algorithm}

\begin{algorithm}[htb]
\caption{$\mathbf{FindCorrDP}(X,Y,\eps)$}\label{algo-dGH-dyn}
\begin{algorithmic}[1]
 \STATE{Build $\mathrm{LX}^{(\eps)}$ and $\mathrm{LY}$}
 \STATE $\mathrm{DYN = zeros(\# \mathrm{LX}^{(\eps)},\# \mathrm{LY})}$
\FOR{$B^Y \in  \mathrm{LY} $}{
    \FOR{$k=1$ to $\#\mathrm{LX}^{(\eps)}$ }{
            \STATE $U^X=\mathrm{LX}^{(\eps)}(k)$
            \IF{$\left|\diam\lc U^X\rc-\diam\lc B^Y\rc\right|>\eps$}
            \STATE{$\mathrm{DYN}\lc U^X,B^Y\rc =0$}
        \ELSIF{$\max\lc\diam\lc U^X\rc ,\diam\lc B^Y\rc \rc\leq\eps$}
        \STATE{$\mathrm{DYN}\lc U^X,B^Y\rc =1$}
         \ELSIF{$\diam\lc U^X\rc >\eps$ \textbf{and} $\diam\lc B^Y\rc \leq\eps$}
         \STATE{$\mathrm{DYN}\lc U^X,B^Y\rc =\mathbf{FindCorrSmall}(U^X,B^Y,\eps)$}
            \ELSE
            \STATE{$\lb U^X_i\rb_{i\in [N_{U_X}]}=\mathbf{OpenPartition}\lc U^X,\delta_\eps(B^Y)\rc$}
            \STATE{$\lb B^Y_j\rb_{j\in [N_{B_Y}]}=\mathbf{OpenPartition}\lc B^Y,\delta_0(B^Y)\rc$}
            
            \FOR{\emph{Each surjection} $\Psi : [N_{U_X}]\twoheadrightarrow  [N_{B_Y}]$}
             {
             \IF{$\mathrm{DYN}\lc U^X_{\Psi^{-1}(j)},B^Y\rc=1,\, \forall j=1,\ldots,M$}
             \STATE $\mathrm{DYN}\lc U^X,B^Y\rc  = 1$
               \STATE     Continue in line 4
               \ENDIF
            }
            \ENDFOR
            
            \ENDIF
    }
    \ENDFOR
    }
    \ENDFOR
\RETURN $\mathrm{DYN(END,END)}$
\end{algorithmic}
\end{algorithm}

Eventually, Algorithm $\mathbf{FindCorrDP}$ (Algorithm \ref{algo-dGH-dyn}) will compute $\mathrm{DYN}(X,Y)$ through a bottom-up approach and thus solve the decision problem \textbf{\textsf{GHDU-dec}} with the given input triple $(X,Y,\eps)$. The correctness of Algorithm \ref{algo-dGH-dyn} is stated in the following theorem; see Appendix \ref{sec:corr-dp} for its proof. Note that, the given pseudocode of Algorithm \ref{algo-dGH-dyn} only determines the existence of $\eps$-correspondence without actually constructing a correspondence. However, it is clear that one can inspect the $\mathrm{DYN}$ matrix to produce an $\eps$-correspondence whenever it exists.

\begin{theorem}[Correctness of Algorithm $\mathbf{FindCorrDP}$ (Algorithm \ref{algo-dGH-dyn})]\label{thm:corr-dp-algo}
There exists an $\eps$-correspondence between $X$ and $Y$ if and only if $\mathbf{FindCorrDP}(X,Y,\eps)=1$.
\end{theorem}

\paragraph{Complexity analysis}

To analyze the complexity of Algorithm $\mathbf{FindCorrDP}$ (Algorithm \ref{algo-dGH-dyn}), we consider the following growth condition in a similar spirit to the FGC:

\begin{definition}[Second $(\eps,\gamma)$-growth condition]\label{def:growth2}
For $\eps\geq 0$, and $\gamma\in\mathbb{N}$, we say that an ultrametric space $(X,u_X)$ satisfies the \emph{second $(\eps,\gamma)$-growth condition} (SGC) if for all $x\in X$, and $t\geq 2\eps$,
\begin{equation*}
    \#\left\{{[x']_{\mathfrak{o}\lc t-2\eps\rc}}:\,x'\in [x]_{\mathfrak{c}\lc t\rc}\right\}\leq \gamma.
\end{equation*}
\end{definition}

We denote by $\mathcal{U}_2(\eps,\gamma)$ the collection of all finite ultrametric spaces satisfying the second $(\eps,\gamma)$-growth condition. Note that for any $0\leq \eps'<\eps$, $\mathcal{U}_2(\eps,\gamma)\subseteq \mathcal{U}_2(\eps',\gamma)$.

\nomenclature{$\mathcal{U}_2(\eps,\gamma)$}{The collection of all finite ultrametric spaces satisfying the second $(\eps,\gamma)$-growth condition; page \pageref{def:growth2}.}

\begin{remark}[Relation with the notion of degree bound from \cite{touli2018fpt}]\label{rmk:degree bound in main text}
If we let
\[\gamma_\eps(X)\coloneqq\sup_{x\in X,t\geq 0}\#\left\{{[x']_{\mathfrak{o}\lc t-2\eps\rc}}:\,x'\in [x]_{\mathfrak{c}\lc t\rc}\right\},\]
then for any $\gamma\geq\gamma_\eps(X)$, $X\in\mathcal{U}_2(\eps,\gamma)$. The information captured by $\gamma_\eps$ is in a similar spirit to the concept called degree bound of merge trees as considered in \cite{touli2018fpt}: the $\eps$-degree bound $\tau_\eps(M_X)$ of a merge tree $M_X$ is the largest sum of degrees of all tree vertices inside any closed $\eps$ balls in $M_X$\footnote{In \cite{touli2018fpt}, the degree bound is actually defined for two merge trees: for two merge trees $M_X$ and $M_Y$, the number $\tau_\eps(M_X,M_Y)\coloneqq\max(\tau_\eps(M_X),\tau_\eps(M_Y))$is called the $\eps$-degree bound of $(M_X,M_Y)$.}. See Appendix \ref{sec:tree structure of dend} for a detailed comparison between $\gamma_\eps(X)$ and $\tau_\eps$.
\end{remark}

\begin{remark}[Interpretation of the SGC and its relation with the FGC]\label{rmk:2nd-growth-int}
The second $(\eps,\gamma)$-growth condition is equivalent to saying for any $x\in X$ and $t>2\eps$, the number of descendants of $[x]_{\mathfrak{c}(t)}$ at level $t-2\eps$ is bounded above by $\gamma$. Note that if $X\in\mathcal{U}_1(\eps,\gamma)$, then for any $t>\eps$, the number of descendants of any class $[x]_{\mathfrak{c}(t)}$ at $t-\eps$ is bounded above by $\gamma$ (cf. Remark \ref{rmk:numberofbranch}). This implies that $X\in\mathcal{U}_2\lc\frac{\eps}{2},\gamma\rc$.  In other words, $\mathcal{U}_1(\eps,\gamma)\subseteq \mathcal{U}_2\lc\frac{\eps}{2},\gamma\rc$, which indicates that the second growth condition is \textbf{less} rigid than the first growth condition.
\end{remark}

\begin{remark}[Relation between the SGC and the doubling constant]\label{rmk:sgc-doubling}
Recall that given $K>0$, a metric space $(X,d_X)$ is said to be $K$-doubling if for each $r>0$, a closed ball with radius $r$ can be covered by at most $K$ closed balls with radius $\frac{r}{2}$. The SGC is related to the doubling constant as follows: (1) if a finite ultrametric space $X\in\mathcal{U}_2(\eps,\gamma)$ for some $\eps>0$ and $\gamma\geq 1$, then $X$ is $K_{\gamma,\eps}$-doubling for $K_{\gamma,\eps}\coloneqq \gamma^{\lfloor\frac{\diam(X)}{4\eps}\rfloor+1}$; (2) conversely, a $K$-doubling ultrametric space satisfies the second $(0,K)$-growth condition. See Appendix \ref{sec:rmk-sgc-proof} for the proof of the fact. See Lemma \ref{lm:ultrametricty sgc} for a generalization of the latter fact in the case of finite metric spaces.
\end{remark}

{Under the SGC, Algorithm $\mathbf{FindCorrDP}$ (Algorithm \ref{algo-dGH-dyn}) runs in polynomial time and moreover, Algorithm $\mathbf{FindCorrDP}$ is FPT with respect to parameters in the SGC.}
\begin{theorem}[Time complexity of Algorithm $\mathbf{FindCorrDP}$ (Algorithm \ref{algo-dGH-dyn})]\label{thm:com-dyn-alg}
Fix some $\eps\geq 0$ and $\gamma\geq 1$. Then, for any $X,Y\in\mathcal{U}_2(\eps,\gamma)$, Algorithm $\mathbf{FindCorrDP}(X,Y,\eps)$ (Algorithm \ref{algo-dGH-dyn}) runs in time $O\lc n^2\log(n)2^{\gamma}\gamma^{\gamma+2}\rc$, where $n:=\max(\#X,\#Y)$.
\end{theorem}
See Appendix \ref{sec:com-dp-proof} for its proof.
\subsubsection{Computing the exact value of $\dgh$}
Given two finite ultrametric spaces $X$ and $Y$, we compute the exact value of $\dgh(X,Y)$ in the following way. Define 
\[ \mathcal{E}(X,Y)\coloneqq\{|u_X(x,x')-u_Y(y,y')|:\,\forall x,x'\in X\text{ and }\forall y,y'\in Y\}.\]
Then, for any correspondence $R$ between $X$ and $Y$, we have $\dis(R)\in\mathcal{E}(X,Y)$ by finiteness of $X$ and $Y$ and by Equation (\ref{eq:dist}). Therefore, in order to compute $\dgh(X,Y)$, we first sort the elements in $\mathcal{E}(X,Y)$ in ascending order as $\eps_0<\eps_1<\cdots<\eps_M$. If $i$ is the smallest integer such that $\mathbf{FindCorrDP}(X,Y,{\eps_i})=1 $, then $\dgh(X,Y)=\frac{\eps_i}{2}$. We summarize the process in Algorithm \ref{algo-dGH-opt} and analyze its complexity  in Theorem \ref{thm:dGH-complex}.

\begin{theorem}[Time complexity of Algorithm $\mathbf{dGH}$ (Algorithm \ref{algo-dGH-opt})] \label{thm:dGH-complex}
Fix some $\eps\geq 0$ and $\gamma\geq 1$. Let $X,Y\in\mathcal{U}_2(\eps,\gamma)$ and assume that $\eps\geq 2\,\dgh(X,Y)$. Then, the algorithm $\mathbf{dGH}$ (Algorithm \ref{algo-dGH-opt}) with input $(X,Y)$ runs in time $O\lc n^4\log(n)2^\gamma\gamma^{\gamma+2}\rc$, where $n=\max(\#X,\#Y)$.
\end{theorem}

\begin{remark} 
Though the complexity in Theorem \ref{thm:dGH-complex} depends on inherent structures of input spaces, it never means that we to figure out parameters $\eps$ and $\gamma$ beforehand in order to apply our algorithm.
\end{remark}

\begin{proof}[Proof of Theorem \ref{thm:dGH-complex}]
By Proposition \ref{prop:basic property ultrametric}, we have that $\#\mathcal{E}(X,Y)=O(n^2)$. Then, sorting $\mathcal{E}(X,Y)$ takes time $O(n^2\log(n^2))=O(n^2\log(n))$ in average. For each $\eps_i\in\mathcal{E}(X,Y)$ such that $\eps_i\leq\eps$, we need to invoke once Algorithm $\mathbf{FindCorrDP}$ (Algorithm \ref{algo-dGH-dyn}) with inputs $(X,Y,\eps_i)$. For all such $\eps_i$s, $X,Y\in\mathcal{U}_2(\eps_i,\gamma)$ and thus Algorithm $\mathbf{FindCorrDP}$ with inputs $(X,Y,\eps_i)$ runs in time $O\lc n^2\log(n)2^{\gamma}\gamma^{\gamma+2}\rc$ (cf. Theorem \ref{thm:com-dyn-alg}). Therefore, the total time complexity of Algorithm $\mathbf{dGH}$ is bounded by 
\[O(n^2)\times O\lc n^2\log(n)2^{\gamma}\gamma^{\gamma+2}\rc= O\lc n^4\log(n)2^\gamma\gamma^{\gamma+2}\rc.\]
\end{proof}

\begin{algorithm}[htb]
\caption{$\mathbf{dGH}(X,Y)$}\label{algo-dGH-opt}
\begin{algorithmic}[1]
 \STATE{$\mathcal{E}$ $\gets$ \textbf{sort}($\mathcal{E}(X,Y)$, `ascend')}
\FOR{i=1 to $\#\mathcal{E}$}
\IF{$\mathbf{FindCorrDP}(X,Y,\mathcal{E}(\mathrm{i}))$}
\RETURN $\frac{\mathcal{E}(\mathrm{i})}{2}$
\ENDIF
\ENDFOR
\end{algorithmic}
\end{algorithm}

\begin{remark}[Comparison to \cite{touli2018fpt}] \label{rem:tw}
Whereas methods from \cite{touli2018fpt} can be adapted to obtain a 2-approximation of $\dgh$ between two finite ultrametric spaces, our algorithm $\mathbf{dGH}$ (Algorithm \ref{algo-dGH-opt}) can obtain the exact value \emph{in the same time complexity}. We now elaborate upon this statement.

As illustrated in Remark \ref{rmk:merge tree}, each finite ultrametric space naturally maps into a merge tree. In this way, we define the $\eps$-degree bound of an ultrametric space as the $\eps$-degree bound of its corresponding merge tree. By Remark \ref{rmk:detail comparison with degree bound}, if any ultrametric space $X$ has $\eps$-degree bound $\tau_\eps$, it automatically satisfies the second $(\eps,\tau_\eps)$-growth condition. 

Now consider the case where two merge trees $M_X$ and $M_Y$ arise from finite ultrametric spaces $X$ and $Y$ such that $\dgh(X,Y) = \frac{\eps}{2}$. In this case, if $d_\mathrm{I}$ denotes the interleaving distance between merge trees of \cite{morozov2013interleaving},  by \cite[Remark 6.3 and Corollary 6.13]{memoli2019gromov}, then \begin{equation}\label{eq:comp}
    \frac{1}{2}d_\mathrm{I}(M_X,M_Y)\leq\dgh(X,Y)\leq {d_\mathrm{I}(M_X,M_Y)}.
    \end{equation}
    Let $\tau\coloneqq\tau_{\eps}(M_X,M_Y)$ denote the $\eps$-degree bound of $(M_X,M_Y)$, then by above arguments we have that $X,Y\in\mathcal{U}_2(\eps,\tau)$. Let $\delta\coloneqq d_\mathrm{I}(M_X,M_Y)$. Since $\delta=d_\mathrm{I}(M_X,M_Y)\leq\eps$, by monotonicity of the degree bound, the $\delta$-degree bound $\tau_\delta\coloneqq\tau_\delta(M_X,M_Y)$ of $(M_X,M_Y)$ satisfies that $\tau_\delta\leq\tau$. Then, it is shown in \cite{touli2018fpt} that one can compute $\delta=d_\mathrm{I}(M_X,M_Y)$ in time
    \[O\lc n^4\log(n)2^{\tau_\delta}\tau_\delta^{\tau_\delta+2}\rc=O\lc n^4\log(n)2^\tau\tau^{\tau+2}\rc,\]
    which by Equation (\ref{eq:comp}) is a 2-approximation of $\dgh(X,Y)$. Note that, in contrast, by Theorem \ref{thm:dGH-complex}, with the same time complexity, our algorithm can compute the \textbf{exact} value of $\dgh(X,Y)$. 
\end{remark}

\begin{remark}[Improved time complexity for computing $\dgh$]\label{rmk:improved complexity}
Following  essentially the same strategy used for proving \cite[Theorem 5]{touli2018fpt}, we can improve the time complexity for computing $\dgh$ to $O\lc n^2\log^3(n)2^{2\gamma}(2\gamma)^{2\gamma+2}\rc$ under the same assumptions in Theorem \ref{thm:dGH-complex}. We provide details in Appendix \ref{sec:proof of rmk improved complexity}.
\end{remark}

\subsubsection{Additive approximation of $\dgh$ between arbitrary finite metric spaces}\label{sec:additive approx}
For any finite metric space $(X,d_X)$, we introduce the following notion of \emph{absolute ultrametricity} which quantifies how far $X$ is being an ultrametric space:
\begin{definition}[Absolute ultrametricity]
For a finite metric space $(X,d_X)$, we define the {absolute ultrametricity} of $X$ by
\[\absult(X)=\inf_{u}\|d_X-u\|_\infty\]
where the infimum is over all possible ultrametrics on $X$.
\end{definition}

Note that $X\in\ufin$ iff $\absult(X)=0$.

The notion of absolute ultrametricity is related to a more involved notion simply called \emph{ultrametricity}; see \cite{chowdhury2016improved} for a detailed study.

One natural ultrametric on $X$ which can be used to approximate $d_X$ is the so-called \emph{single-linkage ultrametric} (or maximal subdominant) $u_X^*$ \cite{carlsson2010characterization}. The ultrametric $u_X^*$ is defined as follows:
\[u_X^*(x,x')\coloneqq\inf_{x=x_1,x_2,\ldots,x_n=x'}\max_{i=1,\ldots,n-1}d_X(x_i,x_{i+1}),\]
where the infimum is taken over all finite chains $x_1,x_2,\ldots,x_n\in X$ such that $x_1=x$ and $x_n=x'$. It turns out that the single-linkage ultrametric is a fairly good ultrametric approximation of $d_X$:

\begin{proposition}[{\cite[Theorem 3.3]{kvrivanek1988complexity}}]\label{prop:ultrametricity double}
For any finite metric space $(X,d_X)$, 
\[\|d_X-u_X^*\|_\infty=2\,\absult(X).\]
\end{proposition}
\begin{remark}\label{rmk:property of single linkage ultrametric}
\begin{enumerate}
    \item Proposition \ref{prop:ultrametricity double} implies that if $(X,u_X)$ is already an ultrametric space, then $u_X^*=u_X$.
    \item The time complexity of computing $u_X^*$ from $d_X$ is bounded above by $O(n^2)$ where $n\coloneqq\#X$, cf. \cite{mullner2011modern}.
\end{enumerate}
\end{remark}

For a metric space $(X,d_X)$, its \emph{separation} is defined as $\mathrm{sep}(X):=\inf\{d_X(x,x')|\,x\neq x'\}.$ The following result illustrates how to transfer a doubling condition on a given metric space $(X,d_X)$ to a SGC on its corresponding single-linkage ultrametric space $(X,u_X^*)$. \label{separation}

\nomenclature{$\mathrm{sep}(X)$}{The separation of $X$; page \pageref{separation}.}

\begin{lemma}[Transfer from the doubling property on $d_X$ to the SGC on $u_X^\ast$]\label{lm:ultrametricty sgc}
Let $(X,d_X)$ be a finite metric space. Let $K$ be a doubling constant for $X$ and let $\delta\coloneqq2\,\absult(X)$. Let $s:=\mathrm{sep}(X)$ denote the separation of $X$. Then, for any $\eps\geq 0$, we have that $(X,u_X^*)$ satisfies the second $\lc\eps,\max\lc K,K^{\log_2\lc\frac{2\delta+4\eps}{s}\rc+1}\rc\rc$-growth condition.
\end{lemma}

The proof is postponed to Appendix \ref{sec:proof of approximation}. By Remark \ref{rmk:property of single linkage ultrametric}, when $(X,u_X)$ is itself an ultrametric space, $u_X^*=u_X$ and thus $X\in\mathcal{U}_2\lc\eps,\max\lc K,K^{\log_2\lc\frac{4\eps}{s}\rc+1}\rc\rc$. In particular, if $\eps=0$, we have that $X\in\mathcal{U}_2\lc0,K\rc$, which coincides with the second claim of Remark \ref{rmk:sgc-doubling}.

Now, let $\ms^\mathrm{fin}$ denote the collection of all finite metric spaces. We denote by $\mathfrak{H}:\ms^\mathrm{fin}\rightarrow\mathcal{U}^\mathrm{fin}$ the \emph{single-linkage map} sending $(X,d_X)\in \ms^\mathrm{fin}$ to $\mathfrak{H}(X)\coloneqq (X,u_X^*)\in \ums^\mathrm{fin}$. 

\begin{proposition}\label{prop:2delta approximation}
Let $X,Y\in \ms^\mathrm{fin}$ and let $\delta\coloneqq\max(\absult(X),\absult(Y))$. Then,
\[\dgh(\mathfrak{H}(X),\mathfrak{H}(Y))\leq\dgh(X,Y)\leq \dgh(\mathfrak{H}(X),\mathfrak{H}(Y))+2\delta.\]
\end{proposition}
\begin{proof}
The leftmost inequality follows directly from the stability result of the single-linkage map, cf. \cite[Proposition 2]{carlsson2010characterization}. For the rightmost inequality, we first have the following obvious observation:
\begin{claim}
Given a finite set $X$ and two metrics $d_1,d_2$ on the set $X$, we have that
\[\dgh((X,d_1),(X,d_2))\leq \|d_1-d_2\|_\infty.\]
\end{claim}
Then, we have that 
\begin{align*}
    \dgh((X,d_X),(Y,d_Y))&\leq \dgh((X,d_X),(X,u_X^*))+\dgh((X,u_X^*),(Y,u_Y^*))+\dgh((Y,d_Y),(Y,u_Y^*))\\
    &\leq \delta+\dgh((X,u_X^*),(Y,u_Y^*))+\delta\leq \dgh((X,u_X^*),(Y,u_Y^*))+2\delta.
\end{align*}
This implies that $ \dgh(X,Y)\leq \dgh(\mathfrak{H}(X),\mathfrak{H}(Y))+2\delta.$
\end{proof}

This proposition indicates that $\dgh(\mathfrak{H}(X),\mathfrak{H}(Y))$ is a $2\delta$-additive approximation to $\dgh(X,Y)$. Applying Theorem \ref{thm:dGH-complex} to computing $\dgh(\mathfrak{H}(X),\mathfrak{H}(Y))$, this immediately gives rise to the following time complexity result for computing an additive approximation of $\dgh(X,Y)$.

\begin{corollary}[Computing an additive approximation to $\dgh(X,Y)$]\label{coro:complexity approximation}
Let $X$ and $Y$ be two $K$-doubling finite metric spaces for some $K>0$. Let $\delta\coloneqq2\max(\absult(X),\absult(Y))$. Let $s\coloneqq\min(\mathrm{sep}(X),\mathrm{sep}(Y))$. Then, the $2\delta$-additive approximation $\dgh(\mathfrak{H}(X),\mathfrak{H}(Y))$ of $\dgh(X,Y)=:\frac{\eps}{2}$ can be computed in time $O\lc n^4\log(n)2^\gamma\gamma^{\gamma+2}\rc$, where $n=\max(\#X,\#Y)$ and $\gamma\coloneqq\max\lc K,K^{\log_2\lc\frac{2\delta+4\eps}{s}\rc+1}\rc$.
\end{corollary}
\begin{proof}
By Remark \ref{rmk:property of single linkage ultrametric}, the time complexity of computing $u_X^*$ from $d_X$ and computing $u_Y^*$ from $d_Y$ is bounded above by $O(n^2)$ where $n\coloneqq\max(\#X,\#Y)$.

Note that $\eps=2\dgh(X,Y)$. Then, by Lemma \ref{lm:ultrametricty sgc}, $\mathfrak{H}(X)$ and $\mathfrak{H}(Y)$ both satisfy the second $\lc\eps,\gamma\rc$-growth condition, where $\gamma\coloneqq\max\lc K,K^{\log_2\lc\frac{2\delta+4\eps}{s}\rc+1}\rc$. 
Let $\eps_u\coloneqq 2\,\dgh(\mathfrak{H}(X),\mathfrak{H}(Y))$. Note that $\eps_u\leq \eps$ by Proposition \ref{prop:2delta approximation}. Then, $\mathfrak{H}(X)$ and $\mathfrak{H}(Y)$ both satisfy the second $\lc\eps_u,\gamma\rc$-growth condition.
Therefore, by Theorem \ref{thm:dGH-complex}, $\frac{\eps_u}{2}=\,\dgh(\mathfrak{H}(X),\mathfrak{H}(Y))$, which is a $2\delta$ additive approximation of $\frac{\eps}{2}=\dgh(X,Y)$, can be computed in time $O\lc n^4\log(n)2^{\gamma}\gamma^{\gamma+2}\rc$, where $\gamma\coloneqq\max\lc K,K^{\log_2\lc\frac{2\delta+4\eps}{s}\rc+1}\rc$.

Therefore, the total time complexity is bounded by 
$$O(n^2)+O\lc n^4\log(n)2^{\gamma}\gamma^{\gamma+2}\rc=O\lc n^4\log(n)2^{\gamma}\gamma^{\gamma+2}\rc.$$
\end{proof}

\section{Discussion}

It is well known that computing $\dgh$ between finite metric spaces leads to NP-hard problems. This hardness result holds even in the context of ultrametric spaces, which are highly structured metric spaces appearing in many practical applications.

In contrast to the hardness results for $\dgh$, by exploiting the ultrametric structure of the input spaces we first devised a polynomial time algorithm for computing $\ugh$, an ultrametric variant of $\dgh$, on the collection of all finite ultrametric spaces. Indeed, as a consequence of being more rigid than $\dgh$, we proved that $\ugh$ can be computed in  $O(n\,\log(n))$ time via Algorithm \ref{algo-uGH}, which we also extended to the case of ultra-dissimilarity spaces.

From a different angle, but also with the goal of taming the NP-hardness associated to computing  $\dgh$ on the collection of all finite ultrametric spaces, as a second contribution, we first devised a recursive algorithm (Algorithm \ref{algo-dGH-rec}) and then based on this, a dynamic programming  FPT-algorithm (Algorithm \ref{algo-dGH-dyn}) for computing $\dgh$. 

We  provide implementations of Algorithm \ref{algo-uGH} for $\ugh$ and Algorithm \ref{algo-dGH-rec} for $\dgh$ in our github repository \cite{github2019}. 

We leave for future work finding extensions of Algorithms \ref{algo-dGH-rec} and \ref{algo-dGH-dyn} to the case of ultra-dissimilarity spaces and eventually general tree metric spaces. 

\bibliographystyle{plain}
\bibliography{ugh-bib}


\newpage
\appendix
\section{Data structure for ultrametric spaces and implementation details}\label{sec:data-structure} 
Whereas dendrograms are helpful for our theoretical development, we found a certain rooted tree structure associated to ultrametric spaces to be extremely helpful for designing our algorithms. In this section, we provide a detailed description of such rooted tree structure.
\subsection{Tree structure for ultrametric spaces}\label{sec:tree structure of dend} 
Tree structures for ultrametric spaces are thoroughly studied in the literature \cite{petrov2014gomory,kloeckner2015geometric,dovgoshey2018isomorphic}. Following the \emph{labeled rooted tree} language used in \cite{dovgoshey2018isomorphic}, we provide a description of a \emph{weighted rooted tree} representation of any finite ultrametric space.

A node weighted rooted tree is a tuple $T=(V,E,w,r)$ where $(V,E)$ denotes an undirected tree with $V$ being the vertex set and $E$ being the edge set, $w:V\rightarrow\mathbb R_{\geq0}$ denotes a node weight function and $r\in V$ is a specified vertex called the root of $T$. 
Two weighted rooted trees $T_1=(V_1,E_1,w_1,r_1)$ and $T_2=(V_2,E_2,w_2,r_2)$ are said to be \emph{isomorphic}, if there exists a bijection $f:V_1\rightarrow V_2$ such that 
\begin{enumerate}
    \item for every $x,y\in V_1$, $\{x,y\}\in E_1$ iff $\{f(x),f(y)\}\in E_2$;
    \item for every $x\in V_1$, $w_1(x)=w_2(f(x))$;
    \item $f(r_1)=r_2$.
\end{enumerate}

\begin{remark}[Standard terminology for rooted trees]
Given any weighted rooted tree $T=(V,E,w,r)$, we call a collection of distinct vertices $x_0,x_1,\ldots,x_k\in V$ a \emph{path} if for each $i=0,\ldots,k-1$ we have $\{x_i,x_{i+1}\}\in E$. If for any given distinct $x,y\in V$ there exists a path $x_0=r,x_1,\ldots,x_k=y$ such that $x=x_i$ for some $i=0,\ldots,k-1$, then we say that $x$ is an \emph{ancestor} of $y$ and $y$ is a \emph{descendant} of $x$. If furthermore $\{x,y\}\in E$, then we say that $x$ is the \emph{parent} of $y$ and also that $y$ is a \emph{child} of $x$.
\end{remark}

The following useful fact will be utilized multiple times in the sequel.

\begin{lemma}\label{lm:sum of children}
Given a weighted rooted tree $T=(V,E,w,r)$, we denote by $k_x$ the number of children of any given $x\in V$. Then,
\[\sum_{x\in V}k_x=O(\#V).\]
\end{lemma}
\begin{proof}
Note that $\sum_{x\in V}\mathrm{degree}(x)=2\cdot\#E.$ Since $T$ is a tree, we have that $\#E=\#V-1$. Moreover, $\mathrm{degree}(x)=k_x+1-\delta_{rx}$. Therefore,
\[\sum_{x\in V}k_x=O\!\lc \sum_{x\in V}\mathrm{degree}(x)\rc=O(\#E)=O(\#V).\]
\end{proof}

A dendrogram automatically induces a weighted rooted tree as one can deduce from its graphic representation (see Figure \ref{fig:dendro-tree}). We describe this relationship between dendrograms and weighted rooted trees as follows. Let $(X,u_X)$ be a finite ultrametric space and let $\theta_X$ be its corresponding dendrogram. Following the notation from Section \ref{sec:dp-main-text} we let $V_X$ denote the collection of all closed balls in $X$. It is then clear that $V_X$ is a finite set. Furthermore, $V_X$ contains $X$ and all singletons $\{x\}$ for $x\in X$. Define a collection $E_X$ of two-element \emph{subsets} of $V_X$ as follows: for any two different $B,B'\in V_X$, $\{B,B'\}\in E_X$ iff $B'$ (resp. $B$) is the smallest (under inclusion) ball different from but containing $B$ (resp. $B'$). Then, it is easy to see from the dendrogram that $(V_X,E_X)$ is a combinatorial tree, i.e., a connected graph without cycles (see Figure \ref{fig:dendro-tree} for an illustration), a fact which we record for later use:

\begin{lemma}
$(V_X,E_X)$ is an undirected tree with vertex set $V_X$ and edge set $E_X$.
\end{lemma}

\begin{lemma}\label{lm:spectrum via VX}
Let $X$ be a finite ultrametric space. Then, 
\[\spec(X)=\{\diam(B):\, B\in V_X\}. \]
\end{lemma}

Now, we define a \emph{weighted} rooted tree $T_X$ associated to the ultrametric space $X$.

\begin{definition}[Weighted rooted tree associated to $X$]\label{def:weighted rooted tree}
Define $w_X:V_X\rightarrow\mathbb R_{\geq 0}$ by $w_X(B)\coloneqq\diam(B)$ for each $B\in V_X$. Let $r_X\coloneqq X\in V_X$. Then, we call the tuple $T_X\coloneqq(V_X,E_X,w_X,r_X)$ the weighted rooted tree associated to $X$.
\end{definition}

\begin{figure}[ht]
    \centering
    \includegraphics[width=\textwidth]{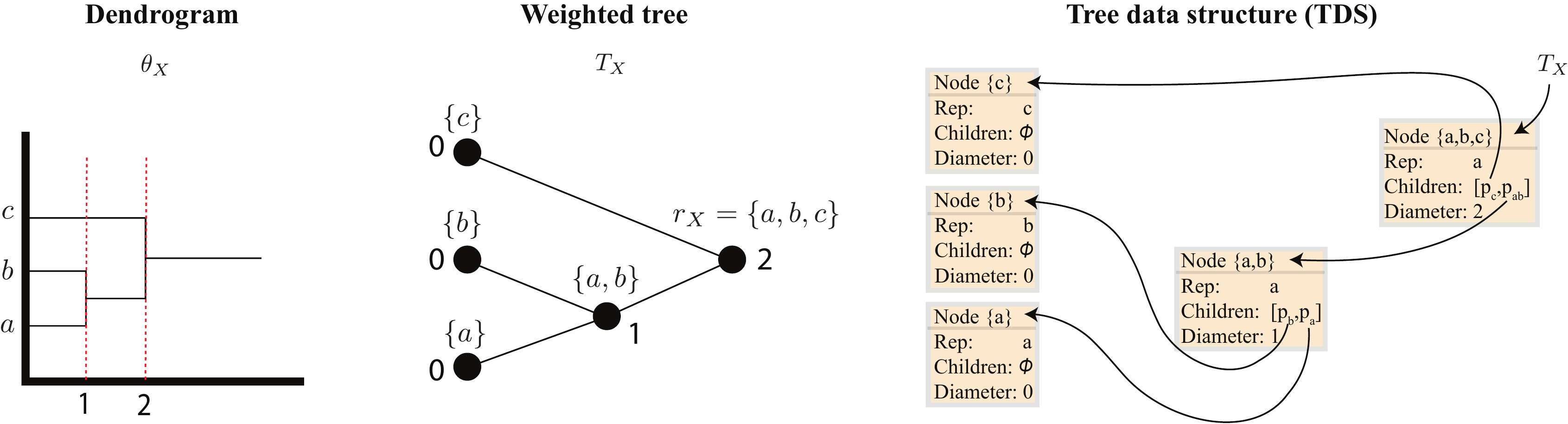}
    \caption{\textbf{Tree structure of dendrograms.} The leftmost figure represents the dendrogram $\theta_X$ of a three-point ultrametric space $X$. The middle figure represents the weighted rooted tree $T_X$ associated to $X$. The numbers beside the nodes represent the weight values given by $w_X$. Note that the tree structure of $T_X$ is inherited directly from the tree structure of $\theta_X$. The rightmost figure shows the TDS associated to $T_X$. }
    \label{fig:dendro-tree}
\end{figure}

\begin{remark}\label{rmk:tree number}
It is obvious that the set of singletons $\{\{x\}\in V_X:\,x\in X\}$ coincides with the set of leaves of $T_X$. Since for every rooted tree $\#\mathrm{vertices}=O(\#\mathrm{leaves})$ (indeed, $\#\mathrm{leaves}\leq \#\mathrm{vertices}\leq 2\#\mathrm{leaves}$), we have that $\#V_X=O(n)$ where $n\coloneqq\#X$. Moreover, since $\#E_X=\#V_X-1=O(V_X)$, we have that $\#E_X=O(n)$ as well.
\end{remark}

\begin{proposition}[{\cite[Theorem 1.10]{dovgoshey2018isomorphic}}]\label{prop:isomorphism isometry}
For any two finite ultrametric spaces $X$ and $Y$, let $T_X$ and $T_Y$ denote their corresponding weighted rooted trees. Then, $X$ is isometric to $Y$ iff $T_X$ is isomorphic to $T_Y$.
\end{proposition}

\begin{remark}[Relation to merge trees]\label{rmk:merge tree} For the precise definition of merge trees, see for example \cite{morozov2013interleaving}. Let $X$ be a finite ultrametric space and let $T_X$ be its associated weighted rooted tree.
We first transform $T_X$ into a topological/metric tree by replacing each edge $\{B,B'\}$ with an interval with length $|\diam(B)-\diam(B')|$. We then attach to the root $r_X$ the half line $[0,\infty)$ to obtain the topological space $M_X$. We then define the height function $h_X:M_X\rightarrow\mathbb{R}$ as follows:
\begin{enumerate}
    \item $h_X(B)=w_X(B)=\diam(B)$ for each $B\in V_X$;
    \item for each edge, inside its corresponding interval, $h_X$ is defined as the linear interpolation between the function values at the end points;
    \item for $t$ on the half line $[0,\infty)$, $h_X(t)\coloneqq h_X(r_X)+t$.
\end{enumerate}
In this way, each finite ultrametric space $X$ canonically induces the merge tree $\lc M_X,h_X\rc$.
\end{remark}

\begin{remark}[Detailed comparison between $\gamma_\eps(X)$ and $\tau_\eps(M_X)$]\label{rmk:detail comparison with degree bound}
Recall that for $\eps\geq 0$, the $\eps$-degree bound $\tau_\eps(M_X)$ of a merge tree $M_X$ is the maximum over all closed $\eps$ balls the sum of degrees of all vertices inside the ball (cf. Remark \ref{rmk:degree bound in main text}). By Remark \ref{rmk:merge tree}, each finite ultrametric space $X$ canonically induces the merge tree $(M_X,h_X)$. Then, we call $\tau_\eps(X)\coloneqq\tau_\eps(M_X)$ the $\eps$-degree bound of the ultrametric space $X$. It is easy to see that for any ultrametric space $X$, we have that 
\[\gamma_\eps(X)\leq \tau_\eps(X)\leq 2\gamma_\eps(X).\] 
See Figure \ref{fig:degreebound} for an illustration and a sketch of the proof of this fact. In particular, this implies that $X\in\mathcal{U}_2(\frac{\eps}{2},\gamma)$ for all $\gamma\geq \tau_\eps(X)$. 
\end{remark}
\begin{figure}[ht]
    \centering
    \includegraphics[width=\textwidth]{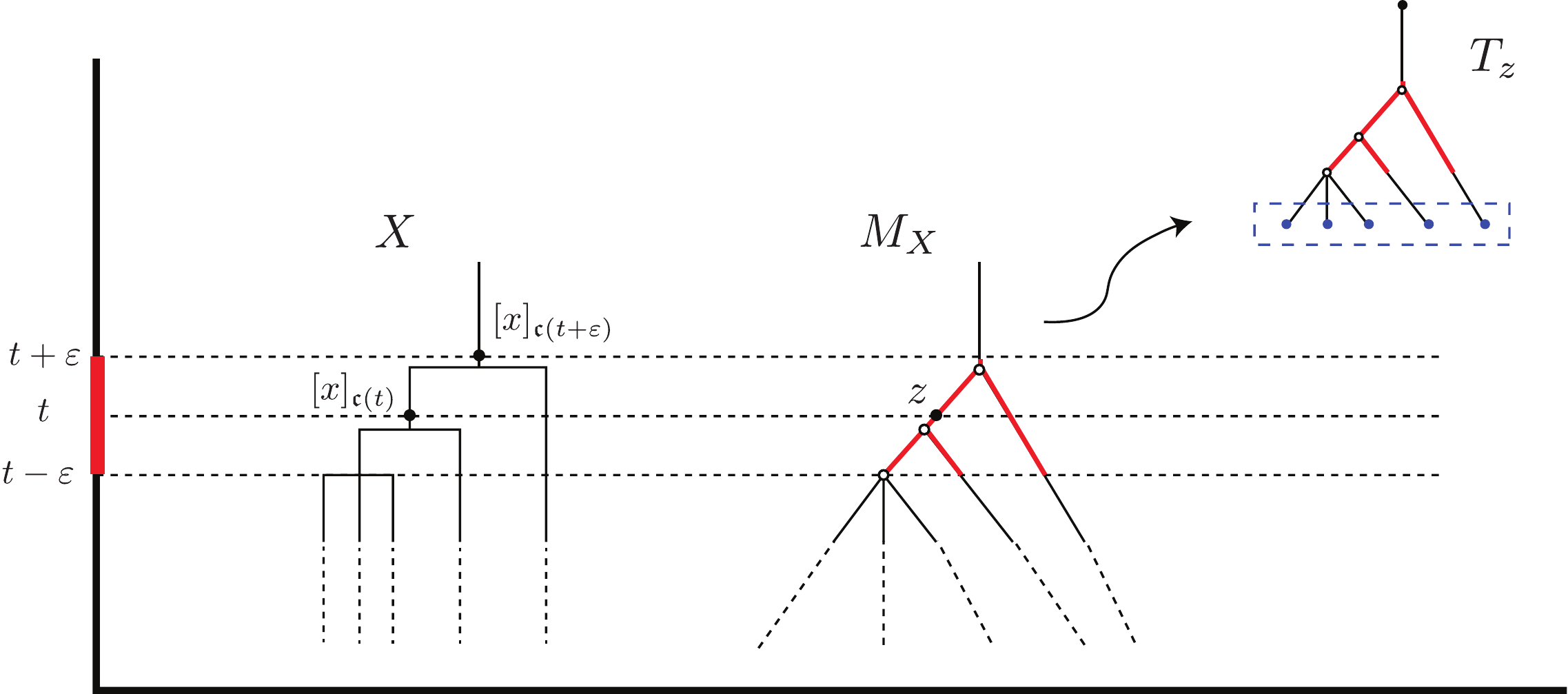}
    \caption{\textbf{Illustration of Remark \ref{rmk:detail comparison with degree bound}.} In this figure, we draw part of the dendrogram of an ultrametric space $X$ and its corresponding merge tree $M_X$. Pick any $z\in M_X$ with height $t$. Then, $z$ corresponds to $[x]_{\mathfrak{c}(t)}$ for some point $x\in X$ (cf. Remark \ref{rmk:merge tree}). Consider the closed ball $B_\eps^{M_X}(z)\subseteq M_X$ highlighted in red in the figure. Considering a small closed neighborhood of the highlighted part in $M_X$, we obtain the rooted tree $T_z$ in the top right corner of the figure. It is obvious that the number of leaves of $T_z$ inside the dotted box equals $\gamma_\eps(x,t+\eps)\coloneqq\#\left\{{[x']_{\mathfrak{o}\lc t-\eps\rc}}:\,x'\in [x]_{\mathfrak{c}\lc t+\eps\rc}\right\}$. On the other hand, the sum of degrees of vertices of the tree $M_X$ contained in the ball $B_\eps^{M_X}(z)$, denoted by $\tau_\eps(z)$, equals the number of edges in $T_z$. Therefore, by the standard relationship between the number of leaves and the number of edges in a rooted tree, we have that $\gamma_\eps(x,t+\eps)\leq\tau_\eps(z)\leq 2\,\gamma_\eps(x,t+\eps)$. From this observation, we conclude that $\gamma_\eps(X)\leq \tau_\eps(X)\leq 2\,\gamma_\eps(X).$ }
    \label{fig:degreebound}
\end{figure}

\subsection{Data structure for ultrametric spaces and algorithms for fundamental operations}\label{sec: data structure and algorithm}
Given a finite ultrametric space $(X,u_X)$, let $T_X$ be its corresponding weighted rooted tree as described in Definition \ref{def:weighted rooted tree}. We utilize a special {self-referential} tree structure to represent $T_X$; for a description of self-referential {tree} data structures, we refer readers to \cite[Chapter 6.5]{kernighan2006c}. In order to represent vertices of $T_X$, we design a special class \texttt{Node} with three fields: \texttt{Representative}, \texttt{Diameter} and \texttt{Children}. Recall that each vertex of $T_X$ represents a closed ball of $X$. Then, for each vertex/ball $B\in V_X$, its corresponding \texttt{Node} object,  \textbf{which will still be denoted by $B$,} contains the following fields:
\begin{enumerate}
    \item \texttt{Representative}: an element $x\in B$;\footnote{Any choice of $x\in B$ will do.}
    \item \texttt{Diameter}: the diameter of $B$;
    \item \texttt{Children}: a list of pointers to all \texttt{Node} objects representing the children of $B$ in $T_X$.
\end{enumerate}
In what follows, we will sometimes call a \texttt{Node} object simply a \emph{node} and we will for instance write $B.\texttt{Diameter}$ to extract the diameter value of a node $B$.

Now, for a given ultrametric space $X$, the \emph{tree data structure} (TDS) which we will use to represent $T_X$ consists of a collection of nodes stored in memory, one for each vertex $B\in V_X$. This TDS is held by a \texttt{Node} pointer referencing the node representing the root $r_X=X$. We call this pointer (to a \texttt{Node} object)  the \textit{root pointer} of the TDS.
The role of the root pointer of a TDS is analogous to the role of the head pointer of a linked list.
See Figure \ref{fig:dendro-tree} for an illustration. 

\begin{framed}\label{notation frame}
\noindent\textbf{Note.} In the sequel, we will overload the symbol $T_X$ and will use it to denote both the weighted rooted tree corresponding to $X$ and to also represent its associated TDS (both in the sense of its root pointer and in the sense of the collection of its nodes). The symbol $B$ for representing any ball $B\in V_X$ is also overloaded to represent its corresponding \texttt{Node} object in $T_X$. In all our algorithms, every ultrametric space is understood as a \texttt{Node} object. 
\end{framed}
\nomenclature{$T_X$}{The weighted rooted tree corresponding to $X$ and also its associated TDS; page \pageref{notation frame}.}

\begin{remark}[Relationship with the distance matrix data structure]\label{rmk:distance mtx}
Starting from the root node $r_X$ of $T_X$, one can progressively trace all nodes in $T_X$ to completely reconstruct the distance matrix of the ultrametric space $X$. It is clear that the reconstruction process takes time at most $O(n^2)$ where $n\coloneqq\#X$. Conversely, given the distance matrix of $X$, one can construct the TDS $T_X$ in the same time complexity $O(n^2)$ (one can use for example the single-linkage algorithm \cite{mullner2011modern} to first obtain the dendrogram induced by the distance matrix (see Figure \ref{fig:dendro-tree})). If the ultrametric spaces are given in terms of distance matrices, we first need to convert them into TDSs before applying the algorithms described in this paper. The time complexity $O(n^2)$ due to this preprocessing is not counted when analyzing our algorithms.
\end{remark}

\begin{remark}[Subtree rooted at a ball]\label{rmk:subtree tb}
Given the TDS $T_X$ associated to the ultrametric space $X$, it is easy to induce a TDS $T_B$ for each closed ball $B\in V_X$ (which is itself an ultrametric space). Such $T_B$ consists of all descendant nodes of $B$ in $T_X$ and is held by a pointer to $B$, i.e., $*(T_B)=B$. Here $*p$ denotes the datum referenced by a pointer $p$.
\end{remark}

\begin{remark}[Finding parent nodes]\label{rmk:parent node}
Given any non-root node $B$ in $T_X$, we let $\texttt{Parent}(B,T_X)$ denote its parent node in $T_X$. To actually find the node $\texttt{Parent}(B,T_X)$ given the node $B$, one can start from the root node $r_X$ and recursively search $T_X$ for the parent of $B$. Such a search process takes at most time $O(\#T_X)=O(n)$ where $n\coloneqq \#X$.
\end{remark}

One major advantage of adopting the above-mentioned TDS is that it allows us to efficiently implement certain fundamental operations on ultrametric spaces. For example, it allows to efficiently obtain the spectrum of an ultrametric space:
\begin{lemma}\label{lm:spectrum computation}
Given an ultrametric space $X$, determining and sorting $\spec(X)$ can be done in time $O(n\log(n))$ where $n\coloneqq\#X$.
\end{lemma}
\begin{proof}
Note that, by Lemma \ref{lm:spectrum via VX}, we only need to inspect each node in $T_X$ which takes time $O(n)$. Since $\#\spec(X)=O(n)$ (cf. Proposition \ref{prop:basic property ultrametric}), the sorting process takes time $O(n\log(n))$.
\end{proof}

Next, we introduce algorithms for other three fundamental operations on ultrametric spaces.

\paragraph{Closed quotient} In Algorithm \ref{algo-c-quotient} we introduce a recursive algorithm for the $t$-closed quotient operation.
In the algorithm, the notation $\gets$ represents variable assignment and the notation $\&Z$ denotes the memory address of the variable $Z$. In line 6 of the algorithm, the \emph{one point tree} data structure consists of a single node such that $\texttt{Diameter}=0$ and $\texttt{Children}$ is an empty list. Note that in the worst-case scenario, Algorithm \ref{algo-c-quotient} inspects all nodes of $T_X$ during the recursion process. At the recursion call with input $(B,t)$ where $B$ is a node of $T_X$, the main computational cost lies in line 1 for copying the \texttt{Node} object $B$ (in the place of $X$). If we let $k_B\coloneqq\mathrm{length}(B.\texttt{Children})$, it then costs time $O(k_B)$ to copy $B$ and it takes at most time $O(k_B)$ to update $B.\texttt{Children}$ according to the for-loop in line 2. Therefore, the total time complexity of Algorithm \ref{algo-c-quotient} is bounded above by $\sum_{B\in V_X}O(k_B)$. By Lemma \ref{lm:sum of children} and Remark \ref{rmk:tree number}, we have that  
\[\sum_{B\in V_X}O(k_B)=O(V_X)=O(n),\]
where $n\coloneqq \#X$. Therefore, the time complexity of Algorithm \ref{algo-c-quotient} is bounded above by $O(n)$.

\begin{algorithm}[htb]
\caption{$\mathbf{ClosedQuotient}(X,t)$}\label{algo-c-quotient}
\begin{algorithmic}[1]
  \STATE{$Y\gets X$}\;
  \FOR{Each $p\in Y.\texttt{Children}$}
     \IF{$*p.\texttt{Diameter}> t$}
     \STATE{$p\gets\&\mathbf{ClosedQuotient}(*p,t)$}
     \ELSE
     \STATE{$Z\gets$ the one point tree with the same representative in $*p$}
     \STATE{$p\gets\&Z$ }
     \ENDIF
\ENDFOR
     \RETURN $Y$
\end{algorithmic}
\end{algorithm}

\paragraph{Open partition} In Algorithm \ref{algo-o-part} we give pseudocode for constructing the $t$-open partition of an ultrametric space. In order to implement the `append' operation (appearing in line 3 and line 6) in constant time, we use a doubly linked list $P$ to represent the open partition. Algorithm \ref{algo-o-part} will recursively inspect all nodes corresponding to balls in $X_{\mathfrak{o}(t)}$ as well as all of their ancestor nodes. For each inspected node $B$ in $T_X$,  if we let $k_B\coloneqq \mathrm{length}(B.\texttt{Children})$, it then takes time $O(k_B)$ to update the list $P$. Then, following a similar argument as in the complexity analysis of Algorithm \ref{algo-c-quotient}, if we let $k$ be the cardinality of $P$ (i.e., $k=\#X_{\mathfrak{o}(t)}$), then Algorithm \ref{algo-o-part} will inspect at most $O(k)$ nodes in $T_X$ and thus it generates $P$ in at most $O(k)$ steps.

\begin{algorithm}[htb]
\caption{$\mathbf{OpenPartition}(X,t)$}\label{algo-o-part}
\begin{algorithmic}[1]
 \STATE $P = [\,]$ 
  \IF{$X.\texttt{Diameter}< t$}
    \STATE{$P.\mathrm{append}(\&X)$}
  \ELSE 
  {\FOR{Each $p_i \in X.\texttt{Children}$}
    \STATE $P.\mathrm{append}(\mathbf{OpenPartition}(*p_i,t))$
  \ENDFOR}
  \ENDIF
  \RETURN $P$
\end{algorithmic}
\end{algorithm}

\paragraph{Isometry between ultrametric spaces} By Proposition \ref{prop:isomorphism isometry}, two ultrametric spaces $X$ and $Y$ are isometric iff their corresponding weighted rooted trees $T_X$ and $T_Y$ are isomorphic. By adapting the algorithm in \cite[Example 3.2]{aho1974design}, determining isomorphism between rooted trees can be done in time $O(\#\mathrm{vertices})$; see also \cite[Theorem 3.3]{aho1974design} (and its corollary). Then, by Remark \ref{rmk:tree number}, we have the following result:
\begin{lemma}\label{lm:isometry linear time}
We can determine whether $X$ and $Y$ are isometric in $O(n)$ time.
\end{lemma}

\subsection{Subspace tree data structure and union of non-intersecting subspaces}\label{sec:union of susbets} In this section, we explain how to utilize the TDS described in the previous section to efficiently perform the union operation (under the conditions specified by Equation \eqref{eq:union condition} below).

First of all, we introduce a TDS for representing subspaces of a given ultrametric space $X$.  We assume that the distance matrix $u_X$ is available and assume that a TDS $T_X$ representing $X$ has already been computed.
\begin{definition}[Subspace tree data structure]
For any non-empty subspace $U\subseteq X$, we say that a tree data structure $T_U$ representing the ultrametric space $(U,u_X|_{U\times U})$ is a \emph{subspace tree data structure subordinate to} $T_X$, if each vertex (i.e., each ball) in $V_{U}\cap V_X$ is represented by a \texttt{Node} object belonging to the tree data structure $T_X$.
\end{definition}

\begin{remark}[Construction of subspace TDSs]
For any node $B\in T_X$ (which  represents a ball in $X$), the TDS $T_B$ described in Remark \ref{rmk:subtree tb} is obviously a subspace TDS subordinate to $T_X$ representing the subspace $B$. However, for an arbitrary subset $U$ the situation is different from the case of a ball. First, it is easy to verify that any such $U$ can be written as a union of non-intersecting balls $\{B_i\}_{i=1}^k$ satisfying the condition in Equation \eqref{eq:union condition} (see also  Lemma \ref{lm:ball decomposition} below). Then, a subspace TDS representing $U$ subordinate to $X$  can be constructed by applying the union process which we describe below to the set of balls $\{B_i\}_{i=1}^k$.
\end{remark}

Consider a set of non-empty and non-intersecting subspaces $\{U_1,\ldots,U_k\}$ of a given ultrametric space $X$ such that for any distinct $i,j=1,\ldots,k$, we have
\begin{equation}\label{eq:union condition}
    \min_{x_i\in U_i,x_j\in U_j}u_X(x_i,x_j)>\max(\diam(U_i),\diam(U_j)).
\end{equation}
This condition is compatible with the sets obtained by taking a slice of the dendrogram $\theta_X$, which is in turn equivalent to considering open/closed equivalence classes of ultrametric spaces (see the discussion below Definition \ref{def:dendrogram}). We assume that each $U_i$ is represented by a subspace TDS $T_{U_i}$ subordinate to $T_X$.

Now given the above data, we describe how to construct a subspace TDS $T_U$ subordinate to $T_X$ representing the union $U\coloneqq \cup_{i=1}^kU_i$. The whole process is organized through the following three steps.

\makeatletter
\newcommand*{\rom}[1]{\expandafter\@slowromancap\romannumeral #1@}
\makeatother

\paragraph{(\rom{1}) Constructing a TDS induced by representatives}
For each $i=1,\ldots,k$ let $x_i$ be the representative of $U_i$ as given in the TDS $T_{U_i}$ and let $X_k\coloneqq\{x_1,\ldots,x_k\}$. We first consider the ultrametric $u_{X_k}\coloneqq u_X|_{X_k\times X_k}$ on $X_k$ induced by the restriction of $u_X$ to $X_k\times X_k$. Then, we construct a new TDS $T_{X_k}$ to represent $(X_k,u_{X_k})$ (cf. Remark \ref{rmk:distance mtx}). This construction is possible due to the fact that $(X_k,u_{X_k})$ is itself an ultrametric space.

It takes time at most $O(k^2)$ to both construct the metric $u_{X_k}$ and create the new TDS $T_{X_k}$ (cf. Remark \ref{rmk:distance mtx}).

\paragraph{(\rom{2}) Constructing the preliminary union of $U_i$s} Recall that up to this point, we have at our disposal the following TDSs: $T_{U_1},\ldots,T_{U_k}$, and $T_{X_k}$. Based on these data, we progressively \textit{modify} leaf nodes of $T_{X_k}$ and utilize $T_{U_1},\ldots,T_{U_k}$ to find a TDS representation for the union $U$. For pedagogical reasons, we refer to the outcome TDS as $\Tilde{T}_U$. $\Tilde{T}_U$ may not be subordinate to $T_X$, and we thus name it the \emph{preliminary union} of $U_i$s. The modification process can be summarized as simply replacing each leaf node in $T_{X_k}$ with the root node of certain $T_{U_i}$ as shown in Figure \ref{fig:preliminary union}. More precisely, we traverse all nodes $B$ in $T_{X_k}$ and, if for such a node $B$ there exists an index $i_0$ such that $*(B.\texttt{Children}(i_0)).\texttt{Diameter}=0$ (which means $B$ is the parent of a leaf node), then we let $j_0$ be the index such that $*(B.\texttt{Children}(i_0)).\texttt{Representative}=x_{j_0}\in X_k$, and modify $B$ by assigning $B.\texttt{Children}(i_0)=T_{U_{j_0}}$ (and of course we free the memory used for storing the original node $*(B.\texttt{Children}(i_0))$). 

\begin{figure}[ht]
    \centering
    \includegraphics[width=\textwidth]{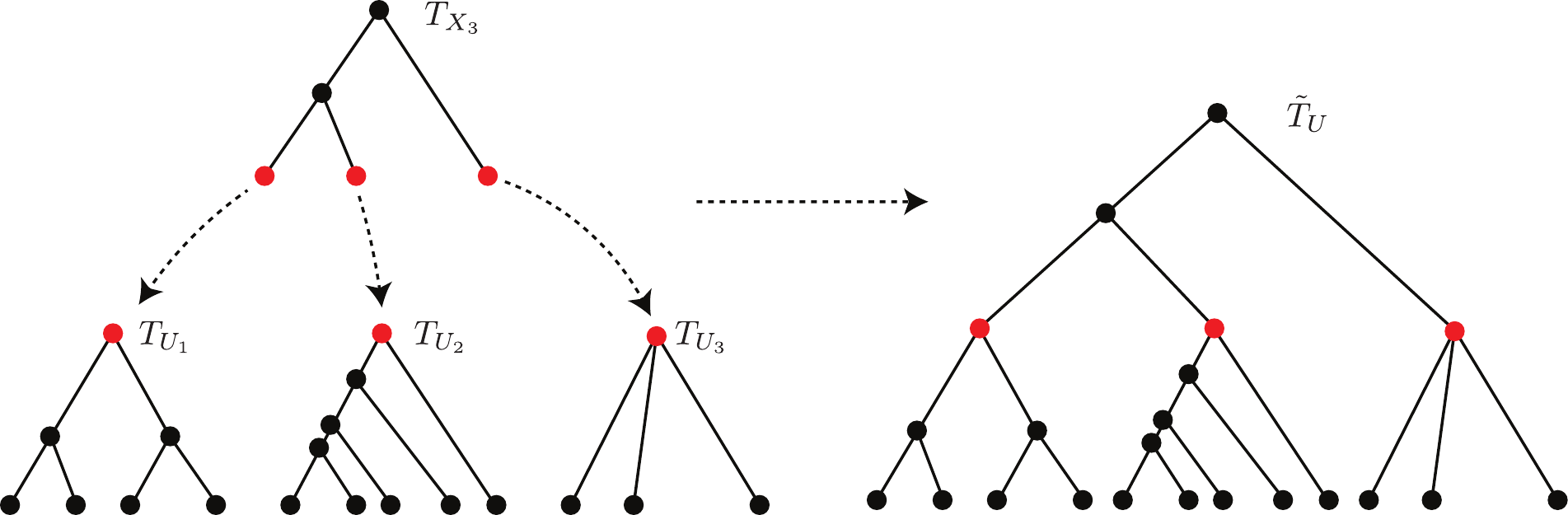}
    \caption{\textbf{Illustration of the process for producing the preliminary union $\Tilde{T}_U$.} }
    \label{fig:preliminary union}
\end{figure}

For the modification process described above, we need to modify at most $O(k)$ pointers, and for each such pointer it takes time $O(k)$ to search for $U_{j_0}$ with the desired representative as described above. So, the time needed for constructing the preliminary union $\Tilde{T}_U$ is at most $O(k^2)$.

\paragraph{(\rom{3}) Taming process for $\tilde{T}_{U}$} The TDS $\tilde{T}_{U}$ may not be a subspace TDS subordinate to $T_X$, i.e., $\tilde{T}_{U}$ may contain \texttt{Node} objects representing balls in $V_X$ which do not belong to the TDS $T_X$. We will thus further \emph{tame} $\Tilde{T}_{U}$ so that the outcome TDS is subordinate to $T_X$. 
For pedagogical reasons, we use the symbol $T_U$  (which is our final TDS representation for $U$) to refer to the tamed $\Tilde{T}_U$.

To accomplish this, we first create an array $\mathrm{L}$ consisting of pointers to all nodes in $\Tilde{T}_U\backslash \cup_{i=1}^k T_{U_i}$, i.e., all modified nodes from $T_{X_k}$. We sort $\mathrm{L}$ according to increasing \texttt{Diameter} values of the nodes referenced by its pointers. There are $O(k)$ such nodes and thus building and sorting $\mathrm{L}$ takes time $O(k\log(k))$. Let $\ell\coloneqq\mathrm{length}(\mathrm{L})=O(k)$. For each $i=1,\ldots,\ell$, we compare $*(\mathrm{L}(i))$ with each node in $T_X$ as described next. If for some $B\in T_X$, the sets $B.\texttt{Children}$ and $*(\mathrm{L}(i)).\texttt{Children}$ agree, we then replace the node $*(\mathrm{L}(i))$ in $\tilde{T}_{U}$ with $B$. There are two cases which can arise during this replacement:
\begin{enumerate}
    \item if $*(\mathrm{L}(i))$ is not the root node of $*(\Tilde{T}_U)$, we find the parent node of $*(\mathrm{L}(i))$ (cf. Remark \ref{rmk:parent node}) in $\Tilde{T}_U$ and let $i_0$ be the index such that $\texttt{Parent}\lc*(\mathrm{L}(i)),\Tilde{T}_U\rc.\texttt{Children}(i_0)=\mathrm{L}(i)$. Then, we replace $\texttt{Parent}\lc*(\mathrm{L}(i)),\Tilde{T}_U\rc.\texttt{Children}(i_0)$ with a pointer to $B\in T_X$;
    \item otherwise if $*(\mathrm{L}(i))$ is the root node $*(\Tilde{T}_U)$, we free the memory used for storing $*(\Tilde{T}_U)$, and then assign $\Tilde{T}_U=\&B$.
\end{enumerate}

For each $i=1,\ldots,\ell$ let $k_i\coloneqq \mathrm{length}(*(\mathrm{L}(i)).\texttt{Children})$. Determining whether $B.\texttt{Children}=*(\mathrm{L}(i)).\texttt{Children}$ takes time $O(k_i)$. Therefore, the time complexity of the replacement process mentioned above can be bounded as follows:
\[\sum_{i=1}^\ell O(\#T_X)\times k_i=\sum_{i=1}^\ell O(nk_i)= O\lc n \sum_{i=1}^\ell k_i\rc=O(nk),\]
where $n\coloneqq\#X$ (we used Lemma \ref{lm:sum of children} in the rightmost equality). The time incurred when finding and accessing the parent node of a given node in $\mathrm{L}$ is at most $O(k)$ (cf. Remark \ref{rmk:parent node}). So the total time required for taming $\Tilde{T}_U$ is at most $O(nk)$.

Therefore by following steps (\rom{1}), (\rom{2}) and (\rom{3}), the total time complexity of the union operation can be bounded by $O(k^2)+O(nk)=O(nk)$.

\subsection{Implementation details for Algorithm \ref{algo-dGH-dyn}}\label{sec:implementation detail}

In this section, we provide details for one possible implementation of Algorithm \ref{algo-dGH-dyn}.


\subsubsection{Preprocessing}\label{sec:preprocessing}
To achieve an actual implementation of Algorithm \ref{algo-dGH-dyn}, some preprocessing is needed in order to construct the {arrays} $\mathrm{LX}^{(\eps)}$ and $\mathrm{LY}$ defined in Section \ref{sec:dp-main-text}. Here, for completeness we describe one possible implementation of these preprocessing steps. We first construct subspace TDSs for subspaces in $\mathrm{LX}^{(\eps)}$ and $\mathrm{LY}$. We then construct arrays $\mathrm{pLX}^{(\eps)}$ and $\mathrm{pLY}$ of \texttt{Node} pointers referencing subspaces in $\mathrm{LX}^{(\eps)}$ and $\mathrm{LY}$, respectively. For this purpose, we augment the class \texttt{Node} by incorporating an integer field called \texttt{Order} and a Boolean field called \texttt{IsBall}. For each node $B$ in $T_X$, \texttt{Order} is initially set to $-1$ and $\texttt{IsBall}$ is set to \texttt{True}. We initialize nodes in $T_Y$ in the same way.

\paragraph{Construction of $\mathrm{pLY}$} Given the TDS $T_Y$, we first create an array $\mathrm{pLY}$ containing pointers to all nodes in $T_Y$. Then, we sort $\mathrm{pLY}$ according to increasing \texttt{Diameter} values of the \texttt{Node} objects referenced by its pointers. This finalizes constructing the array $\mathrm{pLY}$. For each $i=1,\ldots,\#\mathrm{pLY}$, we set $*(\mathrm{pLY}(i)).\texttt{Order}=i$. In this way, each element in $T_Y$ is such that its field \texttt{Order} is different from $-1$. 

By Lemma \ref{lm:counting-lxly}, $\# T_Y=\#V_Y =O(n)$. So building and sorting the list $\mathrm{pLY}$ can be done in time $O(n\log(n))$. The above process for setting $B^Y.\texttt{Order}$ for all $B^Y\in\mathrm{pLY}$ can be done in time $O(n)$.

\paragraph{Construction of $\mathrm{pLX}^{(\eps)}$} Unlike the case of $\mathrm{LY}$, the array $\mathrm{LX}^{(\eps)}$ can contain subspaces not belonging to $V_X$. In order to construct $\mathrm{pLX}^{(\eps)}$, we follow the three steps that we describe next:
\begin{enumerate}
    \item Create subspace TDSs subordinate to $T_X$ for representing each of the subspaces in $\mathrm{LX}^{(\eps)}\backslash \mathrm{LX}$;
    \item Create a \emph{list} $\mathrm{pLX}^{(\eps)}$ of pointers referencing the root nodes representing each of the subspaces in $\mathrm{LX}^{(\eps)}$.
    \item Transform the \emph{list} $\mathrm{pLX}^{(\eps)}$ into an \emph{array} (still denoted by $\mathrm{pLX}^{(\eps)}$). In this way, the random access time to elements in $\mathrm{pLX}^{(\eps)}$ is $O(1)$.
\end{enumerate}

Whereas the third step is clear, we will provide a detailed description for the first and second steps. In fact, these two steps are accomplished at the same time: We first construct $\mathrm{pLX}$ via a process analogous to the one used above for constructing $\mathrm{pLY}$. This step can be done in time $O(n\log(n))$. Then, we follow two substeps: (a) for each $B^X\in \mathrm{LX}$, we will first construct subspace TDSs subordinate to $T_X$ for all subspaces in $B^X_{(\eps)}\backslash\{ B_X\}$; we then construct the list $\mathrm{p}B^X_{(\eps)}$ of pointers referencing all subspaces in $B^X_{(\eps)}$; (b) we will use these constructions from (a) for all $B^X$ to complete step 1 and step 2. Below, we describe substep (a) and substep (b) in detail.

\subparagraph{(a) Constructions regarding a single $B^X$} We set $\mathrm{p}B^X_{(\eps)}$ to be an empty list. 
Applying $\mathbf{OpenPartition}$ (Algorithm \ref{algo-o-part}) to $B^X$ we obtain the open partition $B^X_{\mathfrak{o}\lc \rho_\eps\lc B^X\rc\rc}=\left\{B_i\right\}_{i=1}^{N_{B^X}}$. By the SGC, $N_{B^X}\leq \gamma$ so that the open partition process takes time at most $O(\gamma)$. For each non-empty index set $I\subseteq[N_{B^X}]$ (there are $O(2^\gamma)$ such $I$s), we apply the union operation (described in Appendix \ref{sec:union of susbets}) to obtain the TDS $T_{U_I}$ for the union $U_I\coloneqq\cup_{i\in I}B_i$ which takes time at most $O(n\times\#I)=O(n\gamma)$. For the new nodes thus created, i.e., for nodes $U$ in $T_{U_I}\backslash T_X$, we set $U.\texttt{IsBall}=\texttt{False}$ and set $U.\texttt{Order}=-1$. If $U_I.\texttt{Diameter}= B^X.\texttt{Diameter}$, we first set $U_I.\texttt{Order}=B^X.\texttt{Order}$ and then update $\mathrm{p}B^X_{(\eps)}$ by appending $T_{U_I}$ (which is a pointer to the node $U_I$) to it. Otherwise, we delete all nodes in $T_{U_I}\backslash T_X$.

In summary, the time complexity for both constructing subspace TDSs subordinate to $T_X$ for all elements in $B^X_{(\eps)}\backslash\{ B_X\}$ and constructing the list $\mathrm{p}B^X_{(\eps)}$ is bounded by $O(n2^\gamma\gamma^2)$.

\subparagraph{(b) Completing step 1 and step 2} Finally, we apply the constructions described above to all $B^X\in\mathrm{LX}$ and assemble the corresponding outputs to complete step 1 and step 2 concurrently. More specifically, we first sort $\mathrm{LX}$ according to increasing values of \texttt{Order}. Then, we apply the above constructions over all $B^X\in \mathrm{LX}$ with respect to this order. After this, subspace TDSs for all elements in $\mathrm{LX}^{(\eps)}\backslash \mathrm{LX}$ have been constructed and stored in memory. Finally, we merge all the resulting lists $\mathrm{p}B^X_{(\eps)}$s together to obtain the list $\mathrm{pLX}^{(\eps)}$. Since $\#\mathrm{LX}=O(n)$ and since for each $B^X$, $\mathrm{length}(\mathrm{p}B^X_{(\eps)})=O(2^\gamma)$, the time needed for constructing subspace TDSs for elements in $\mathrm{LX}^{(\eps)}\backslash \mathrm{LX}$ is bounded by $O(n^22^\gamma\gamma^2)$, and the merging process takes time at most $O(n2^\gamma)$ (the complexity can be reduced to $O(n)$ if each $\mathrm{p}B^X_{(\eps)}$ is represented by a doubly linked list).

Therefore, step 1 and step 2 together can be accomplished in time
\[O\big(n\log(n)+n^22^\gamma\gamma^2+n2^\gamma\big)=O\big(n^2\log(n)2^\gamma\gamma^2\big).\]
Since $\#\mathrm{pLX}^{(\eps)}=O(n2^\gamma)$, the time for transforming $\mathrm{pLX}^{(\eps)}$ into an array is bounded by $O(n2^\gamma)$. So, the total time complexity for building the array $\mathrm{pLX}^{(\eps)}$ is at most $O\big(n^2\log(n)2^\gamma\gamma^2\big)$.

\paragraph{Data structure for storing and accessing indices of elements in $\mathrm{LX}^{(\eps)}$} Given any $U^X\in\mathrm{LX}^{(\eps)}$, there exists a unique integer $\mathrm{ind}$ such that $U^X=*(\mathrm{pLX}^{(\eps)}(\mathrm{ind}))$ and we refer to $\mathrm{ind}$ as the \emph{index} of $U^X$ in $\mathrm{LX}^{(\eps)}$. Now, given any $U^X$ in the form of a subspace TDS subordinate to $T_X$, in order to access the index of $U^X$ in $\mathrm{LX}^{(\eps)}$ efficiently, we construct a $\underbrace{(1+\#V_X)\times \cdots\times (1+\#V_X)}_{\gamma\text{ terms}}$ dimensional multi-array $\mathrm{INDX}$ for storing all such indices. Each dimension of $\mathrm{INDX}$ is indexed by integers in the range $\{-1\}\cup \{1,\ldots,\#V_X\}$.
The following lemma gives rise to our strategy for indexing this multi-array.

\begin{lemma}\label{lm:ball decomposition}
For any subset $U^X\subseteq X$, there is a \emph{unique} maximal set of non-intersecting closed balls $\{B_1,\ldots,B_k\}\subseteq V_X$ such that $U^X=\cup_{i=1}^kB_i$. Here `maximal' means that if there exists another set of non-intersecting closed balls $\{B_1',\ldots,B_l'\}\subseteq V_X$ such that $U^X=\cup_{i=1}^lB_i'$, then for each $i=1,\ldots,l$, there exists some  $j=1,\ldots,k$ such that $B_i'\subseteq B_j$.
We call the unique maximal set $\{B_1,\ldots,B_k\}\subseteq V_X$ the \emph{ball decomposition of $U^X$}.
\end{lemma}

Now, given $U^X\in\mathrm{LX}^{(\eps)}$, let $\mathrm{ind}$ denote its index in $\mathrm{LX}^{(\eps)}$. Let $\{B_1,\ldots,B_k\}$ be the ball decomposition of $U^X$ whose elements are labeled such that
\[B_1.\texttt{Order}<\cdots<B_k.\texttt{Order}. \]
We then store $\mathrm{ind}$, the index of $U^X$, in $\mathrm{INDX}$ as follows
\[\mathrm{INDX}(B_1.\texttt{Order},B_2.\texttt{Order},\ldots,B_k.\texttt{Order},\underbrace{-1,\ldots,-1)}_{\gamma-k\text{ terms}}=\mathrm{ind}. \]

\paragraph{Computation of the ball decomposition} Given $U^X\in\mathrm{LX}^{(\eps)}$, we proceed to compute its ball decomposition as follows. If $U^X.\texttt{Order}\neq -1$, then $U^X\in \mathrm{LX}$, i.e., $U^X$ already represents a ball in $X$. In this case, $\{U^X\}$ is the ball decomposition of $U^X$. Otherwise, we traverse {all} nodes of $T_{U^X}$ to create the list $\mathrm{p}U^X$ consisting of pointers to all nodes $\{B_1,\ldots,B_k\}\subseteq T_{U^X}$ such that: for each $i=1,\ldots,k$, $B_i.\texttt{IsBall}=\texttt{True}$ but the parent node of $B_i$ satisfies $\texttt{Parent}\lc B_i,T_{U^X}\rc.\texttt{IsBall}=\texttt{False}$. Then, $\{B_1,\ldots,B_k\}$ is the desired ball decomposition of $U^X$.

The computation of the ball decomposition of $U^X\in\mathrm{LX}^{(\eps)}$ described above can be done in time $O(\gamma\log(\gamma))$ (including the sorting time). Therefore, the total time complexity for storing the indices of all $U^X\in\mathrm{LX}^{(\eps)}$ into $\mathrm{INDX}$ is bounded by $O(n2^\gamma\gamma\log(\gamma))$ and the time needed for finding the index of $U^X$ into $\mathrm{LX}^{(\eps)}$ is bounded by $O(\gamma\log(\gamma))$. We remark that the space complexity of the multi-array $\mathrm{INDX}$ is $O(n^\gamma)$. The actual size of $\mathrm{LX}^{(\eps)}$ is however $O(n2^\gamma)$. To reduce the space complexity, one could consider a sparse multi-array data structure or binary search trees.

\subsubsection{A refined union operation via an improved taming process}\label{sec:refined union} 
In line 16 of Algorithm \ref{algo-dGH-dyn} we need to construct the union space $U_{\Psi^{-1}(j)}^X\coloneqq \cup_{i\in \Psi^{-1}(j)}U_i^X$ where all $U_i^X$s are subspaces of $U^X$. By the SGC, $\#\Psi^{-1}(j)=O(\gamma)$. Therefore, by results in Section \ref{sec:union of susbets}, it takes time $O(n\gamma)$ to construct a subspace TDS (subordinate to $T_X$) for representing $U_{\Psi^{-1}(j)}^X$ given the subspace TDSs (subordinate to $T_X$) for $U_i^X$s.

Recall from Appendix \ref{sec:union of susbets} that the union operation consists of three steps where the final step, i.e., the taming process, has the leading time complexity.
In this section, we provide a \emph{refined union operation} via an \emph{improved taming process} for constructing $U_{\Psi^{-1}(j)}^X$ such that the time complexity of this refined union operation is thus reduced to $O(\gamma^2)$. In the sequel, we use the shorthand $U\coloneqq U_{\Psi^{-1}(j)}^X$ and also let $k\coloneqq \#\Psi^{-1}(j)=O(\gamma)$.

Note that in Appendix \ref{sec:union of susbets}, the taming process for the preliminary union $\Tilde{T}_U$ requires pairwise comparisons between all nodes in $\Tilde{T}_U\backslash \cup_{i=1}^k T_{U_i}$ and all nodes in $T_X$. The fact that $\#T_X=O(n)$ explains the $n$ factor in the time complexity bound $O(nk)$ for the taming process. However, to tame $\Tilde{T}_U$, we only need to compare every node in $\Tilde{T}_U\backslash \cup_{i=1}^k T_{U_i}$ with a certain \emph{subset} of nodes in $T_X$. We obtain the improved taming process for $\Tilde{T}_U$ by restricting the pairwise comparisons in this way and by keeping the rest of the taming process unchanged.

Now, we explain how to restrict the pairwise comparison. Let $i_0\coloneqq U^X.\texttt{Order}$ and let $B^X\coloneqq\mathrm{LX}(i_0)$. Since $U^X\in\mathrm{LX}^{(\eps)}$, we have that $U^X\in B^X_{(\eps)}$, i.e., the set of $\eps$-maximal unions of closed balls in $B^X$ (cf. Section \ref{sec:dp-main-text}). Starting from the node $B^X$, we traverse all of its descendants $B$ in order to identify all those for which $B.\texttt{Diameter}\geq B^X.\texttt{Diameter}-2\eps$. We let $\mathrm{L}B^X$ denote the set of all such descendants. Then, in the improved taming process of $\Tilde{T}_U$, we only compare all nodes in $\Tilde{T}_U\backslash \cup_{i=1}^k T_{U_i}$ with all nodes in $\mathrm{L}B^X$. 

By the SGC, $\#\mathrm{L}B^X=O(\gamma)$ and thus the time complexity of the improved taming process described above is at most $O(\gamma k)=O(\gamma^2)$. Therefore, constructing $T_U$ via this taming process has cost at most $O(\gamma^2)$.

\section{Extension of $\ugh$ to ultra-dissimilarity spaces}\label{sec:ext-treegram}
In this section, we will consider the collection $\ums^\mathrm{dis}$ of finite ultra-dissimilarity spaces, which are generalizations of ultrametric spaces (see also \cite{smith2016hierarchical} for a more general notion called ultra-network).
\begin{definition}[Ultra-dissimilarity space]\label{def:ultra-dissimilarity}
An ultra-dissimilarity space is any pair $(X,u_X)$ where $X$ is a \emph{finite} set and $u_X:X\times X\rightarrow \mathbb{R}_{\geq 0}$ satisfies, for all $x,x',x''\in X$:
\begin{enumerate}
    \item[(1)] \textbf{Symmetry:} $u_X(x,x') = u_X(x',x)$,
    \item[(2)] \textbf{Strong triangle inequality:} $u_X(x,x'')\leq \max\left(u_X(x,x'),u_X(x',x'')\right),$
    \item[(3)] \textbf{Definiteness:} $\max\left(u_X(x,x),u_X(x',x')\right)\leq u_X(x,x')$, and the equality takes place if and only if $x=x'$.
\end{enumerate}
We refer to $u_X$ as the \emph{ultra-dissimilarity} on $X$. It is obvious that any finite ultrametric space is an ultra-dissimilarity space. Then, $\ums^\mathrm{fin}\subseteq\ums^\mathrm{dis}$, where $\ums^\mathrm{fin}$ denotes the collection of finite ultrametric spaces.

We say two ultra-dissimilarity spaces $(X,u_X)$ and $(Y,u_Y)$ are \emph{isometric} if there exists a bijective function $f:X\rightarrow Y$ such that for any $x,x'\in X$
\[u_Y(f(x),f(x'))=u_X(x,x'). \]
Such an $f$ is called an \emph{isometry}.
\end{definition}

\begin{remark}[Informal interpretation]
For each $x\in X$, the value $u_X(x,x)$ is regarded as the `birth time' of the point $x$; when $u_X$ is an actual ultrametric on $X$, all points are born at time $0$. The value $u_X(x,x')$ for different points $x$ and $x'$ encodes the time when the two points `merge'.
Note that then condition (3) above can be informally interpreted as encoding the property that two points cannot merge before their respective birth times, and that if they merge at the same time they are born, then they are actually the \emph{same} point.
\end{remark}

Given two ultra-dissimilarity spaces $X$ and $Y$ and any correspondence $R$ between them, without any obstacle, we define $\disna(R)$ in exactly the same way by Equation (\ref{eq:dist}), i.e.,
\[\disna(R)\coloneqq\sup_{(x,y),(x',y')\in R}\Lambda_\infty(u_X(x,x'),u_Y(y,y')).\]

\begin{definition}[$\ugh$ between ultra-dissimilarity spaces]
For two ultra-dissimilarity spaces $X$ and $Y$, we define $\ugh(X,Y)$ by 
\begin{equation}\ugh(X,Y) := \inf_{R}\disna(R).
\end{equation}
\end{definition}

Given an ultra-dissimilarity space $X$, as we did in Section \ref{sec:quotient operation}, we consider a notion of closed equivalence classes $[x]_{\mathfrak{c}(t)}^X\coloneqq\{x'\in X:\,u_X(x,x')\leq t\}$ for any $x\in X$ and $t\geq 0$. Furthermore, let 
\begin{equation}\label{eq:equivalence class treegram}
    [\![x]\!]_{\mathfrak{c}(t)}^X \coloneqq  \left\{
\begin{array}{cl}
[x]_{\mathfrak{c}(t)}^X & \mbox{if $u_X(x,x)\leq t$}\\
\{x\} & \mbox{if $u_X(x,x)>t$}
\end{array}
\right.
\end{equation}
In words, if the `birth time' of $x$ is no larger than $t$, i.e., $u_X(x,x)\leq t$, then $[\![x]\!]_{\mathfrak{c}(t)}^X$ is the same as $[x]_{\mathfrak{c}(t)}^X$. Otherwise, $[\![x]\!]_{\mathfrak{c}(t)}^X$ denotes the singleton $\{x\}$.
We let $X_{\mathfrak{c}(t)}\coloneqq\{[\![x]\!]_{\mathfrak{c}(t)}^X:\,\forall x\in X\}$ and define by $u_{X_{\mathfrak{c}(t)}}$ an ultra-dissimilarity on $X_{\mathfrak{c}(t)}$ given by:

\begin{equation}
    u_{X_{\mathfrak{c}(t)}}\left([\![x]\!]_{\mathfrak{c}(t)}^X ,[\![x']\!]_{\mathfrak{c}(t)}^X \right) \coloneqq  \left\{
\begin{array}{cl}
u_X(x,x') & \mbox{if $[\![x]\!]_{\mathfrak{c}(t)}^X \neq[\![x']\!]_{\mathfrak{c}(t)}^X $, or $x=x'$ and $u_X(x,x)>t$}\\
0 & \mbox{otherwise.}
\end{array}
\right.
\end{equation}

\begin{definition}[Closed quotient on ultra-dissimilarity spaces]\label{def:ultrarightquotient}
Given an ultra-dissimilarity space $(X,u_X)$ and $t\geq 0$, Then, we call $\lc X_{\mathfrak{c}(t)},u_{X_{\mathfrak{c}(t)}}\rc $ the \emph{closed quotient} of $X$ at level $t$.
\end{definition}

We are still using the notation $X_{\mathfrak{c}(t)}$ to denote the resulting quotient space as we did in the case of ultrametric spaces (Definition \ref{def:ultraquotient}) because if $(X,u_X)$ is actually an ultrametric space, then the new definition agrees with the one given previously.

It is obvious that for any ultra-dissimilarity space $X$ and any $t\geq 0$, $X_{\mathfrak{c}(t)}$ is still an ultra-dissimilarity space. Then, the closed quotient gives rise to a map which we call the \emph{$t$-closed quotient operator} $Q_{\mathfrak{c}\lc t\rc}:\mathcal{U}^\mathrm{dis}\rightarrow\mathcal{U}^\mathrm{dis}$ sending $X\in \mathcal{U}^\mathrm{dis}$ to $X_{\mathfrak{c}(t)}\in\mathcal{U}^\mathrm{dis}$.

\begin{theorem}[Structural theorem for $\ugh$ on ultra-dissimilarity spaces]\label{thm:ugh-structn}
For any two finite  ultra-dissimilarity spaces $X$ and $Y$ one has that
\begin{equation*}
    \ugh(X,Y) = \min\left\{t\geq 0:\, \lc X_{\mathfrak{c}(t)},u_{X_{\mathfrak{c}(t)}}\rc \cong \lc Y_{\mathfrak{c}(t)},u_{Y_{\mathfrak{c}(t)}}\rc \right\}.
\end{equation*}
\end{theorem}
\begin{proof}
We first prove a weaker version (with $\inf$ instead of $\min$):
\begin{equation}\label{eq:ughnwinf}
    \ugh(X,Y) = \inf\left\{t\geq 0:\, (X_{\mathfrak{c}(t)} ,u_{X_{\mathfrak{c}(t)} })\cong (Y_{\mathfrak{c}(t)} ,u_{Y_{\mathfrak{c}(t)} })\right\}.
\end{equation}
Suppose first that $X_{\mathfrak{c}(t)}\cong Y_{\mathfrak{c}(t)}$ for some $t\geq 0$, i.e. that there exists an isometry  $f_t:X_{\mathfrak{c}(t)}\rightarrow Y_{\mathfrak{c}(t)}$. Then, we define
\[R_t\coloneqq \left\{(x,y)\in X\times Y:\,[\![y]\!]_{\mathfrak{c}\lc t\rc}^Y= f_t([\![x]\!]_{\mathfrak{c}(t)}^X)\right\}.\] 
Since $f_t$ is bijective, $R_t$ is a correspondence between $X$ and $Y$.

Then, we show that $\disna(R_t)\leq t$, which will imply that $\ugh(X,Y)\leq t$. Choose any $(x,y),(x\p,y\p)\in R_t$. If $u_X(x,x')\leq t$, then $u_X(x,x')\leq \max(t,u_Y(y,y'))$. Otherwise, if $u_X(x,x\p)>t$, we have the following two cases:

\begin{enumerate}
    \item $x=x'$. Then, $[\![x]\!]_{\mathfrak{c}(t)}^X =[\![x']\!]_{\mathfrak{c}(t)}^X=\{x\}$. Thus, $[\![y]\!]_{\mathfrak{c}\lc t\rc}^Y =f_t([\![x]\!]_{\mathfrak{c}(t)}^X)= f_t([\![x']\!]_{\mathfrak{c}(t)}^X)=[\![y']\!]_{\mathfrak{c}\lc t\rc}^Y $. Since $f_t$ is isometry, we have that
    \[u_{Y_{\mathfrak{c}(t)} }\lc[\![y]\!]_{\mathfrak{c}\lc t\rc}^Y ,[\![y]\!]_{\mathfrak{c}\lc t\rc}^Y \rc=u_{X_{\mathfrak{c}(t)} }\lc[\![x]\!]_{\mathfrak{c}(t)}^X ,[\![x]\!]_{\mathfrak{c}(t)}^X \rc)=u_X(x,x)>0.\]
    This implies that 
    \[u_Y(y,y)=u_{Y_{\mathfrak{c}(t)} }\lc[\![y]\!]_{\mathfrak{c}\lc t\rc}^Y ,[\![y]\!]_{\mathfrak{c}\lc t\rc}^Y \rc=u_X(x,x)>t,\]
    and thus $[\![y]\!]_{\mathfrak{c}\lc t\rc}^Y=\{y\}$. Similarly, $[\![y']\!]_{\mathfrak{c}\lc t\rc}^Y=\{y'\}$ and thus $\{y\}=[\![y]\!]_{\mathfrak{c}\lc t\rc}^Y =[\![y']\!]_{\mathfrak{c}\lc t\rc}^Y =\{y'\}$. Then, we have that $y=y'$ and thus $u_Y(y,y')=u_X(x,x')$.
    \item $x\neq x'$. Then, $[\![x]\!]_{\mathfrak{c}(t)}^X \neq [\![x']\!]_{\mathfrak{c}(t)}^X$. Since $f_t$ is an isometry, we have that 
    \[[\![y]\!]_{\mathfrak{c}\lc t\rc}^Y =f_t\lc[\![x]\!]_{\mathfrak{c}(t)}^X \rc\neq f_t\lc[\![x']\!]_{\mathfrak{c}(t)}^X \rc=[\![y']\!]_{\mathfrak{c}\lc t\rc}^Y \]
    and thus $u_Y(y,y\p)=u_{Y_{\mathfrak{c}(t)} }\lc[\![y]\!]_{\mathfrak{c}\lc t\rc}^Y ,[\![y']\!]_{\mathfrak{c}\lc t\rc}^Y \rc=u_{X_{\mathfrak{c}(t)} }\lc[\![x]\!]_{\mathfrak{c}(t)}^X ,[\![x']\!]_{\mathfrak{c}(t)}^X \rc=u_X(x,x\p)$.
\end{enumerate}

Therefore, $u_X(x,x\p)\leq\max(t,u_Y(y,y\p))$. Similarly, $u_Y(y,y\p)\leq\max(t,u_X(x,x\p))$. Then, we have that $\disna(R_t)\leq t$ and thus $\ugh(X,Y) \leq \inf\left\{t\geq 0:\, X_{\mathfrak{c}(t)}\cong Y_{\mathfrak{c}(t)}\right\}.$

Conversely, let $R$ be a correspondence between $X$ an $Y$ and let $t\coloneqq\disna(R)$. Define a map $f:X\rightarrow Y$ by taking $x\in X$ to an arbitrary $y\in Y$ such that $(x,y)\in R$. Consider the induced quotient map $f_t:X_{\mathfrak{c}(t)}\rightarrow Y_{\mathfrak{c}(t)}$, defined by $f_t\lc[\![x]\!]_{\mathfrak{c}(t)}^X \rc=[\![f(x)]\!]_{\mathfrak{c}(t)}^Y$. We now show that $f_t$ is well-defined. For any $(x,y),(x\p,y\p)\in R$ such that $[\![x']\!]_{\mathfrak{c}(t)}^X =[\![x]\!]_{\mathfrak{c}(t)}^X $, we have the following two cases:
\begin{enumerate}
    \item $u_X(x,x')\leq t$. Then, since $\disna(R)=t$, we have that $u_Y(y\p,y)\leq\max\lc t,u_X(x,x') \rc\leq t$.
    \item $u_X(x,x)>t$ and $x=x'$. Then, since $\disna(R)=t$, $u_Y(y,y')=u_X(x,x')=u_X(x,x)>t$. Similarly, $u_Y(y,y)=u_X(x,x)=u_X(x',x')=u_Y(y',y')$. Therefore, $y=y'$ by condition (3) of the definition of ultra-dissimilarity spaces (cf. Definition \ref{def:ultra-dissimilarity}).
\end{enumerate}
Therefore, $[\![y]\!]_{\mathfrak{c}\lc t\rc}^Y =[\![y']\!]_{\mathfrak{c}\lc t\rc}^Y$, which implies that $f_t$ is well-defined. Similarly, the quotient map $g_t:Y_{\mathfrak{c}(t)}\rightarrow X_{\mathfrak{c}(t)}$ induced by a map $g:Y\rightarrow X$ such that $g(y)=x$ where $x\in X$ is chosen such that $(x,y)\in R$ is well-defined. It is clear that $g_t$ is the inverse of $f_t$ and thus $f_t$ is bijective.
Now we show that $f_t$ is an isometry. Choose $[\![x]\!]_{\mathfrak{c}(t)}^X ,[\![x']\!]_{\mathfrak{c}(t)}^X \in X_{\mathfrak{c}(t)}$ and let $y=f(x)$ and $y'=f(x')$. Let $s\coloneqq u_{X_{\mathfrak{c}(t)} }\lc[\![x]\!]_{\mathfrak{c}(t)}^X ,[\![x']\!]_{\mathfrak{c}(t)}^X \rc$. If $[\![x]\!]_{\mathfrak{c}(t)}^X \neq[\![x']\!]_{\mathfrak{c}(t)}^X$, then $s>t$ and thus $u_X(x,x\p)=s$. Since $\disna(R)=t<s$, $u_Y(y,y\p)$ is forced to be equal to $s$ and thus 
\[u_{Y_{\mathfrak{c}(t)}}\lc[\![f(x)]\!]_{\mathfrak{c}(t)}^Y ,[\![f(x')]\!]_{\mathfrak{c}(t)}^Y \rc=u_{Y_{\mathfrak{c}(t)}}\lc[\![y]\!]_{\mathfrak{c}\lc t\rc}^Y ,[\![y']\!]_{\mathfrak{c}\lc t\rc}^Y \rc=s=u_{X_{\mathfrak{c}(t)}}\lc[\![x]\!]_{\mathfrak{c}(t)}^X ,[\![x']\!]_{\mathfrak{c}(t)}^X \rc.\]
If $[\![x]\!]_{\mathfrak{c}(t)}^X =[\![x']\!]_{\mathfrak{c}(t)}^X$, then we have the following two cases.
\begin{enumerate}
    \item $u_X(x,x)\leq t$. Then, $[\![x]\!]_{\mathfrak{c}(t)}^X =[\![x']\!]_{\mathfrak{c}(t)}^X $ implies that $u_X(x,x')\leq t$ and $u_{X_{\mathfrak{c}(t)}}\lc[\![x]\!]_{\mathfrak{c}(t)}^X ,[\![x']\!]_{\mathfrak{c}(t)}^X \rc=0$. Since $(x,f(x))\in R$ and $\disna(R)\leq t$, we have that $u_Y(f(x),f(x))\leq\max(t,u_X(x,x))\leq t$. Then, 
    \[u_{Y_{\mathfrak{c}(t)}}\lc[\![f(x)]\!]_{\mathfrak{c}(t)}^Y ,[\![f(x')]\!]_{\mathfrak{c}(t)}^Y \rc=u_{Y_{\mathfrak{c}(t)}}\lc[\![f(x)]\!]_{\mathfrak{c}(t)}^Y ,[\![f(x)]\!]_{\mathfrak{c}(t)}^Y \rc=0=u_X\lc[\![x]\!]_{\mathfrak{c}(t)}^X ,[\![x']\!]_{\mathfrak{c}(t)}^X\rc. \]
    \item $u_X(x,x)>t$. Then, $x=x'$ and $u_{X_{\mathfrak{c}(t)}}\lc[\![x]\!]_{\mathfrak{c}(t)}^X ,[\![x']\!]_{\mathfrak{c}(t)}^X \rc=u_X(x,x)>t$. Since $(x,f(x))\in R$ and $\disna(R)\leq t$, we have $\Lambda_\infty(u_X(x,x),u_Y(f(x),f(x)))\leq t$, which implies that $u_Y(f(x),f(x))=u_X(x,x)>t$. Then, 
    \begin{align*}
        &u_{Y_{\mathfrak{c}(t)}}\lc[\![f(x)]\!]_{\mathfrak{c}(t)}^Y ,[\![f(x')]\!]_{\mathfrak{c}(t)}^Y \rc=u_{Y_{\mathfrak{c}(t)}}\lc[\![f(x)]\!]_{\mathfrak{c}(t)}^Y ,[\![f(x)]\!]_{\mathfrak{c}(t)}^Y \rc=u_X(x,x)\\
        =&u_{X_{\mathfrak{c}(t)}}\lc[\![x]\!]_{\mathfrak{c}(t)}^X ,[\![x]\!]_{\mathfrak{c}(t)}^X \rc=u_{X_{\mathfrak{c}(t)}}\lc[\![x]\!]_{\mathfrak{c}(t)}^X ,[\![x']\!]_{\mathfrak{c}(t)}^X \rc.
    \end{align*}
\end{enumerate}
Therefore, $f_t$ is an isometry and thus $\ugh(X,Y) \geq \inf\left\{t\geq 0:\, X_{\mathfrak{c}(t)}\cong Y_{\mathfrak{c}(t)}\right\}.$

Now, since $X$ is finite, it is obvious that for each $t\geq 0$, there exists $\eps>0$ such that whenever $s\in[t,t+\eps]$, we have that $X_{\mathfrak{c}(t)}\cong X_{\mathfrak{c}(s)}$. Therefore, the infimum of Equation (\ref{eq:ughnwinf}) is attained which concludes the proof.
\end{proof}

Analogously to the case of ultrametric spaces, the structural theorem (Theorem \ref{thm:ugh-structn}) for $\ugh$ on the collection $\ums^\mathrm{dis}$ of all finite ultra-dissimilarity spaces allows us to devise an algorithm similar to Algorithm \ref{algo-uGH} for computing $\ugh$ between ultra-dissimilarity spaces. The argument for the complexity analysis of Algorithm \ref{algo-uGH} can be adapted to show that the time complexity of computing $\ugh$ on $\ums^\mathrm{dis}$ is still $O(n\log(n))$.

\paragraph{Graphical representations of ultra-dissimilarity spaces} As shown in Theorem \ref{thm:dendroultra}, ultrametric spaces are equivalent to dendrograms. Analogously, ultra-dissimilarity spaces can be viewed as certain objects named \emph{treegrams}. To define treegrams, we first introduce a notion called {subpartitions}: given a set $X$, a partition $P'$ of a subset $X'\subseteq X$ is called a \emph{subpartition}.  We denote by $\mathrm{SubPart}(X)$ the collection of all {subpartitions} of $X$. For any subpartitions $P_1$ and $P_2$, we say $P_1$ is \emph{coarser} than $P_2$ if any block in $P_2$ is contained in some block in $P_1$.

\begin{example}[Examples of subpartitions]
\begin{enumerate}
    \item Note that the empty set $P=\emptyset$ is a subpartition of any set $X$.
    \item Given a finite set $X$, let $P=\{B_1,\ldots,B_n\}$ be a partition. Then, for any non-empty subset $X'\subseteq X$, we obtain a subpartition $P|_{X'}$ by restricting $P$ to $X'$ as follows: $P|_{X'}=\{B_1\bigcap X',\ldots,B_n\bigcap X'\}\backslash\{\emptyset\}.$ 
\end{enumerate}
\end{example}

\begin{definition}[Treegrams]\label{def:treegram}
A treegram $\theta_X$ over a finite set $X$ is a function $\theta_X:[0,\infty)\rightarrow \mathrm{SubPart}(X)$ satisfying the following conditions:
\begin{enumerate}
    \item[(1)] For $0\leq s<t$, $\theta_X(t)$ is coarser than $\theta_X(s)$.
    \item[(2)] There exists $t_X>0$ such that $\theta_X(t_X)=\{X\}.$
    \item[(3)] For any $r\geq0$, there exists $\eps>0$ such that $\theta_X(r)=\theta_X(t)$ for $t\in[r,r+\eps].$
    \item[(4)] For each $x\in X$, there exists $t\geq 0$ such that $\{x\}\in \theta_X(t)$ is a block.
\end{enumerate}
\end{definition}

Our definition is a slight modification of treegrams defined in {\cite{smith2016hierarchical,kim2018formigrams}} where the domain of treegrams is the entire real line $\mathbb R$ instead of $\mathbb R_{\geq 0}$.

\begin{figure}[ht]
    \centering
    \includegraphics[width=0.5\textwidth]{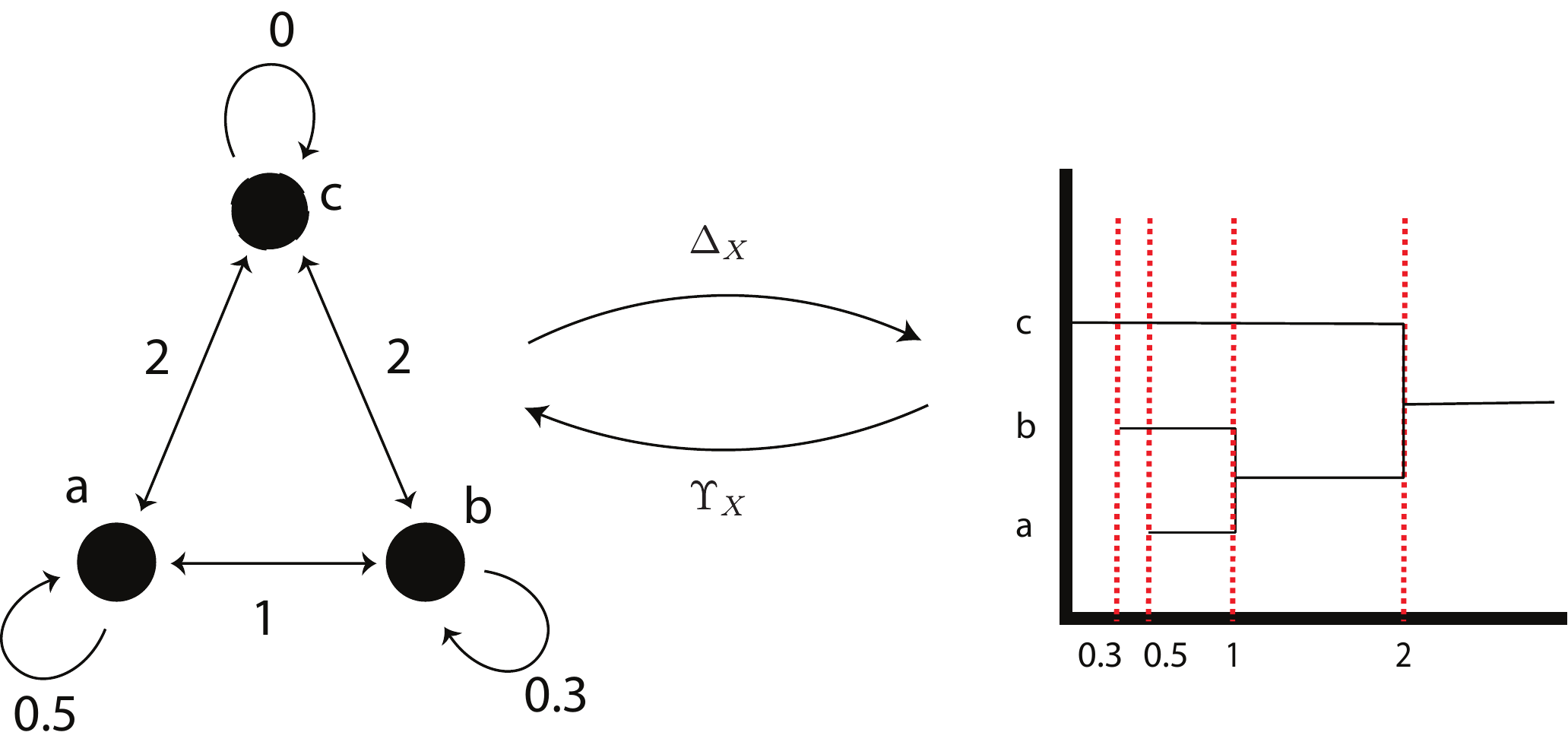}
    \caption{\textbf{Treegrams and ultra-dissimilarity spaces}.}
    \label{fig:ultra-tree}
\end{figure}

Fix a finite set $X$, denote by $\mathcal{U}^\mathrm{dis}(X)$ the collection of all ultra-dissimilarities over $X$ and denote by $\mathcal{T}(X)$ the collection of all treegrams over $X$. We define a map $\Delta_X:\mathcal{U}^\mathrm{dis}(X)\rightarrow\mathcal{T}(X)$ by sending $u_X$ to a treegram $\theta_X$ as follows: for each $t\geq 0$, let $\hat{X}_t\coloneqq\{x\in X:\,u_X(x,x)\leq t\}$ and let $\theta_X(t)\coloneqq\left\{[x]_{\mathfrak{c}(t)}^X:\,x\in \hat{X}_t\right\}$. Note in particular that $\theta_X$ satisfies condition (4) of Definition \ref{def:treegram} due to the definiteness of $u_X$ (cf. condition (3) in Definition \ref{def:ultra-dissimilarity}). Conversely, we define a map $\Upsilon_X:\mathcal{T}(X)\rightarrow\mathcal{U}^\mathrm{dis}(X)$ as follows. Let $\theta_X\in\mathcal{T}(X) $. Then, we define an ultra-dissimilarity $u_X\coloneqq\Upsilon_X(\theta_X)$ on $X$ by: 
\[u_X(x,x')\coloneqq\inf\left\{t\geq 0:\, [x]_{t}^{\theta_X}=[x']_{t}^{\theta_X}\right\},\quad \forall x,x'\in X,\]
where $[x]_{t}^{\theta_X}\in \theta_X(t)$ denotes the block containing $x$. Note that definiteness of $u_X$ follows from condition (4) of Definition \ref{def:treegram}. Then, in analogy to Theorem \ref{thm:dendroultra}, we have the following theorem. See Figure \ref{fig:ultra-tree} for an illustration.

\begin{theorem}\label{thm:treegram-ultra-equiv}
Given any finite set $X$, $\Delta_X:\mathcal{U}^\mathrm{dis}(X)\rightarrow\mathcal{T}(X)$ is bijective with inverse $\Upsilon_X:\mathcal{T}(X)\rightarrow\mathcal{U}^\mathrm{dis}(X)$.
\end{theorem}

{It is obvious that the collection of dendrograms $\mathcal{D}(X)$ over $X$ is a proper subset of $\mathcal{T}(X)$ and that the collection of ultrametrics $\mathcal{U}(X)$ over $X$ is a proper subset of $\mathcal{U}^\mathrm{dis}(X)$. Then, Theorem \ref{thm:treegram-ultra-equiv} is actually an extension/generalization of Theorem \ref{thm:dendroultra}.}

Via Theorem \ref{thm:treegram-ultra-equiv}, one can easily represent an ultra-dissimilarity space via a treegram. In particular, see Figure \ref{fig:simp-treegrams} for an illustration of the closed quotient operator via treegrams.

\begin{figure}
    \centering
    \includegraphics[width=0.8\textwidth]{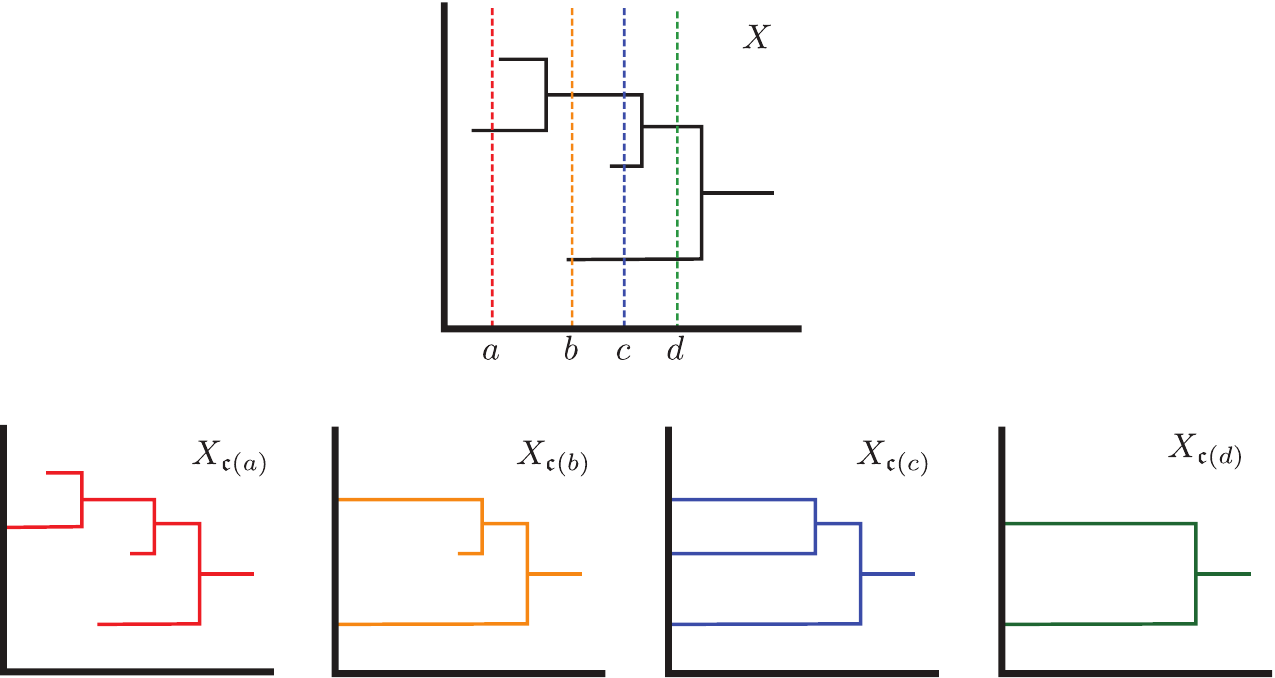}
    \caption{\textbf{Illustration of Definition \ref{def:ultrarightquotient}.} We represent a 4-point ultra-dissimilarity space $X$ as a treegram in the first row of the figure. The second row shows the treegrams corresponding to $X_{\mathfrak{c}(t)} $ at different levels $t$.
    }
    \label{fig:simp-treegrams}
\end{figure}

\section{Relegated proofs}\label{sec:proofs}

\subsection{Proofs from Section \ref{sec:dghp}}

\begin{proof}[Proof of Proposition \ref{prop:dgh-dlp-eq}]
For any correspondence $R$ between $X$ and $Y$, we have
\begin{align*}
    \dis(R,u_X,u_Y)&=\sup_{(x,y),(x',y')\in R}|u_X(x,x')-u_Y(y,y')|\\
    &=\sup_{(x,y),(x',y')\in R}\left|\left((u_X)^\frac{1}{p}(x,x')\right)^p-\left((u_Y)^\frac{1}{p}(y,y')\right)^p\right|\\
    &=\left(\disp\left(R,(u_X)^\frac{1}{p},(u_Y)^\frac{1}{p}\right)\right)^p
\end{align*}
Therefore, by Equation (\ref{eq:dgh-distortion}) and Equation (\ref{eq:dgh-p-distortion}), we have 
\[\dgh(X,Y)=\left(\dlp\left(S_\frac{1}{p}(X),S_\frac{1}{p}(Y)\right)\right)^p.\]

Similarly, $\disp\lc R,u_X,u_Y\rc=\lc\dis\lc R,(u_X)^p,(u_Y)^p\rc\rc^\frac{1}{p}$  and thus $ \dlp(X,Y) = \lc\dgh(S_p(X),S_p(Y))\rc^\frac{1}{p}.$
\end{proof}

\begin{proof}[Proof of Theorem \ref{thm:ugh-dual}]
Let $Z$ be an ultrametric space such that there exist isometric embeddings $\varphi_X:X\hookrightarrow Z$ and $\varphi_Y:Y\hookrightarrow Z$. Let $\eta\coloneqq d_\mathrm{H}^Z(X,Y)$, where we identify $X$ and $Y$ with $\varphi_X(X),\varphi_Y(Y)\subseteq Z$, respectively. Define $R\coloneqq\{(x,y)\in X\times Y:\,u_Z(x,y)\leq\eta\}.$ That $R$ is a correspondence between $X$ and $Y$ follows from the condition that $d_\mathrm{H}^Z(X,Y)=\eta$ and compactness of $X$ and $Y$. Now, consider any $(x,y),(x',y')\in R$. Without loss of generality, we assume that $u_X(x,x')\geq u_Y(y,y')$. If $u_X(x,x')=u_Y(y,y')$, then $\Lambda_\infty(u_X (x,x'),u_Y (y,y'))=0$. So we further assume that $u_X(x,x')> u_Y(y,y')$. Then, 
\begin{align*}
    &\Lambda_\infty(u_X (x,x'),u_Y (y,y')) = u_X(x,x')=u_Z(x,x') \\
    \leq &\max(u_Z (x,y),u_Z(y,y'),u_Z(y',x'))\\
    \leq &\max(u_Z(x,y),u_Z(x',y'))\leq\eta.
\end{align*}
In the second inequality we used the assumption that $u_X(x,x')> u_Y(y,y')$.
Thus, by taking supremum over all pairs $(x,y),(x',y')\in R$, one has $\disna(R)\leq \eta.$ Then, we obtain that $\ugh(X,Y)=\inf_{R}\disna(R)\leq \dghp\infty(X,Y).$

For the reverse inequality, let $R$ be an arbitrary correspondence between $X$ and $Y$. Let $\eta\coloneqq\disna(R)$. Define a function $u:X\sqcup Y\times X\sqcup Y\rightarrow\R_{\geq 0}$ as follows:
\begin{enumerate}
    \item $u|_{X\times X}\coloneqq u_X$ and $u|_{Y\times Y}\coloneqq u_Y$;
    \item for any $(x,y)\in X\times Y$, $u(x,y)\coloneqq\inf_{(x',y')\in R}\max(u_X(x,x'), u_Y (y',y),\eta)$;
    \item for any $(y,x)\in Y\times X$, $u(y,x)\coloneqq u(x,y)$.
\end{enumerate}
Now we show that $u $ is an ultrametric on the disjoint union $X\sqcup Y$. Because of the symmetric roles of $X$ and $Y$, we only need to verify the following two cases:
\begin{description}
\item[Case 1:] $\forall x,x\p\in X$ and $\forall y\in Y$, $ u(x,y)\leq \max(u(x,x\p),u(x\p,y))$;

\item[Case 2:] $\forall x,x\p\in X$ and $\forall y\in Y$, $ u(x,x\p)\leq\max( u(x,y),u(x\p,y))$.
\end{description}

For Case 1,
\begin{align*}
 \max(u(x,x\p),u(x\p,y)) & = \max\left( u(x,x\p),\inf_{(x_1,y_1)\in R}\max( u_X (x',x_1), u_Y (y_1,y),\eta)\right)\\
& = \inf_{(x_1,y_1)\in R}\max\left( u(x,x\p),u_X (x',x_1),u_Y (y_1,y),\eta\right)\\
& \geq \inf_{(x_1,y_1)\in R}\max\left( u_X(x,x_1), u_Y (y_1,y),\eta\right)\\
& =  u (x,y). 
\end{align*}

For Case 2, 
\begin{align*}
 &\max\left( \inf_{(x_1,y_1)\in R}\max\left( u_X(x,x_1), u_Y  (y_1,y),\eta \right),\inf_{(x_2,y_2)\in R}\max\left( u_X  (x_2,x'), u_Y  (y_2,y), \eta \right)\right)\\
 = & \inf_{(x_1,y_1),(x_2,y_2)\in R}\max\left( u_X  (x,x_1), u_Y  (y_1,y), \eta ,  u_X  (x_2,x'),  u_Y  (y_2,y), \eta \right)\\
 \geq & \inf_{(x_1,y_1),(x_2,y_2)\in R}\max\left( u_X  (x,x_1), u_X  (x_2,x'),  u_Y  (y_1,y_2), \eta \right)\\
 \geq & \inf_{(x_1,y_1),(x_2,y_2)\in R}\max\left( u_X  (x,x_1),u_X  (x_2,x'), u_X  (x_1,x_2)\right)\geq u_X (x,x\p)=u(x,x').
\end{align*}
The second inequality above follows from the fact that $\Lambda_\infty(u_X(x_1,x_2),u_Y(y_1,y_2))\leq\disna(R)= \eta$ and Equation (\ref{eq:lambda-infty-other-def}).

Note that $u(x,y)=\eta$ for every $(x,y)\in R$. Therefore, 
$ \dghp\infty(X,Y)\leq d_\mathrm{H}^{(X\sqcup Y,u)}(X,Y)=\eta$. This implies that $ \dghp\infty(X,Y) \leq \inf_{R}\disna(R)=\ugh(X,Y).$
\end{proof}

\subsection{Proofs from Section \ref{sec:computing-dgh}}
\subsubsection{Proof of Theorem \ref{thm:complexityrecdgh} (Time complexity of Algorithm $\mathbf{FindCorrRec}$ (Algorithm \ref{algo-dGH-rec}))}\label{sec:proof-rec}

\begin{lemma}[Inheritance of the FGC]
Let $X$ and $Y$ be finite ultrametric spaces such that $X,Y\in\mathcal{U}_1(\eps,\gamma)$. Assume that $0\leq \diam(Y)-\diam(X)\leq \eps$ and that $\diam(Y)>\eps$. Write $X_{\mathfrak{o}\lc \delta_\eps(Y)\rc}=\{X_i\}_{i\in [N_X]}$ and $Y_{\mathfrak{o}\lc \delta_0(Y)\rc}=\{Y_j\}_{j\in [N_Y]}$. Then, given any surjection $\Psi:[N_X]\twoheadrightarrow [N_Y]$, for each $j\in[N_Y]$, we have that $X_{\Psi^{-1}(j)}\in\mathcal{U}_1(\eps,\gamma)$ and $Y_j\in\mathcal{U}_1(\eps,\gamma)$.
\end{lemma}
\begin{proof}
For notational simplicity, let $\delta_0\coloneqq\delta_0(Y)$ and let $\delta_\eps\coloneqq\delta_\eps(Y)$.

We first prove that $Y_j\in\mathcal{U}_1(\eps,\gamma)$. For any $y\in Y_j$, note that $Y_j=[y]_{\mathfrak{o}(\delta_0)}^{Y}$. Given any $t\geq \eps$, if $t\leq \delta_0$, then we easily see that 
$[y]_{\mathfrak{o}(t-\eps)}^{Y_j}=[y]_{\mathfrak{o}(t-\eps)}^{Y}.$
Then, 
\[\#[y]_{\mathfrak{c}(t)}^{Y_j}\leq\underbrace{\#[y]_{\mathfrak{c}(t)}^{Y}\leq\gamma\cdot\#[y]_{\mathfrak{o}(t-\eps)}^{Y}}_{\text{FGC}}=\gamma\cdot\#[y]_{\mathfrak{o}(t-\eps)}^{Y_j}. \]
Otherwise, if $t> \delta_0$, we then have that $[y]_{\mathfrak{c}(t)}^{Y_j}=Y_j=[y]_{\mathfrak{c}(\delta_0)}^{Y_j}$. Then,
\[\#[y]_{\mathfrak{c}(t)}^{Y_j}=\#[y]_{\mathfrak{c}(\delta_0)}^{Y_j}\leq\gamma\cdot\#[y]_{\mathfrak{o}(\delta_0-\eps)}^{Y_j}\leq \gamma\cdot\#[y]_{\mathfrak{o}(t-\eps)}^{Y_j}, \]
where the first inequality follows from the previous case $t\leq\delta_0$. Therefore, $Y_j\in\mathcal{U}_1(\eps,\gamma)$.

We then prove that $X_{\Psi^{-1}(j)}\in\mathcal{U}_1(\eps,\gamma)$. For $X_{\Psi^{-1}(j)}$, we choose any $x\in X_{\Psi^{-1}(j)}$ and any $t\geq \eps$. If $t\leq \diam(X)$, then it is easy to see that $t-\eps<\delta_\eps$. Therefore, $[x]_{\mathfrak{o}(t-\eps)}^{X_{\Psi^{-1}(j)}}=[x]_{\mathfrak{o}(t-\eps)}^X$ and thus
\[\#[x]_{\mathfrak{c}(t)}^{X_{\Psi^{-1}(j)}}\leq\underbrace{\#[x]_{\mathfrak{c}(t)}^{X}\leq\gamma\cdot\#[x]_{\mathfrak{o}(t-\eps)}^{X}}_{\text{FGC}}= \gamma\cdot\#[x]_{\mathfrak{o}(t-\eps)}^{X_{\Psi^{-1}(j)}}. \]
If $t>\diam(X)$, then we have that $[x]_{\mathfrak{c}(t)}^{X_{\Psi^{-1}(j)}}=X_{\Psi^{-1}(j)}=[x]_{\mathfrak{c}(\diam(X))}^{X_{\Psi^{-1}(j)}}$. Then,
\[\#[x]_{\mathfrak{c}(t)}^{X_{\Psi^{-1}(j)}}=\#[x]_{\mathfrak{c}(\diam(X))}^{X_{\Psi^{-1}(j)}}\leq\gamma\cdot\#[x]_{\mathfrak{o}(\diam(X)-\eps)}^{X_{\Psi^{-1}(j)}}\leq \gamma\cdot\#[x]_{\mathfrak{o}(t-\eps)}^{X_{\Psi^{-1}(j)}}, \]
where the first inequality follows from the previous case $t\leq\diam(X)$. Therefore, $X_{\Psi^{-1}(j)}\in\mathcal{U}_1(\eps,\gamma)$.

\end{proof}

In this way, the inputs to each subproblem encountered while running Algorithm \ref{algo-dGH-rec} will satisfy the first $(\eps,\gamma)$-growth condition. This justifies the analysis presented below regarding the number and the size of subproblems.

\begin{proof}[Proof of Theorem \ref{thm:complexityrecdgh}]
We are going to invoke the master theorem \cite{cormen2009introduction} in order to analyze the complexity of our recursive algorithm (Algorithm \ref{algo-dGH-rec}).

\paragraph{Number of subproblems} By the FGC (or Remark \ref{rmk:numberofbranch}), we have that $\max(N_X,N_Y)\leq \gamma$. There will be at most $\gamma^\gamma$ surjections  $\Psi:[N_X]\twoheadrightarrow [N_Y]$. For each such a surjection, Algorithm \ref{algo-dGH-rec} inspects $N_Y$ ($\leq \gamma$) subproblems. Therefore, there are at most $\gamma^{\gamma+1}$ subproblems.

\paragraph{Size of a subproblem} Fix a surjection $\Psi:[N_X]\twoheadrightarrow [N_Y]$. For each $k\in [N_Y]$, we write $Y_k\coloneqq[y_k]_{\mathfrak{o}\lc \delta_0\rc}^Y$ for some $y_k\in Y$. Then, for a fixed $j\in [N_Y]$, we have that
\[\#Y = \#[y_j]_{\mathfrak{o}\lc \delta_0\rc}^Y+\sum_{k\in [N_Y]\backslash\{j\}}\#[y_k]_{\mathfrak{o}\lc \delta_0\rc}^Y\geq \#Y_j+\frac{M-1}{\gamma}\cdot\#Y,\]
where the last inequality follows from the FGC and the fact that $[y_k]_{{\delta_0+\eps}}^Y=Y$ for each $k\in [N_Y]$. Therefore, $\#Y_j\leq \left(1-\frac{M-1}{\gamma}\right)\cdot\#Y$. Now, regarding $X_{\Psi^{-1}(j)}$, since $\Psi$ is a surjection and $M\geq 2$, there exists $i\notin\Psi^{-1}(j) $ such that $X_i\bigcap X_{\Psi^{-1}(j)}=\emptyset$. Assume $X_i=[x_i]_{\mathfrak{o}\lc \delta_\eps\rc}^X$ for some $x_i\in X$. Then, \[\#[x_i]^X_{\mathfrak{o}\lc \delta_\eps\rc}\geq\frac{\#[x_i]^X_{\mathfrak{c}\lc \delta_\eps+\eps\rc}}{\gamma} \geq\frac{\#[x_i]^X_{\mathfrak{o}\lc \delta_\eps+\eps\rc}}{\gamma} \geq\frac{\#[x_i]^X_{\mathfrak{c}\lc \delta_\eps+2\eps\rc}}{\gamma^2}=\frac{\#X}{\gamma^2},\]
where the last equality follows from the fact that $\delta_\eps+2\eps=\diam(Y)+\eps>\diam(X)$. Therefore, 
\[\#X_{\Psi^{-1}(j)}\leq \#X-\#X_i\left(1-\frac{1}{\gamma^2}\right)\#X.\]
When $\gamma\geq 1$, we have that $1-\frac{1}{\gamma^2}\geq 1-\frac{M-1}{\gamma}$. So, $\max(\#X_{\Psi^{-1}(j)},\#Y_j)\leq \left(1-\frac{1}{\gamma^2}\right)n$ and thus the size of a subproblem is bounded above by $\left(1-\frac{1}{\gamma^2}\right)n$.

\paragraph{Base case complexity and other work} If $(X,Y,\eps)$ is one of the base cases, it takes time at most $O(n^2)$ in order to either directly generate a correspondence or 0. In order to utilize the results from the subproblems, we need at most $O(n^2)$ time to construct distance matrices $u_X$ and $u_Y$ from TDSs in order to implement the unions $X_{\Psi^{-1}(j)}$ (cf. Appendix \ref{sec:data-structure}). It then takes time at most $O(n^2)$ in total to construct all the unions $X_{\Psi^{-1}(j)}$, $R_j$ and $R$, and to transpose $R$ when/if needed.

\paragraph{Conclusion} Denote by $W(n)$ the time complexity of the algorithm where $n=\max(\#X,\#Y)$.  Then,
\[W(n)\leq \gamma^{\gamma+1}\cdot W\left(\frac{n}{\gamma^2/(\gamma^2-1)}\right)+O\left(n^2\right).\]

Since by assumption that $\gamma\geq 2$, the critical exponent $\log_{\frac{\gamma^2}{\gamma^2-1}}{\gamma^{\gamma+1}}$ is strictly greater than 2. Therefore, by the master theorem we have that
$ W(n)=O\left(n^{(\gamma+1)\log_{\mathrm{b}(\gamma)}{\gamma}}\right).$ This concludes the proof.
\end{proof}
\subsubsection{Proof of Proposition \ref{prop:smallcorr}}\label{sec: small-algo-prop}
In the following lemmas, \emph{we will always assume that $X$ and $Y$ are finite ultrametric spaces and that there exists $\eps\geq 0$ such that $\diam(X)>\eps$ and $\diam(Y)\leq \eps$. }

\begin{lemma}\label{lm:injlm1}
There exists an $\eps$-correspondence between $X$ and $Y$ if and only if there exists an $\eps$-correspondence between $X\ct{\eps}$ and $Y$. 
\end{lemma}

\begin{proof}
Suppose that $R$ is an $\eps$-correspondence between $X$ and $Y$. Then, we define the set $R_\eps\subseteq X\ct{\eps}\times Y$ as follows:
\[ R_\eps\coloneqq\left\{\left([x]_{\mathfrak{c}(\eps)}^X,y\right)\in X\ct{\eps}\times Y:\,(x,y)\in R\right\}.\]
It is easy to verify that $R_\eps$ is a correspondence between $X\ct{\eps}$ and $Y$. Now, for any two pairs $\left([x]_{\mathfrak{c}(\eps)}^X,y\right),\lc[x']_{\mathfrak{c}(\eps)}^X,y'\rc\in R_\eps$, we have the following:
\begin{align*}
    \left|u_{X\ct{\eps}}\left([x]_{\mathfrak{c}(\eps)}^X,[x']_{\mathfrak{c}(\eps)}^X\right)-u_Y(y,y')\right|&=
\begin{cases}
\left|u_X(x,x')-u_Y(y,y')\right| & \mbox{if $[x]_{\mathfrak{c}(\eps)}^X\neq[x']_{\mathfrak{c}(\eps)}^X $}\\
u_Y(y,y') & \mbox{if $[x]_{\mathfrak{c}(\eps)}^X=[x']_{\mathfrak{c}(\eps)}^X $}
\end{cases}\\
&\leq \begin{cases}
\dis(R) & \mbox{if $[x]_{\mathfrak{c}(\eps)}^X\neq[x']_{\mathfrak{c}(\eps)}^X $}\\
\diam(Y) & \mbox{if $[x]_{\mathfrak{c}(\eps)}^X=[x']_{\mathfrak{c}(\eps)}^X $}
\end{cases}\leq \eps
\end{align*}
Therefore, $\dis(R_\eps)\leq\eps$ and thus $R_\eps$ is an $\eps$-correspondence between $X\ct{\eps}$ and $Y$.

Now for the converse, suppose that there exists an $\eps$-correspondence $R_\eps$ between $X\ct{\eps}$ and $Y$. Then, we define $R\subseteq X\times Y$ as follows:
\[R\coloneqq\left\{(x,y)\in X\times Y:\left([x]_{\mathfrak{c}(\eps)}^X,y\right)\in R\right\}.\]
It is easy to verify that $R$ is a correspondence between $X$ and $Y$. Then, for any $(x,y),(x',y')\in R$, if $[x]_{\mathfrak{c}(\eps)}^X\neq[x']_{\mathfrak{c}(\eps)}^X $, we have
\[|u_{X}(x,x')-u_Y(y,y')|\leq\left|u_{X\ct{\eps}}\left([x]_{\mathfrak{c}(\eps)}^X,[x']_{\mathfrak{c}(\eps)}^X\right)-u_Y(y,y')\right|\leq \dis(R_\eps)\leq \eps.\]
If $[x]_{\mathfrak{c}(\eps)}^X=[x']_{\mathfrak{c}(\eps)}^X $, then $u_X(x,x')\leq \eps$. Moreover, $u_Y(y,y')\leq \diam(Y)\leq \eps$. Then, $|u_{X}(x,x')-u_Y(y,y')|\leq \eps$. Therefore, $\dis(R)\leq\eps$ and thus $R$ is an $\eps$-correspondence between $X$ and $Y$. This concludes the proof.
\end{proof}

\begin{lemma}\label{lm:injlm2}
For any $\eps$-correspondence $R_\eps$ between $X\ct{\eps}$ and $Y$, there exists a surjection $\psi$ from $Y$ to $X\ct{\eps}$ such that $R_\eps=\{(\psi(y),y):\,y\in Y\}$.
\end{lemma}
\begin{proof}
We first show that if $\left([x]_{\mathfrak{c}(\eps)}^X,y\right),([x']_{\mathfrak{c}(\eps)}^X,y)\in R_\eps$, then $[x]_{\mathfrak{c}(\eps)}^X=[x']_{\mathfrak{c}(\eps)}^X$. Otherwise, suppose that $[x]_{\mathfrak{c}(\eps)}^X\neq[x']_{\mathfrak{c}(\eps)}^X$, which is equivalent to the condition that $u_X(x,x')>\eps$. Then, 
\[\eps\geq \dis(R_\eps)\geq \big|u_{X\ct{\eps}}\left([x]_{\mathfrak{c}(\eps)}^X,[x']_{\mathfrak{c}(\eps)}^X\right)-u_Y(y,y)\big|=u_X(x,x')>\eps, \]
which is a contradiction!
Then, $R_\eps$ naturally induces a well-defined map $\psi:Y\rightarrow X\ct{\eps}$ taking $y\in Y$ to $[x]_{\mathfrak{c}(\eps)}^X$ such that $\left([x]_{\mathfrak{c}(\eps)}^X,y\right)\in R_\eps$. It is easy to check that $\psi$ is surjective and that $R_\eps=\{(\psi(y),y):\,y\in Y\}$.
\end{proof}

Recall from Section \ref{sec:additive approx} that $\mathrm{sep}(X)\coloneqq\min\{d_X(x,x'):\,x,x'\in X\text{ and }x\neq x'\}$ denotes the \emph{separation} of a finite metric space $(X,d_X)$.

\begin{lemma}\label{lm:injlm3}
Assume that $\mathrm{sep}(X)>\eps$. Then any injective map $\varphi:X\rightarrow Y$ with $\dis(\varphi)\leq\eps$ induces an $\eps$-correspondence between $X$ and $Y$.
\end{lemma}
\begin{proof}
Suppose $X=\{x_1,\ldots,x_n\}$ and $\mathrm{im}(\varphi)=\{y_1,\ldots,y_n\}$ where $y_i=\varphi(x_i)$ for every $i=1,\ldots,n$. For any $y\in Y$, define 
\[i_y\coloneqq\min\left\{\mathop{\mathrm{argmin}}_{j=1,\ldots,n}u_Y(y,y_j)\right\}. \]
Obviously, $i_{y_j}=j$. Then, we define $R\subseteq X\times Y$ as follows:
\[ R\coloneqq\{(x_{i_y},y)\in X\times Y:\,\forall y\in Y\}.\]
It is easy to check that $R$ is a correspondence. Now, we verify that $\dis(R)\leq\eps$. Let $(x_i,y),(x_j,y')\in R$, where $i=i_y$ and $j=i_{y'}$. If $i= j$, then $|u_X(x_i,x_i)-u_Y(y,y')|=u_Y(y,y')\leq \diam(Y)\leq\eps.$
Now assume $i\neq j$. It is obvious that $(x_i,y_i),(x_j,y_j)\in R$ and thus $u_X(x_i,x_j)-u_Y(y_i,y_j)\leq\dis(\varphi)\leq\eps$. Since $u_X(x_i,x_j)-u_Y(y,y')\geq\mathrm{sep}(X)-\diam(Y)\geq 0$, the inequality $|u_X(x_i,x_j)-u_Y(y,y')|\leq\eps$ follows from the following observation:

\begin{claim}\label{clm:injob}
For $y,y'\in Y$, if $i_y\neq i_{y'}$, then $u_Y(y,y')\geq u_Y(y_i,y_j)$ where $i\coloneqq i_y$ and $j\coloneqq i_{y'}$.
\end{claim}

\begin{proof}[Proof of Claim \ref{clm:injob}]
Suppose otherwise that $u_Y(y,y')<u_Y(y_i,y_j)$. If $u_Y(y_i,y)\leq u_Y(y,y')$, then 
\[u_Y(y',y_i)\leq\max(u_Y(y,y'),u_Y(y,y_i))\leq u_Y(y,y').\]
By definition of $j=i_{y'}$, we have that $u_Y(y_j,y')\leq u_Y(y_i,y')\leq u_Y(y,y')$. Then, $u_Y(y_i,y_j)\leq \max(u_Y(y_i,y'),u_Y(y',y_j))\leq u_Y(y,y')$, which is a contradiction. Therefore, $u_Y(y_i,y)>u_Y(y,y')$ and similarly $u_Y(y_j,y')>u_Y(y,y')$. Then, by the strong triangle inequality we have that $u_Y(y,y_i)=u_Y(y',y_i)$ and $u_Y(y,y_j)=u_Y(y',y_j)$. By definition of $i=i_y$ and $j=i_{y'}$, we have that 
\[u_Y(y,y_j)=u_Y(y',y_j)\leq u_Y(y',y_i)=u_Y(y,y_i), \]
which implies that $j\in \mathop{\mathrm{argmin}}_{k=1,\ldots,n}u_Y(y,y_k)$ and thus $j>i$. However, we can similarly prove that $i>j$, which is a contradiction. Therefore, $u_Y(y,y')\geq u_Y(y_i,y_j)$.
\end{proof}
\end{proof}

\begin{proof}[Proof of Proposition \ref{prop:smallcorr}]
By Lemma \ref{lm:injlm1}, we only need to prove that there exists an $\eps$-correspondence between $X\ct{\eps}$ and $Y$ if and only if there exists an injective map $\varphi:X\ct{\eps}\rightarrow Y$ with $\dis(\varphi)\leq \eps.$

Assuming the existence of such a correspondence $R_\eps$, then by Lemma \ref{lm:injlm2}, there exists a surjection $\psi:Y\twoheadrightarrow X\ct{\eps}$ such that $R_\eps=\{(\psi(y),y):y\in Y\}$. Then, we construct an injective map $\varphi:X\ct{\eps}\rightarrow Y$ by mapping each $[x]_{\mathfrak{c}(\eps)}^X\in X\ct{\eps}$ to $y$, where $y$ is arbitrarily chosen from $\psi^{-1}\left([x]_{\mathfrak{c}(\eps)}^X \right)$. Then, $\dis(\varphi)\leq \dis(\psi)\leq \eps.$

Now, assume that there exists an injective map $\varphi:X\ct{\eps}\rightarrow Y$ with $\dis(\varphi)\leq \eps.$ Obviously, we have $\mathrm{sep}(X\ct{\eps})>\eps$, and thus, by Lemma \ref{lm:injlm3}, there exists a correspondence $R_\eps$ between $X\ct{\eps}$ and $Y$ with $\dis(R_\eps)\leq \eps$.
\end{proof}
\subsubsection{Proof of Remark \ref{rmk:rho-eps-delta-eps}}\label{sec:proof of remark rho}
We first establish the following characterization of elements in $ V_X^{(\eps)}$.

\begin{lemma}\label{lm:char-LX-eps}
Let $X$ be a finite ultrametric space and let $\eps\geq0$. Then, a subset $U^X\subseteq X$ belongs to $ V_X^{(\eps)}$ if and only if $U^X$ contains all $x\in X$ such that $u_X\lc x,U^X\rc<\diam\lc U^X\rc -2\eps$, where $u_X\lc x,U^X\rc\coloneqq\min\left\{u_X(x,x'):\,x'\in U^X\right\}$.
\end{lemma}
\begin{proof}
If $U^X\in V_X^{(\eps)}$, then $U^X$ is an $\eps$-maximal union of closed balls in $B^X\subseteq X$ (cf. Section \ref{sec:dp-main-text}): write $B^X_{\mathfrak{o}\lc \rho_\eps\lc B^X\rc\rc}=\{B_1^X,\ldots,B_N^X\}$, where $\rho_\eps\lc B^X\rc\coloneqq\max\lc\diam\lc B^X\rc-2\eps,0\rc$; then, $U^X=\bigcup_{i\in I}B^X_i$ for some non-empty $I\subseteq\{1,\ldots,N\}$ and $U^X$ satisfies that $\diam\lc U^X\rc=\diam\lc B^X\rc$. Without loss of generality, we assume that $\diam\lc U^X\rc -2\eps>0$.  Given any $x\in X$ such that $u_X\lc x,U^X\rc<\diam\lc U^X\rc -2\eps$, there exists $x'\in U^X$ such that $u_X(x,x')<\diam\lc U^X\rc -2\eps=\diam(B^X)-2\eps$. Since $x'\in U^X\subseteq B^X$, $x'\in B^X_i$ for some $i\in I$. Then,
\[B^X_i=[x']^X_{\mathfrak{o}\lc \rho_\eps\lc B^X\rc\rc}=\{x''\in X:\,u_X(x',x'')<\diam(B^X)-2\eps \}.\]
Therefore, $x\in B^X_i\subseteq U^X$.

Now, let $U^X\subseteq X$ be a subset containing all $x\in X$ such that $u_X\lc x,U^X\rc<\diam\lc U^X\rc -2\eps$. For any $x\in U^X$, consider the ball $B^X\coloneqq [x]^X_{\mathfrak{c}(\delta) }$, where $\delta\coloneqq \diam\lc U^X\rc$. It is obvious that $U^X\subseteq B^X$ and $\diam(U^X)=\diam(B^X)$. If $\diam\lc U^X\rc \leq 2\eps$, then $\rho_\eps(B^X)=0$. Therefore, $B^X_{\mathfrak{o}\lc \rho_\eps\lc B^X\rc\rc}=\{\{x\}:\,x\in B^X\}$ and thus obviously, $U^X\in B^X_{(\eps)}\subseteq  V_X^{(\eps)}$. Now, assume that $\diam\lc U^X\rc >2\eps$. Then, $\rho_\eps(B^X)=\diam(B^X)-2\eps=\diam\lc U^X\rc -2\eps$. By assumption we have that for each $x\in U^X$
\[[x]^X_{\mathfrak{o}\lc \rho_\eps\lc B^X\rc\rc}=\{x'\in X:\,u_X(x,x')<\diam\lc U^X\rc -2\eps \}\subseteq U^X. \]
This implies that 
\[U^X=\bigcup_{x\in U^X}[x]^X_{\mathfrak{o}\lc \rho_\eps\lc B^X\rc\rc}.\]
So $U^X$ is the union of some elements in $B^X_{\mathfrak{o}\lc \rho_\eps\lc B^X\rc\rc}$ and thus $U^X\in B^X_{(\eps)}\subseteq  V_X^{(\eps)}$.
\end{proof}

\begin{proof}[Proof of Remark \ref{rmk:rho-eps-delta-eps}]
Let $U^X\in  V_X^{(\eps)}$ and let $B^Y\in V_Y $ be such that $|\diam\lc U^X\rc -\diam\lc B^Y\rc |\leq \eps$ and $\diam\lc B^Y\rc >\eps$. Let $U^X_{\mathfrak{o}\lc \delta_\eps\lc B^Y\rc\rc}=\left\{U^X_1,\ldots,U^X_{N_{U_X}}\right\}$. For any subset $I\subseteq[N_{U_X}]$ (which can be a singleton), we will prove next that the union $U^X_I\coloneqq\bigcup_{i\in I}U^X_i$ belongs to $ V_X^{(\eps)}$. For any $x\in X$, suppose that there is $x'\in U^X_i\subseteq U^X_I$ such that $u_X(x,x')< \diam(U^X_I)-2\eps\leq \diam\lc U^X\rc -2\eps$. Then, $x\in U^X$ since $U^X\in  V_X^{(\eps)}$ (cf. Lemma \ref{lm:char-LX-eps}). Now, $u_X(x,x')<\diam(U^X_I)-2\eps\leq\diam\lc U^X\rc -2\eps\leq\delta_\eps\lc B^Y\rc $. So $x$ and $x'$ belong to the same block in $U^X_{\mathfrak{o}\lc \delta_\eps\lc B^Y\rc\rc}$ and thus $x\in U^X_i$. Therefore, $x\in U^X_I$ and thus by Lemma \ref{lm:char-LX-eps} we have that $U^X_I\in  V_X^{(\eps)}$. 
\end{proof}

\subsubsection{Proof of Theorem \ref{thm:corr-dp-algo} (correctness of Algorithm $\mathbf{FindCorrDP}$ (Algorithm \ref{algo-dGH-dyn}))}\label{sec:corr-dp}

\begin{proof}
We prove a more general result, namely that for any $\lc U^X,B^Y\rc \in V_X^{(\eps)}\times  V_Y $, $\mathrm{DYN}\lc U^X,B^Y\rc =1$ if and only if there exists an $\eps$-correspondence between $U^X$ and $B^Y$. 

If $\lc U^X,B^Y\rc $ belongs to one of the base cases, the statement is obviously true. For non-base cases, we prove the claim by induction on $\diam\lc B^Y\rc \in\mathrm{spec}(Y)$. For this, we exploit the fact that the spectrum $\spec(Y)=\{0=t_0<\cdots<t_{M}=\diam(Y)\}$ is a finite set. When $\diam\lc B^Y\rc =t_0=0$, for any $U^X$, $\lc U^X,B^Y\rc $ belongs to one of the base cases, so the statement holds true trivially. Let $1< i_0\leq M$ and suppose that the claim holds true for all $t_i$ when $i<i_0$ and all $\mathrm{DYN}\lc U^X,B^Y\rc $ are known whenever $\diam\lc B^Y\rc <t_{i_0}$. Then, the induction step follows directly from Theorem \ref{thm:ums-dgh} and Proposition \ref{prop:smallcorr}. {We elaborate this via the two cases described in page \pageref{two cases} as follows:}
\begin{enumerate}
    \item If $\diam\lc B^Y\rc >\eps$, Algorithm $\mathbf{FindCorrDP}$ partitions $U^X$ and $B^Y$ to obtain $U^X_{\mathfrak{o}\lc \delta_\eps\lc B^Y\rc\rc}=\{U^X_i\}_{i\in [N_{U_X}]}$ and $B^Y_{\mathfrak{o}\lc{\delta_0\lc B^Y\rc} \rc}=\{B^Y_j\}_{j\in [N_{B_Y}]}$, respectively. It is obvious that $B^Y_j\in V_Y $ for each $j\in [N_{B_Y}]$, and by Remark \ref{rmk:rho-eps-delta-eps} we know that $U^X_i\in V_X^{(\eps)}$ for each $i\in [N_{U_X}]$. Since $\diam(B_j^Y)<\diam\lc B^Y\rc =t_{i_0}$ for each $j\in [N_{B_Y}]$, by the induction assumption, the value $\mathrm{DYN}\lc U^X_{\Psi^{-1}(j)},B^Y_j\rc $ has already been computed for any surjection $\Psi:[N_{U_X}]\rightarrow [N_{B_Y}]$ so we already know whether or not there exists any $\eps$-correspondence between $U^X_{\Psi^{-1}(j)}$ and $B^Y_j$. $\mathrm{DYN}\lc U^X,B^Y\rc $ is then determined via Theorem \ref{thm:ums-dgh}: this value is 1 if there exists an $\eps$-correspondence between $U^X$ and $B^Y$, and is 0 otherwise.
    \item If $\diam\lc B^Y\rc \leq \eps$, Algorithm $\mathbf{FindCorrDP}$ assigns the value $\mathbf{FindCorrSmall}(U^X,B^Y,\eps)$ to $\mathrm{DYN}\lc U^X,B^Y\rc$. Then, due to Proposition \ref{prop:smallcorr}, $\mathrm{DYN}\lc U^X,B^Y\rc =1$ if and only if there exists an $\eps$-correspondence between $U^X$ and $B^Y$.
\end{enumerate}

Since we know that $X$ and $Y$ are at the end of the arrays $\mathrm{LX}^{(\eps)}$ and $\mathrm{LY}$, respectively, then $\mathrm{DYN}(\mathrm{END},\mathrm{END})=1$ if and only if there exists an $\eps$-correspondence between $X$ and $Y$.
\end{proof}
\subsubsection{Proof of Remark \ref{rmk:sgc-doubling}}\label{sec:rmk-sgc-proof}
\begin{proof}
Let $X$ be a finite ultrametric space such that $X\in\mathcal{U}_2(\eps,\gamma)$ for some $\eps\geq 0$ and $\gamma\geq 1$.
For any $x\in X$ and any $r>\diam(X)$, we have that $ {B}_r(x)=X= {B}_{\diam(X)}(x)$. So if $ {B}_{\diam(X)}(x)$ can be covered by $K$ many balls with radius $\frac{\diam(X)}{2}$, then it is obvious that $ {B}_r(x)$ can also be covered by $K$ balls with radius $\frac{r}{2}>\frac{\diam(X)}{2}$. Therefore, to determine the doubling constant of $X$, we only need to consider a radius $r$ within the range $(0,\diam(X)]$. Let $k=\lfloor\frac{r}{4\eps}\rfloor+1$, then $k$ is the unique integer such that $r-2\eps\cdot k<\frac{r}{2}\leq r-2\eps\cdot(k-1)$. We assume that $r-2\eps\cdot k\geq 0$ (the case when $r-2\eps\cdot k< 0$ can be proved similarly and we omit the details). Then, by the SGC we have that 
\begin{align*}
    \#\left\{[x']_{\mathfrak{c}\lc\frac{r}{2}\rc}:\,x'\in [x]_{\mathfrak{c}\lc{r}\rc}\right\}&\leq \#\left\{[x']_{\mathfrak{c}\lc{r-2\eps\cdot k}\rc}:\,x'\in [x]_{\mathfrak{c}\lc{r}\rc}\right\}\leq \#\left\{[x']_{\mathfrak{o}\lc{r-2\eps\cdot k}\rc}:\,x'\in [x]_{\mathfrak{c}\lc{r}\rc}\right\}\\
    &\leq  \gamma\cdot\#\left\{[x']_{\mathfrak{c}\lc{r-2\eps\cdot (k-1)}\rc}:\,x'\in [x]_{\mathfrak{c}\lc{r}\rc}\right\}\leq\gamma\cdot\#\left\{[x']_{\mathfrak{o}\lc{r-2\eps\cdot (k-1)}\rc}:\,x'\in [x]_{\mathfrak{c}\lc{r}\rc}\right\}\\
    &\leq \cdots\leq \gamma^{k-1}\cdot\#\left\{[x']_{\mathfrak{c}\lc{r-2\eps}\rc}:\,x'\in [x]_{\mathfrak{c}\lc{r}\rc}\right\}\leq \gamma^k.
\end{align*}
Since $ {B}_r(x)=[x]_{\mathfrak{c}\lc r\rc}$ and $ {B}_\frac{r}{2}(x')=[x']_{\mathfrak{c}\lc\frac{r}{2}\rc}$, we have that $ {B}_r(x)$ can be covered by a union of at most $\gamma^k$ balls with radius $\frac{r}{2}$: $[x]_{\mathfrak{c}\lc r\rc}=\cup_{x'\in[x]_{\mathfrak{c}\lc r\rc}}[x']_{\mathfrak{c}\lc\frac{r}{2}\rc}$. Since $k\leq \lfloor\frac{\diam(X)}{4\eps}\rfloor+1$ for each $r\in(0,\diam(X)]$, we have that $X$ is $\gamma^{\lfloor\frac{\diam(X)}{4\eps}\rfloor+1}$-doubling. 

Conversely, suppose that $X$ is $K$-doubling. Then, for any $x\in X$ and $t>0$, there exist $x_1,\ldots,x_n$ such that $n\leq K$ and $ {B}_t(x)\subseteq\cup_{i=1}^n  {B}_\frac{t}{2}(x_i)$. Without loss of generality, we assume that $ {B}_\frac{t}{2}(x_i)\cap  {B}_t(x)\neq \emptyset$ for each $i=1,\ldots,n$. Then, by Proposition \ref{prop:basic property ultrametric}, we have that $[x_i]_{\mathfrak{c}\lc\frac{t}{2}\rc}= {B}_\frac{t}{2}(x_i)\subseteq  {B}_t(x)=[x]_{\mathfrak{c}(t)}$. Therefore, $[x_i]_{\mathfrak{c}\lc\frac{t}{2}\rc}\subseteq [x_i]_{\mathfrak{o}\lc t\rc}\subseteq [x]_{\mathfrak{c}(t)}$. Since $[x]_{\mathfrak{c}(t)}\subseteq\cup_{i=1}^n[x_i]_{\mathfrak{c}\lc\frac{t}{2}\rc}\subseteq \cup_{i=1}^n[x_i]_{\mathfrak{o}\lc t\rc}$, we have that
\[\#\left\{{[x']_{\mathfrak{o}(t)}}:\,x'\in [x]_{\mathfrak{c}\lc t\rc}\right\}\leq n\leq K. \]
This implies that $X\in\mathcal{U}_2(0,K)$.
\end{proof}

\subsubsection{Proof of Theorem \ref{thm:com-dyn-alg} (Time complexity of Algorithm $\mathbf{FindCorrDP}$ (Algorithm \ref{algo-dGH-dyn}))}\label{sec:com-dp-proof}

\begin{lemma}[Time complexity of Algorithm $\mathbf{FindCorrSmall}$ (Algorithm \ref{algo-dGH-small})]\label{lm:smallcoman}
Algorithm \ref{algo-dGH-small} runs in time $O(n^2n^n)$ where $n\coloneqq\max(\#X,\#Y)$. 
\end{lemma}
\begin{proof}
$\mathbf{ClosedQuotient}(X,\eps)$ runs in time $O(n)$ (cf. Appendix \ref{sec:data-structure}). There are at most $n^n$ injective maps and for each injective map $\Phi:X\ct{\eps}\rightarrow Y$, we need $O(n^2)$ time to compute $\dis(\Phi)$. Therefore, Algorithm \ref{algo-dGH-small} runs in time bounded by $O(n^2n^n)$.
\end{proof}

\begin{lemma}[Inheritance of the SGC]\label{lm:inheritance of SGC}
If $X$ satisfies the second $(\eps,\gamma)$-growth condition, then so does each $U^X\in V_X^{(\eps)}$ and in particular, so does each ball $B^X\in V_X$.
\end{lemma}
\begin{proof}
Let $U^X\in V_X^{(\eps)}$. Fix a $t\geq 2\eps$ and $x\in U^X$. Note that for two distinct points $x',x''\in [x]_{\mathfrak{c}(t)}^{U^X}$, $[x']_{\mathfrak{o}(t-2\eps)}^{U^X}\neq [x'']_{\mathfrak{o}(t-2\eps)}^{U^X}$ if and only if $u_{U^X}(x',x'')=u_X(x',x'')\geq t-2\eps$. This is then also equivalent to the condition $[x']_{\mathfrak{o}(t-2\eps)}^{X}\neq [x'']_{\mathfrak{o}(t-2\eps)}^{X}$. 
Then, we have that
\[\#\left\{[x']_{\mathfrak{o}(t-2\eps)}^{U^X}:\,x'\in[x]_{\mathfrak{c}(t)}^{U^X}\right\} \leq \#\left\{[x']_{\mathfrak{o}(t-2\eps)}^{X}:\,x'\in[x]_{\mathfrak{c}(t)}^{X}\right\}\leq \gamma.\]
This concludes the proof that $U^X\in\mathcal{U}_2(\eps,\gamma)$.
\end{proof}

\begin{lemma}\label{lm:counting-lxly}
Let $X$ and $Y$ be two finite ultrametric spaces. Then, $\# V_X =O(\#X)$ and $\# V_Y =O(\#Y)$. If $X\in\mathcal{U}_2(\eps,\gamma)$, then $\# V_X^{(\eps)}=O(\#X\cdot 2^\gamma)$.
\end{lemma}

\begin{proof}
By Remark \ref{rmk:tree number}, we have that $\# V_X =O(\#X)$ and $\# V_Y =O(\#Y)$.

$ V_X^{(\eps)}$ is defined in Section \ref{sec:dp-main-text} as: $ V_X^{(\eps)}\coloneqq\bigcup_{B^X\in  V_X }B^X_{(\eps)}$. For notational simplicity, we let $\rho_\eps\coloneqq\rho_\eps\lc B^X\rc$. Each $B^X_{(\eps)}$ is a collection of unions of elements in $B^X_{\mathfrak{o}\lc \rho_\eps\rc}$ and thus $B^X_{(\eps)}$ is a subset of the power set $2^{B^X_{\mathfrak{o}\lc \rho_\eps\rc}}$. Since $B^X$ is a closed ball in $X$, $B^X=[x]_{\mathfrak{c}\lc\rho\rc}^X$ for some $x\in X$ and $\rho\coloneqq\diam\lc B^X\rc $. By the second $(\eps,\gamma)$-growth condition, we have that 
\[\#B^X_{\mathfrak{o}\lc \rho_\eps\rc}=\#\left\{{[x']^X_{\mathfrak{o}\lc \rho_\eps\rc}}:\,x'\in [x]_{\mathfrak{c}\lc\rho\rc}^X\right\}\leq\gamma\] 
and thus $\#2^{B^X_{\mathfrak{o}\lc \rho_\eps\rc}}\leq 2^\gamma$. Then, by $\# V_X =O(\#X)$, we have $\# V_X^{(\eps)}=O(\#X\cdot 2^\gamma).$
\end{proof}

\begin{proof}[Proof of Theorem \ref{thm:com-dyn-alg}]

{\textbf{Preprocessing.} In order to implement the union operation (cf. Appendix \ref{sec:union of susbets}) efficiently, we will reconstruct the distance matrices $u_X$ and $u_Y$ from the TDSs $T_X$ and $T_Y$, respectively. This process takes time at most $O(n^2)$ (cf. Remark \ref{rmk:distance mtx}). $\mathrm{LX}^{(\eps)}$ and $\mathrm{LY}$ can be constructed in time $O\lc n^2\log(n)2^\gamma\gamma^2\rc$ (cf. Appendix \ref{sec:implementation detail}). We create an all-zero matrix $\mathrm{DYN}$ of size $\#\mathrm{LX}^{(\eps)}\times\#\mathrm{LY}$ in time $O(n^22^\gamma)$.

\textbf{Main part of the algorithm.} For each $B^Y\in  \mathrm{LY} $, we have the following cases for $U^X\in \mathrm{LX}^{(\eps)}$:

\begin{enumerate}
    \item $|\diam\lc U^X\rc -\diam\lc B^Y\rc |>\eps$ or $\max(\diam\lc U^X\rc ,\diam\lc B^Y\rc)\leq\eps$: It takes constant time to assign either 0 or 1 to $\mathrm{DYN}(U^X,B^Y)$ based on this.
    
    \item $\diam\lc B^Y\rc \leq \eps<\diam\lc U^X\rc$: In this case, both $\diam\lc U^X\rc $ and $\diam\lc B^Y\rc $ are bounded above by $2\eps$ (since the pair $\lc U^X,B^Y\rc$ does not satisfy the condition in the first case). Then, by the SGC and Lemma \ref{lm:inheritance of SGC}, it is easy to check that $\#U^X,\#B^Y\leq \gamma$. Thus, by Lemma \ref{lm:smallcoman}, Algorithm $\mathbf{FindCorrSmall}$ with input $(U^X,B^Y,\eps)$ runs in time $O(\gamma^2\gamma^\gamma)$.
    
    \item $\diam\lc B^Y\rc >\eps$: In this case, by the SGC and Lemma \ref{lm:inheritance of SGC}, it takes at most $O(\gamma)$ time to partition both $U^X$ and $B^Y$ via Algorithm $\mathbf{OpenPartition}$ (Algorithm \ref{algo-o-part}) into at most $\gamma$ blocks, respectively. We then have at most $\gamma^\gamma$ surjections to consider. Given any such surjection $\Psi$, for each $j\in [N_Y]$ let $k_j\coloneqq\#\{U^X_i\}_{i\in \Psi^{-1}(j)}$. Then, it takes time at most $O(k_j^2+\gamma k_j)=O(\gamma k_j)$ to construct the union $U^X_{\Psi^{-1}(j)}$ via the refined union operation discussed in Appendix \ref{sec:union of susbets} (see also Appendix \ref{sec:refined union}). It takes time at most $O(\gamma\log(\gamma))$ to find the index of $U^X_{\Psi^{-1}(j)}$ in $\mathrm{LX}^{(\eps)}$ and constant time to find the index of $B^Y_j$ in $\mathrm{LY}$ (cf. Appendix \ref{sec:preprocessing}). Therefore, accessing the value $\mathrm{DYN}\lc U^X_{\Psi^{-1}(j)},B^Y_j\rc$ has cost at most $O(\gamma\log(\gamma))$. Then, the time complexity for accessing values in $\mathrm{DYN}$ for the surjection $\Psi$ is at most 
    $$\sum_{j}O(\gamma k_j\log(\gamma))=O(\gamma^2\log(\gamma)),$$
    where we use the fact that $\sum_jk_j=O(\gamma)$. Therefore, the total time complexity of this case is bounded by $O(\gamma^{\gamma+2}\log(\gamma))$
\end{enumerate}

Therefore for a single $B^Y\in\mathrm{LY}$, completing all the operations taking place between line 4 and line 22 of Algorithm \ref{algo-dGH-dyn} requires at most time $O(n2^\gamma)\times  O(\gamma^{\gamma+2}\log(\gamma))=O\lc n2^\gamma\gamma^{\gamma+2}\log(\gamma)\rc$. Then, completing the for-loop in line 3 of Algorithm \ref{algo-dGH-dyn} requires requires at most time
\[O(n)\times O\lc n2^\gamma\gamma^{\gamma+2}\log(\gamma)\rc=O\lc n^22^\gamma\gamma^{\gamma+2}\log(\gamma)\rc\]
operations to fill out the matrix $\mathrm{DYN}$.

\textbf{Total time complexity.} By combining the time complexity of the preprocessing part, we have that the total time complexity of Algorithm \ref{algo-dGH-dyn} is bounded by
\[O\lc n^2\log(n) 2^\gamma\gamma^2\rc+O(n^22^\gamma)+O\lc n^22^\gamma\gamma^{\gamma+2}\log(\gamma)\rc=O\lc n^2\log(n)2^\gamma\gamma^{\gamma+2}\rc.\]
This concludes the proof.}
\end{proof}

\subsubsection{Proof of Remark \ref{rmk:improved complexity}}\label{sec:proof of rmk improved complexity}
{Let $\mathcal{E}(X,Y)=\{\eps_0<\eps_1<\ldots<\eps_M\}$. Observe that for each $\eps_i$, since the number of all vertices in $T_X$ or $T_Y$ are bounded above by $2n$, $\gamma_{\eps_i}(X,Y)\coloneqq\max(\gamma_{\eps_i}(X),\gamma_{\eps_i}(Y))$ takes values in $\{1,2,\ldots,2n\}$. It is obvious that for each $k\in\{1,2,\ldots,2n\}$, the set of $\eps_i$s such that $\gamma_{\eps_i}(X,Y)=k$ is an interval, i.e., a consecutive subsequence of $\mathcal{E}(X,Y)$, denoted by $[\eps_{\ell_k},\ldots,\eps_{r_k}]$. Note that $\eps\coloneqq2\dgh(X,Y)\in\mathcal{E}(X,Y)$. Then, there exists $i\in\{0,1,\ldots,M\}$ such that $\eps_i=\eps$. It is obvious that $i$ is the smallest index such that there exists a $\eps_i$-correspondence between $X$ and $Y$. In order to find the index $i$ (and thus to compute $\dgh(X,Y)$), we apply the same procedure as in the proof of \cite[Theorem 5]{touli2018fpt} to search in $\{1,\ldots,2n\}$ for the smallest number $k^*$ such that $\{\eps_{\ell_{k^*}},\ldots,\eps_{r_{k^*}}\}\cap [\eps,\infty)$ is non-empty. This $k^*$ satisfies the condition $k^*=\gamma_\eps(X,Y)$ and $\ell_{k^*}\leq i\leq r_{k^*}$. Then, we apply binary search to find the index $i$. Via an argument similar to the one stated in \cite[Theorem 5]{touli2018fpt}, this whole process for finding $i$ can be completed in time $O\lc n^2\log^3(n)2^{2k^*}(2k^*)^{2k^*+2}\rc$. Since $X,Y\in\mathcal{U}_2(\eps,\gamma)$, we have that $\gamma\geq \gamma_\eps(X,Y)=k^*$. Therefore, we conclude that the exact value $\dgh(X,Y)$ can be computed in time complexity at most $O\lc n^2\log^3(n)2^{2\gamma}(2\gamma)^{2\gamma+2}\rc$.}

\subsubsection{Proof of Lemma \ref{lm:ultrametricty sgc}}\label{sec:proof of approximation}

\begin{proof}
Pick any positive real number $t\geq 2\eps$. By Proposition \ref{prop:ultrametricity double} we have that $\|d_X-u_X^*\|_\infty=\delta$. Then, $B^{u_X^*}_t(x)\subseteq B^{d_X}_{t+\delta}(x)$ for any $x\in X$ and $t\geq 0$. Here $B^d_t(x)\coloneqq\{x'\in X:\,d(x,x')\leq t\}$ represents the closed ball centered at $x$ with radius $t$ with respect to the metric $d$. For later use, we use $B^d_{\mathfrak{o}(t)}(x)\coloneqq\{x'\in X:\,d(x,x')< t\}$ to denote an open ball. Since $u_X^*\leq d_X$, we have that $B^{d_X}_t(x)\subseteq B^{u_X^*}_t(x)$.

We first assume that $\frac{t+\delta}{2}<t-2\eps$. Since $(X,d_X)$ is $K$-doubling, there exist $x_1,\ldots, x_K\in X$ such that $B_{t+\delta}^{d_X}(x)\subseteq\bigcup_{i=1}^KB_\frac{t+\delta}{2}^{d_X}(x_i).$
For each $x_i$, we have that
\[B_\frac{t+\delta}{2}^{d_X}(x_i)\subseteq B^{u_X^*}_\frac{t+\delta}{2}(x_i)\subseteq B^{u_X^*}_{\mathfrak{o}(t-2\eps)}(x_i),\]
where the last inclusion follows from the assumption that $\frac{t+\delta}{2}<t-2\eps$. Then, we have that
\[B^{u_X^*}_t(x)\subseteq B_{t+\delta}^{d_X}(x)\subseteq\bigcup_{i=1}^KB_\frac{t+\delta}{2}^{d_X}(x_i) \subseteq \bigcup_{i=1}^KB^{u_X^*}_{\mathfrak{o}(t-2\eps)}(x_i).\]
Using the notation for open and closed equivalence relations for ultrametric spaces, we conclude that $[x]_{\mathfrak{c}(t)}\subseteq\bigcup_{i=1}^K[x_i]_{\mathfrak{o}(t-2\eps)}$ which implies that
\[\#\left\{{[x']_{\mathfrak{o}\lc t-2\eps\rc}}:\,x'\in [x]_{\mathfrak{c}\lc t\rc}\right\}\leq K.\]

Now, we assume that $\frac{t+\delta}{2}\geq t-2\eps$ {(equivalently $t\leq \delta+4\eps$)}. First note that $s\coloneqq\mathrm{sep}(X,d_X)=\mathrm{sep}(X,u_X)$. Then, if $t+\delta< s$, we have that $B_t^{u_X^*}(x)=B_t^{d_X}(x)=\{x\}$. Hence,
\[\#\left\{{[x']_{\mathfrak{o}\lc t-2\eps\rc}}:\,x'\in [x]_{\mathfrak{c}\lc t\rc}\right\}\leq 1.\]
Otherwise, we assume that  $t+\delta\geq s$. Then, we let $k\in\mathbb{N}$ be such that $\frac{t+\delta}{2^k}<s\leq \frac{t+\delta}{2^{k-1}}$. Equivalently, we have
\[k-1\leq \log_2\lc \frac{t+\delta}{s}\rc<k.\]
By the $K$-doubling property of $(X,d_X)$, we have that $B_{t+\delta}^{d_X}(x)$ can be covered by at most $K^k$ many balls with radius $\frac{t+\delta}{2^k}$. Since $\frac{t+\delta}{2^k}<s$, these balls are singletons and thus $\#B_{t+\delta}^{d_X}(x)\leq K^k$. Therefore,
\[\#B^{u_X^*}_t(x)\leq \#B_{t+\delta}^{d_X}(x)\leq K^k\leq K^{\log_2\lc \frac{t+\delta}{s}\rc+1}\leq  K^{\log_2\lc \frac{2\delta+4\eps}{s}\rc+1},\]
where we use the assumption $t\leq \delta+4\eps$ in the last inequality. Consequently, 
\[\#\left\{{[x']_{\mathfrak{o}\lc t-2\eps\rc}}:\,x'\in [x]_{\mathfrak{c}\lc t\rc}\right\}\leq K^{\log_2\lc \frac{2\delta+4\eps}{s}\rc+1}.\]

In conclusion, for any $x\in X$ and any $r\geq 2\eps$ we have that 
\[\#\left\{{[x']_{\mathfrak{o}\lc t-2\eps\rc}}:\,x'\in [x]_{\mathfrak{c}\lc t\rc}\right\}\leq \max\lc K,K^{\log_2\lc \frac{2\delta+4\eps}{s}\rc+1}\rc,\]
and thus $(X,u_X^*)\in\mathcal{U}_2\lc\eps,\max\lc K,K^{\log_2\lc\frac{2\delta+4\eps}{s}\rc+1}\rc\rc$. 
\end{proof}

\end{document}